\documentclass[11pt]{article}
\usepackage[a4paper]{geometry}
\usepackage{aeguill}                 
\usepackage{graphicx}
\usepackage{amsmath,amsfonts,amssymb,amsthm}
\usepackage{pstricks} 
\usepackage{epstopdf}
\usepackage{srcltx}
\usepackage[utf8]{inputenc} 
\usepackage{hyperref} 
\usepackage[english]{babel}

\title{Measure equivalence rigidity among the Higman groups}
\author{Camille Horbez and Jingyin Huang\thanks{Both authors thank the Institut Henri Poincaré (UAR 839 CNRS-Sorbonne Université) for its hospitality and support (through LabEx CARMIN, ANR-10-LABX-59-01) during the trimester program \emph{Groups acting on fractals, Hyperbolicity and Self-similarity}. The first named author acknowledges support from the Agence Nationale de la Recherche under Grants ANR-16-CE40-0006 DAGGER and ANR-22-ERCS-0011-01 Artin-Out-ME-OA.}}

\usepackage[all]{xy}
\usepackage{bbm}
\begin{document}
	\maketitle
	\newtheorem{de}{Definition} [section]
	\newtheorem{theo}[de]{Theorem} 
	\newtheorem{prop}[de]{Proposition}
	\newtheorem{lemma}[de]{Lemma}
	\newtheorem{cor}[de]{Corollary}
	\newtheorem{propd}[de]{Proposition-Definition}
	\newtheorem{conj}[de]{Conjecture}
	\newtheorem{claim}{Claim}
	\newtheorem*{claim2}{Claim}
	
	\newtheorem{theointro}{Theorem}
	\newtheorem*{defintro}{Definition}
	\newtheorem{corintro}[theointro]{Corollary}
	\newtheorem{questionintro}[theointro]{Question}
	\newtheorem{propintro}[theointro]{Proposition}

	\theoremstyle{remark}
	\newtheorem{rk}[de]{Remark}
	\newtheorem{ex}[de]{Example}
	\newtheorem{question}[de]{Question}
	\newtheorem*{rkintro}{Remark}

	\normalsize
	
	\newcommand{\cala}{\mathcal{A}}
	\newcommand{\calf}{\mathcal{F}}
	\newcommand{\calg}{\mathcal{G}}
	\newcommand{\calh}{\mathcal{H}}
	\newcommand{\cali}{\mathcal{I}}
	\newcommand{\calp}{\mathcal{P}}
	\newcommand{\call}{\mathcal{L}}
	\newcommand{\calk}{\mathcal{K}}
	\newcommand{\calq}{\mathcal{Q}}
	\newcommand{\calr}{\mathcal{R}}
	\newcommand{\bx}{\bar{x}}
	\newcommand{\be}{\bar{e}}
	
	\newcommand{\dunion}{\sqcup}
	\newcommand{\Stab}{\mathrm{Stab}}
	\newcommand{\Prob}{\mathrm{Prob}}
	\newcommand{\Aut}{\mathrm{Aut}}
	\newcommand{\actson}{\curvearrowright}
	\newcommand{\BS}{\mathrm{BS}}
	\newcommand{\lk}{\mathrm{lk}}
	\newcommand{\fix}{\mathrm{Fix}}
	\newcommand{\CAT}{\operatorname{CAT}}
	\newcommand{\Inc}{\mathrm{Inc}}
	\newcommand{\bdd}{\mathrm{bdd}}
	\newcommand{\cals}{\mathcal{S}}
	\newcommand{\Hig}{\mathrm{Hig}}
	\newcommand{\Out}{\mathrm{Out}}
	\newcommand{\Comm}{\mathrm{Comm}}
	\newcommand{\Cay}{\mathrm{Cay}}
	
	\makeatletter
	\edef\@tempa#1#2{\def#1{\mathaccent\string"\noexpand\accentclass@#2 }}
	\@tempa\rond{017}
	\makeatother

	\begin{abstract}
		We prove that all (generalized) Higman groups on at least $5$ generators are superrigid for measure equivalence. More precisely, let $k\ge 5$, and let $H$ be a group with generators $a_1,\dots,a_k$, and Baumslag-Solitar relations given by $a_ia_{i+1}^{m_i}a_i^{-1}=a_i^{n_i}$, with $i$ varying in $\mathbb{Z}/k\mathbb{Z}$ and nonzero integers $|m_i|\neq |n_i|$ for each $i$. We prove that every countable group which is measure equivalent to $H$, is in fact virtually isomorphic to $H$. A key ingredient in the proof is a general statement providing measured group theoretic invariants for groups acting acylindrically on $\mathrm{CAT}(-1)$ polyhedral complexes with control on vertex and edge stabilizers.	
		
		Among consequences of our work, we obtain rigidity theorems for generalized Higman groups with respect to lattice embeddings and automorphisms of their Cayley graphs. We also derive an orbit equivalence and $W^*$-superrigidity theorem for all free, ergodic, probability measure-preserving actions of generalized Higman groups.
	\end{abstract}

	\section{Introduction}
	
	\subsection{General context and statement of the main theorem}
	
	In 1951, Higman introduced the groups \[\Hig_k=\langle a_1,\dots,a_k|\{a_ia_{i+1}a_i^{-1}=a_{i+1}^{2}\}_{i\in\mathbb{Z}/k\mathbb{Z}}\rangle\] with $k\ge 4$, as the first examples of infinite finitely presented groups without any nontrivial finite quotient \cite{Hig}. These groups have then received a lot of attention, for example they played a role in the work of Platonov and Tavgen' \cite{PT} giving a counterexample to a question of Grothendieck regarding profinite completions. They are also sometimes considered as potential candidates of non-sofic or non-hyperlinear groups; relatedly, see e.g.\ Thom's work on their metric approximations \cite{Tho}, see also \cite{HJ}. They are known to be SQ-universal (i.e.\ every countable group embeds in some quotient of $\Hig_k$) by work of Schupp \cite{Sch}, and even acylindrically hyperbolic by work of Minasyan and Osin \cite{MO}. They are fundamental examples of ``polygons of groups'' \cite{stallings1991non,BH}, which generalize the classical notion of graphs of groups. Their actions on $\mathrm{CAT}(0)$ cube complexes have been studied by Martin \cite{Mar,Mar2}, who obtained rigidity results regarding their automorphism groups that have been a source of inspiration for the present work. We also refer the reader to \cite{Mon,KKR} for closely related examples on the aforementioned themes, based on variations and generalizations of Higman's construction.
	
	In the present paper, we study Higman groups from the viewpoint of measured group theory, and prove the measure equivalence rigidity of $\Hig_k$ whenever $k\ge 5$. We will in fact consider the following larger family of groups, also allowing two consecutive generators to generate a non-amenable Baumslag-Solitar group. Let $k\ge 4$, and let $\sigma=((m_1,n_1),\dots,(m_k,n_k))$ be a $k$-tuple of pairs of non-zero integers with $|m_i|\neq |n_i|$ for every $i\in\{1,\dots,k\}$; the \emph{generalized Higman group} $\Hig_\sigma$ is the group defined by the following presentation: \[\Hig_\sigma=\langle a_1,\dots,a_k|\{a_ia_{i+1}^{m_i}a_i^{-1}=a_{i+1}^{n_i}\}_{i\in\mathbb{Z}/k\mathbb{Z}}\rangle.\]    
	
The notion of \emph{measure equivalence}, due to Gromov \cite{Gro}, is a measurable analogue of quasi-isometry. It has a geometric flavour in that it generalizes the concept of lattices in locally compact groups, and it is also strongly related to the notion of \emph{orbit equivalence} coming from ergodic theory of group actions. These aspects will be exploited in the present paper. Recall that two countable groups $\Gamma_1,\Gamma_2$ are \emph{measure equivalent} if there exists a standard measure space $\Sigma$, equipped with an action of the product $\Gamma_1\times\Gamma_2$ by measure-preserving Borel automorphisms, such that for every $i\in\{1,2\}$, the action of the factor $\Gamma_i$ on $\Sigma$ is free and has a fundamental domain of finite measure. Such a space $\Sigma$ is called a \emph{measure equivalence coupling} between $\Gamma_1$ and $\Gamma_2$. An important example is that any two lattices in the same locally compact second countable group $G$ are always measure equivalent, through their left/right action by multiplication on $G$, preserving the Haar measure. Also any two countable groups $\Gamma_1,\Gamma_2$ which are \emph{virtually isomorphic} (i.e.\ there exist finite-index subgroups $\Gamma_i^0\subseteq\Gamma_i$ and finite normal subgroups $F_i\unlhd\Gamma_i^0$ such that $\Gamma_1^0/F_1$ and $\Gamma_2^0/F_2$ are isomorphic) are automatically measure equivalent. Our main theorem is the following. 
	
	\begin{theo}\label{theointro:main}
		Let $k\ge 5$, and let $\sigma=((m_1,n_1),\dots,(m_k,n_k))$ be a $k$-tuple of pairs of non-zero integers, with $|m_i|\neq |n_i|$ for every $i\in\{1,\dots,k\}$.
		
		Then the generalized Higman group $\Hig_\sigma$ is superrigid for measure equivalence: if a countable group $G$ is measure equivalent to $\Hig_\sigma$, then $G$ is virtually isomorphic to $\Hig_\sigma$.
	\end{theo}
	
As explained later, this theorem follows from a rigidity statement about self measure equivalence couplings of $\Hig_\sigma$ (Theorem~\ref{theointro:main-2}). The analogous (open) question for quasi-isometry would be to compute the group of self quasi-isometries of $\Hig_\sigma$.  
	
	The above rigidity theorem is far from the norm: a celebrated theorem of Ornstein and Weiss \cite{OW}, generalizing earlier work of Dye \cite{Dye1,Dye2}, asserts that all countably infinite amenable groups are measure equivalent. The first measure equivalence rigidity theorem was proved by Furman \cite{Fur}, who showed that any countable group which is measure equivalent to a lattice in a higher-rank simple Lie group with trivial center, is virtually isomorphic to a lattice in the same ambient Lie group. Kida proved that, except for the obvious low-complexity cases, mapping class groups of finite-type, connected, orientable surfaces are superrigid for measure equivalence, in the same sense as in the above theorem \cite{Kid-me}. His work paved the way towards other measure equivalence superrigidity theorems, e.g.\ for the Torelli group (Chifan-Kida \cite{CK}), for the mapping class group of a $3$-dimensional handlebody of genus at least $3$ (Hensel-Horbez \cite{HeH}), and for $\Out(F_N)$, the outer automorphism group of a free group of rank $N\ge 3$ (Guirardel-Horbez \cite{GH}). 
	
	In earlier works, we studied the rigidity/flexibility of several classes of Artin groups from the viewpoint of measured group theory \cite{HH1,HH2,HH3}. While some large-type Artin groups are superrigid for measure equivalence, this is no longer true of the right-angled ones. In fact, using the aforementioned theorem of Ornstein and Weiss and an argument of Gaboriau, we proved that any right-angled Artin group (which is a graph product of infinite cyclic groups), is measure equivalent to the graph product of any countably infinite amenable groups over the same defining graph \cite[Corollary~5]{HH2}. In particular, if one takes $|m_i|=|n_i|=1$ in Theorem~\ref{theointro:main}, one obtains a right-angled Artin group which is far from being rigid for measure equivalence.
	It is quite remarkable that replacing the commutation relations used in right-angled Artin groups, by (possibly solvable) Baumslag-Solitar relations, suddenly yields the strongest possible form of rigidity in measure equivalence.
	
A unifying theme for our previous works on Artin groups, and the present work on Higman groups, is to identify measure equivalence invariants for the class of groups which admit nice actions on non-positively curved (or negatively curved) complexes. Because of the diversity of this class, it might be hopeless to expect a uniform treatment. For instance, several key ingredients in our proof for Artin groups break down completely for Higman groups. It is however worth mentioning that in the course of our proof, we will establish a theorem giving measure equivalence invariants in a general context of groups acting acylindrically on $\CAT(-1)$ complexes, see Theorem~\ref{theo:combination} below.

\paragraph{Several open questions.}
We suspect that Theorem~\ref{theointro:main} should also hold when $k=4$, though one might need a substantial new insight to treat this situation.  %for a more detailed discussion. 
Let us also mention that we do not know about the case where $|m_i|=|n_i|\ge 2$ in Theorem~\ref{theointro:main}.

In the light of quasi-isometric rigidity and flexibility theorems for Baumslag-Solitar groups \cite{FM1,FM2,Why}, we ask the same question for generalized Higman groups.

\begin{question}
	Let $k\ge 4$, and let $\sigma=((m_1,n_1),\dots,(m_k,n_k))$ be a $k$-tuple of pairs of non-zero integers, with $|m_i|\neq |n_i|$ for every $i\in\{1,\dots,k\}$. Let $G$ be a finitely generated group quasi-isometric to $\Hig_\sigma$. Is it true that $G$ is virtually isomorphic to $\Hig_\sigma$?
\end{question}

Theorem~\ref{theo:crigidity0} from the present work reduces this question to the following: if one can prove that any self-quasi-isometry $q:\Hig_\sigma\to \Hig_\sigma$ maps every left coset of a Baumslag-Solitar subgroup to another such coset, up to finite Hausdorff distance, then the answer to the above question is positive.

	\subsection{Some consequences}
	
	In this section, we record various consequences of our work. Most of them are general consequences of the rigidity of self measure equivalence couplings of $\Hig_\sigma$, but the application to $W^*$-superrigidity theorems requires an extra argument. 
	
	\paragraph*{Lattice embeddings and automorphisms of the Cayley graphs.} 
	Let $F_\sigma$ be the finite group consisting of all translations $\tau$ of $\mathbb{Z}/k\mathbb{Z}$ with the property that for every $i\in\mathbb{Z}/k\mathbb{Z}$, one either has $m_{\tau(i)}=m_i$ and $n_{\tau(i)}=n_i$, or $m_{\tau(i)}=-m_i$ and $n_{\tau(i)}=-n_i$. Then $F_\sigma$ naturally acts by automorphisms on $\Hig_\sigma$, and we let $\widehat{\Hig}_\sigma=\Hig_\sigma\rtimes F_\sigma$.

As a first application, let us mention that our work extends a theorem of Martin \cite[Theorem~C]{Mar} regarding the automorphism group of the Higman group, showing that the natural map from $\widehat{\Hig}_\sigma$ to $\Aut(\Hig_\sigma)$ (or even to the abstract commensurator $\Comm(\Hig_\sigma)$) is an isomorphism. Notice however that the abstract commensurator does not always give more information than the automorphism group: in the case of the classical Higman groups, $\Hig_k$ does not contain any proper finite-index subgroup.

We will now focus on applications for which the viewpoint of measure equivalence is used in a more essential way. Recall that given a group $G$, with a generating set $S$, the \emph{Cayley graph} $\Cay(G,S)$ is the simple graph having one vertex per element of $G$, where two vertices $g,h$ are joined by an edge if there exists $s\in S^{\pm 1}$ such that $h=gs$. 
	
	\begin{cor}[{see Corollary~\ref{cor:consequences}}]\label{corintro:lattice}
		Let $k\ge 5$, and let $\sigma=((m_1,n_1),\dots,(m_k,n_k))$ be a $k$-tuple of pairs of non-zero integers, with $|m_i|\neq |n_i|$ for every $i\in\{1,\dots,k\}$.
		\begin{enumerate}
			\item For every locally compact second countable group $G$ and every lattice embedding $\alpha:\Hig_\sigma\to G$, there exists a continuous homomorphism $\pi:G\to\widehat{\Hig}_\sigma$ with compact kernel such that $\pi\circ\alpha=\mathrm{id}$.
			\item For every finite generating set $S$ of $\Hig_\sigma$, the inclusion $\Hig_\sigma\hookrightarrow\widehat{\Hig}_\sigma$ extends to an injective homomorphism $\Aut(\Cay(\Hig_\sigma,S))\hookrightarrow \widehat{\Hig}_\sigma$. 
		\end{enumerate}
	\end{cor}
	
	Notice that it is possible to choose $\sigma$ so that $F_\sigma=\{\mathrm{id}\}$, in which case  $\Aut(\Cay(\Hig_\sigma,S))$  coincides with $\Hig_\sigma$ (for every finite generating set $S$). The application of measure equivalence rigidity to automorphisms of Cayley graphs has already been observed in other settings; see e.g.\ \cite[Section~6]{dlST} for a general reference and other examples.

	\paragraph{Orbit equivalence rigidity and $W^*$-rigidity.} Measure equivalence is intimately related to orbit equivalence of probability measure-preserving group actions \cite{Fur-oe}. Recall that two free, ergodic, measure-preserving actions $G\actson (X,\mu)$ and $H\actson (Y,\nu)$ of countable groups on standard probability spaces are \emph{orbit equivalent} if there exists a measure space isomorphism $f:X\to Y$ such that for $\mu$-almost every $x\in X$, one has $f(G\cdot x)=H\cdot f(x)$. The two actions are \emph{stably orbit equivalent} if there exist positive measure Borel subsets $U\subseteq X$ and $V\subseteq Y$, and a measure-scaling isomorphism $f:U\to V$ (in other words $f$ induces an isomorphism between the rescaled probability spaces $\frac{1}{\mu(U)}U$ and $\frac{1}{\nu(V)}V$) such that for almost every $x\in U$, one has $f(G\cdot x)\cap U= (H\cdot f(x))\cap V$. By a crucial observation of Furman \cite{Fur-oe} (see also \cite[Proposition~6.2]{Gab-l2}), two countable groups are measure equivalent if and only if they admit free, ergodic, measure-preserving actions on standard probability spaces which are stably orbit equivalent.
	
	The actions $G\actson X$ and $H\actson Y$ are \emph{conjugate} if there exist a group isomorphism $\alpha:G\to H$ and a measure space isomorphism $f:X\to Y$ such that for every $g\in G$ and a.e.\ $x\in X$, one has $f(gx)=\alpha(g)f(x)$. There is a weaker notion of \emph{virtual conjugation}, taking into account the possibility of passing to a finite-index subgroup or taking the quotient by a finite normal subgroup, for which we refer to \cite[Definition~1.3]{Kid-oe}.
	
	To every action $G\actson (X,\mu)$ as above, one also associates a von Neumann algebra $L(G\actson X)$ via the Murray-von Neumann cross-product construction \cite{MvN}. The actions $G\actson X$ and $H\actson Y$ are \emph{$W^*$-equivalent} if $L(G\actson X)$ and $L(H\actson Y)$ are isomorphic; this is a weaker notion than orbit equivalence. They are \emph{stably $W^*$-equivalent} if the associated von Neumann algebras have isomorphic amplifications.
	
	\begin{cor}[{see Corollaries~\ref{cor:oe} and~\ref{cor:w}}]\label{cor:oe-w}
		Let $k\ge 5$, and let $\sigma=((m_1,n_1),\dots,(m_k,n_k))$ be a $k$-tuple of pairs of non-zero integers, with $|m_i|\neq |n_i|$ for every $i\in\{1,\dots,k\}$. Let $\Hig_\sigma\actson X$ be a free, ergodic, probability measure-preserving action of $\Hig_\sigma$ on a standard probability space $X$. Let $\Gamma\actson Y$ be a free, ergodic, probability measure-preserving action of a countable group $\Gamma$ on a standard probability space $Y$.
		\begin{enumerate}
		\item If the actions $\Hig_\sigma\actson X$ and $\Gamma\actson Y$ are stably orbit equivalent, or merely stably $W^*$-equivalent, then they are virtually conjugate.
		\item If $\Hig_\sigma=\Hig_k$ belongs to the classical family of Higman groups, and if the actions $\Hig_\sigma\actson X$ and $\Gamma\actson Y$ are orbit equivalent, or merely $W^*$-equivalent, then they are conjugate.
		\end{enumerate}
	\end{cor}

Popa's deformation/rigidity theory has led to many examples of $W^*$-superrigid actions, starting with the works of Peterson \cite{Pet} and Popa and Vaes \cite{PV}; but knowing that all actions of a given group are $W^*$-superrigid is a phenomenon that has only been established in very few cases so far \cite{HPV,CIK,HH1}. Our $W^*$-superrigidity theorems are obtained as a consequence of orbit equivalence rigidity and of the uniqueness of $L^\infty(X)$ as a Cartan subalgebra of $L(\Hig_\sigma\actson X)$, up to unitary conjugacy. The latter is obtained from the existence of decompositions of $\Hig_\sigma$ as an amalgamated free product, using a theorem of Ioana \cite{Ioa}. 

The second statement in the corollary remains true for any generalized Higman group $\Hig_\sigma$ if we further assume that the action $\Hig_\sigma\actson X$ is ergodic in restriction to every finite-index subgroup of $\Hig_\sigma$. When $\Hig_\sigma=\Hig_k$, this extra assumption is trivially satisfied, simply because $\Hig_k$ does not have any proper finite-index subgroup \cite{Hig}. To our knowledge, the groups $\Hig_k$ with $k\ge 5$ are the first instances of groups, all of whose free, ergodic, probability measure-preserving actions are $W^*$-superrigid in the sense of Corollary~\ref{cor:oe-w}(2).
	
	\subsection{Proof ingredients}

We define a \emph{BS subgroup} of $H_\sigma$ to be a conjugate of a subgroup of $H_\sigma$ generated by two consecutive generators $a_i,a_{i+1}$. These subgroups play a central role in our proof of Theorem~\ref{theointro:main}. We now explain the main ingredients, in the (perhaps most interesting) case where all BS subgroups involved in the definition of $H_\sigma$ are amenable. 
	
Using techniques that originate in the work of Furman \cite{Fur}, measure equivalence rigidity of $\Hig_\sigma$ reduces to the following theorem about self-couplings of $\Hig_\sigma$.

\begin{theo}\label{theointro:main-2}
Let $k\ge 5$, and let $\sigma=((m_1,n_1),\dots,(m_k,n_k))$ be a $k$-tuple of pairs of non-zero integers, with $|m_i|\neq |n_i|$ for every $i\in\{1,\dots,k\}$.

For every self measure equivalence coupling $\Sigma$ of $\Hig_\sigma$, there exists a Borel $(\Hig_\sigma\times\Hig_\sigma)$-equivariant map $\Sigma\to\widehat{\Hig}_\sigma$, where the action on $\widehat{\Hig}_\sigma$ is via $(h_1,h_2)\cdot h=h_1hh_2^{-1}$.
\end{theo}

Writing $G=\Hig_\sigma$, proving Theorem~\ref{theointro:main-2} is related to studying stable orbit equivalence between two free measure-preserving actions $\alpha_1:G\to\Aut(Y_1)$ and $\alpha_2:G\to \Aut(Y_2)$ on standard probability spaces. 

Our strategy has two main steps: this general structure is inspired by Kida's work on measure equivalence rigidity of mapping class groups \cite{Kid-me}, even though the way in which each step is carried differs significantly. 

In the first step, we consider the common orbit equivalence relation $\mathcal{R}$ (on a base space $Y$), arising from restricting $\alpha_1$ and $\alpha_2$ to appropriate positive measure subsets of $Y_1$ and $Y_2$. We show that subrelations of $\mathcal{R}$ arising from restricting $\alpha_1$ to a BS subgroup, also correspond (up to a countable partition of the base space) to restricting $\alpha_2$ to a BS subgroup. And two BS subgroups that intersect give subrelations that intersect. Informally, this means that $\mathcal{R}$ completely encodes the BS-subgroups of $\Hig_\sigma$ and their intersection pattern. 

In the second step, we show that $\Hig_\sigma$ is completely recovered by the intersection pattern of its BS subgroups. More precisely, this intersection pattern is encoded in a combinatorial object, Martin's  \emph{intersection graph} $\Theta$, and we show that the natural map from $\widehat{\Hig}_\sigma$ to $\Aut(\Theta)$ is an isomorphism (see Section~\ref{subsec:combinatorial}).
	
The first step is in fact carried in a broader context, inspired by the following. The group $\Hig_\sigma$ has the structure of a complex of groups. This naturally yields an action of $\Hig_\sigma$ on a polyhedral complex $X_\sigma$, which carries an invariant $\mathrm{CAT}(-1)$ metric because we have assumed that $k\ge 5$. This complex was first introduced by Bridson and Haefliger \cite{BH}, and extensively studied by Martin in \cite{Mar,Mar2}. It has one vertex per left coset of a Baumslag-Solitar subgroup $\langle a_i,a_{i+1}\rangle$, one edge per left coset of a subgroup $\langle a_i\rangle$, and one $2$-cell (isomorphic to a $k$-gon) per left coset of the trivial subgroup, where a cell $\tau$ is identified to a face of a cell $\tau'$ if the coset associated to $\tau$ contains the coset associated to $\tau'$. Martin proved in \cite{Mar,Mar2} that the action of $H_\sigma$ on $X_\sigma$ is \emph{weakly acylindrical}: there exist $L,N>0$ such that for any two points $x,y\in X_\sigma$ with $d(x,y)\ge L$, one has $|\Stab_{\Hig_\sigma}(\{x,y\})|\le N$.

To any action of a group $G$ on a polyhedral complex $X$, we associate an \emph{intersection	graph}, considered by Martin in \cite{Mar} in the case of the Higman group: it has the same vertex set as $X$, and two vertices are joined by an edge if their $G$-stabilizers (for the $G$-action on $X$) have a nontrivial intersection. 

The following theorem recasts the first step of our proof at the level of the self measure equivalence coupling of $G$.

\begin{theo}\label{theo:combination}
Let $X$ be a connected $\mathrm{CAT}(-1)$ piecewise hyperbolic polyhedral complex with countably many cells in finitely many isometry types. Let $G$ be a torsion-free countable group acting by cellular isometries on $X$ without any cell inversion, and let $\Theta$ be the intersection graph of the action $G\actson X$. Assume that:
		\begin{enumerate}
			\item[1)] The stabilizer of every vertex of $X$ is amenable.
			\item[2)] Edge stabilizers have infinite index in the incident vertex stabilizers.
			\item[3)] The $G$-action on $X$ is weakly acylindrical.
			\item[4)] For each vertex $v\in X^{(0)}$, there exist two distinct vertices $v_1,v_2\in X^{(0)}\setminus\{v\}$ such that $\Stab_G(\{v,v_1,v_2\})$ is infinite.
		\end{enumerate}
		Then for every self measure equivalence coupling $\Sigma$ of $G$, there exists a $(G\times G)$-equivariant measurable map $\Sigma\to\Aut(\Theta)$, where the $(G\times G)$-action on $\Aut(\Theta)$ is via $(g_1,g_2)\cdot f(x)=g_1f(g_2^{-1}x)$.
	\end{theo}
	
For the Higman groups with amenable BS subgroups, our main theorems follow because $\Aut(\Theta)$ coincides with $\widehat{\Hig}_\sigma$ (see Theorem~\ref{theo:rigidity} below). Let us also mention that Theorem~\ref{theo:full} of the article gives a variation that allows for non-amenable vertex groups, which is used to treat the general case.

\subsubsection{On the proof of Theorem~\ref{theo:combination}}\label{subsec:groupoid-argument}
		
At the level of stable orbit equivalence, proving Theorem~\ref{theo:combination} amounts to the following. Keeping the notation from the previous section, we have two cocycles $\rho_1,\rho_2:\calr\to G$, coming from the two actions $\alpha_1,\alpha_2$ of $G$. Our goal is to prove that the two cocycles are \emph{conjugate} (or \emph{cohomologous}) inside $\Aut(\Theta)$, i.e.\ that there exists a Borel map $\varphi:Y\to\Aut(\Theta)$ such that for almost every $(x,y)\in\calr$, one has $\rho_2(x,y)=\varphi(y)\rho_1(x,y)\varphi(x)^{-1}$. In the body of the paper, we work in the slightly more general framework of measured groupoids.
	
An amenable subgroup of $G$ can act on $X$ in two different ways: either it is \emph{elliptic}, i.e.\ it fixes a point in $X$, or it is \emph{$\partial_\infty X$-elementary}, i.e.\ it fixes a point in $\partial_\infty X$. Our weak acylindricity assumption ensures that these two possibilities are mutually exclusive. It also ensures that $\partial_\infty X$-elementary infinite subgroups fix at most two points in $\partial_\infty X$: if they fixed three points, a barycenter argument would also yield a fixed point inside $X$, which is not possible. 

In a similar way, given a cocycle $\rho:\calg\to G$, an amenable subrelation $\cala$ of $\calr$ is called \emph{elliptic} with respect to $\rho$ if $\rho(\cala)$ fixes a point in $X$, and \emph{$\partial_\infty X$-elementary} with respect to $\rho$ if there is an $(\cala,\rho)$-equivariant measurable map from $Y$ to the set of nonempty subsets of $\partial_\infty X$ of cardinality at most $2$. Using an argument of Adams \cite{Ada} and some $\mathrm{CAT}(-1)$ geometry, we prove that, up to a countable partition of the base space $Y$, any amenable subrelation of $\calr$ has one of these two forms. We use an extra argument, based on the weak acylindricity, to prove that these two forms are mutually exclusive.

Our main task is to find a way to distinguish these two types of amenable subrelations, with no reference to the cocycle $\rho$. That way, a maximal elliptic  subrelation for $\rho_1$ will also be elliptic for $\rho_2$. In other words, subrelations of $\calr$ coming from restricting the $\alpha_1$-action of $G$ on $Y_1$ to some vertex stabilizer for $G\actson X$, are also of the same form for the $\alpha_2$-action on $Y_2$ (up to a countable partition of $Y$). This is the crucial point to define the map $Y\to\Aut(\Theta)$.	

At first glance, the above task might seem hopeless: by the Ornstein-Weiss theorem, all amenable subrelations are isomorphic. The key idea is that they are not distinguished intrinsically, but by the way in which they intersect other amenable subrelations. This is where the fourth assumption of Theorem~\ref{theo:combination} is crucial: it ensures that elliptic subrelations intersect sufficiently many other maximal amenable subrelations (in comparison, this would fail if $G$ were, say, a free product of amenable groups, acting on the associated Bass-Serre tree).

At the group-theoretic level, if $H_1,H_2,H_3$ are three maximal amenable subgroups of $G$, no two of which share a common finite-index subgroup, and such that $H_1\cap H_2\cap H_3$ is infinite, then all $H_i$ have to be elliptic. Indeed, this intersection cannot be both elliptic and $\partial_\infty X$-elementary, so either all $H_i$ are elliptic, or all $H_i$ are $\partial_\infty X$-elementary. The latter is not possible, because $H_1\cap H_2\cap H_3$ would have at least $3$ fixed points in $\partial_\infty X$, and a barycenter argument would also yield a fixed point in $X$.

We can similarly characterize maximal elliptic amenable subrelations (with respect to either $\rho_1$ or $\rho_2$) in terms of their non-isolation in the ambient relation $\calr$. An extra ingredient in the groupoid-theoretic setting is the amenability of the $G$-action on $\partial_\infty X$, established in our earlier work \cite[Theorem~1]{HH1}.

Let us here mention the key difference between our approach and Kida's proof of the measure equivalence rigidity of mapping class groups \cite{Kid-me}. Given a measured equivalence relation coming from an action of the mapping class group, a key point in his proof was to recognize amenable subrelations coming from Dehn twist subgroups. Again he needed to distinguish these amenable subrelations with an amenable subrelation which is $\partial_\infty\mathcal C$-elementary with respect to the action of the group on the curve complex $\mathcal C$. This was done by observing that the former have a non-amenable normalizer (this corresponds to the group theoretic fact that the centralizer of a Dehn twist is a maximal nonamenable subgroup of the mapping class group with infinite center), while the latter only have amenable normalizer. However, this breaks down for classical Higman groups $\Hig_k$, where all centralizers (or even quasi-centralizers) of amenable subgroups are amenable, so we needed a new argument, in terms of intersection patterns of amenable subrelations. On the other hand, when the involved BS subgroups are nonamenable, we do exploit the presence of nonamenable quasi-normalizers in a crucial way.
	
 When $k=4$, it is still true that one can characterize BS subgroups of $\Hig_k$ as maximal amenable subgroups which are not isolated. While there are several arguments to prove this, we do not see a way to adjust them to the groupoid setting, which is significantly more involved. In particular, the above strategy does not apply, as we do not know a hyperbolic space where the Higman groups acts, where the action on the boundary is amenable, and at the same time some form of acylindricity holds.

	\subsubsection{A combinatorial/geometric rigidity theorem}
	\label{subsec:combinatorial}

The second main tool in our proof is the following combinatorial rigidity theorem.	

	\begin{theo}[{see Theorem~\ref{theo:crigidity0}}]
		\label{theo:rigidity}
	Let $k\ge 4$, and let $\sigma=((m_1,n_1),\dots,(m_k,n_k))$ be a $k$-tuple of pairs of non-zero integers, with $|m_i|<|n_i|$ for every $i\in\{1,\dots,k\}$. Let $X_\sigma$ be as above, and let $\Theta_\sigma$ be its intersection graph.
		
	Then the natural maps $\widehat{\Hig}_\sigma\to\Aut(X_\sigma)$ and $\widehat{\Hig}_\sigma\to\Aut(\Theta_\sigma)$ are isomorphisms.
	\end{theo}
	
Combinatorial rigidity theorems have a long history. One cornerstone is Tits's work on spherical buildings \cite{Tit}, which inspired analogous theorems in different settings. Combinatorial rigidity theorems are often used in proofs of measure equivalence rigidity. For instance, Kida crucially exploited Ivanov's theorem stating that all automorphisms of the curve graph of the surface come from the action of the extended mapping class group \cite{Iva}.  Other instances where these were successfully exploited include: Crisp's $\Theta$-graph for large-type Artin groups, a variation over the free splitting graph for $\Out(F_N)$, the disk graph for handlebody groups, or (in a weaker sense) Kim and Koberda's extension graph for our previous measure equivalence classification theorem among right-angled Artin groups \cite{HH2}. 

The key reason which leads to the rigidity phenomenon in Theorem~\ref{theo:rigidity} is different from previous work -- it crucially relies on the interlocking pattern of BS subgroups. More precisely, if we set $m_i=n_i=1$ for all $i$, then the theorem fails -- this follows from \cite[Section 11]{BKS}. Actually, we suspect the theorem fails as long as $|m_i|=|n_i|$ for some $i$. If we still demand $|m_i|\neq |n_i|$, but disrupt the interlocking pattern by changing the $i$-th relation from, say $a_ia_{i+1}a_i^{-1}=a_{i+1}^{2}$, to $a_{i+1}a_{i}a_{i+1}^{-1}=a_{i}^{2}$ and keeping other relations, then the theorem fails again, see Remark~\ref{rk:example}. 

Let us point out that from this viewpoint, Theorem~\ref{theo:rigidity} is also much more subtle than rigidity theorems regarding the automorphism group of the underlying group $\Hig_\sigma$. For instance, when all $m_i=n_i=1$, and $k\ge 5$, the group still has finite outer automorphism group \cite{Ser}, even though the automorphism group of $X_\sigma$ is much larger.

	A key point in the proof of the rigidity theorem for $\Aut(X_\sigma)$ is a partial combinatorial rigidity property of the Cayley graph of a Baumslag-Solitar group $\BS(m,n)=\langle a,t|ta^mt^{-1}=a^n\rangle$ with $|m|<|n|$, namely: if a set-theoretical bijection $\phi:\BS(m,n)\to \BS(m,n)$ is \emph{line-preserving}, i.e.\ $\phi$ preserves horizontal lines (left cosets of $\langle a \rangle$) and vertical lines (left cosets of $\langle t\rangle$), then $\phi$ respects the order along the vertical lines coming from the orientation of edges of Cayley graphs. However, the order along the horizontal lines might not be preserved -- this is why we call it a partial rigidity result. The connection to rigidity of $X_\sigma$ is that an automorphism of $X_\sigma$ that fixes a vertex $x$ induces such a line-preserving bijection of the Baumslag-Solitar subgroup associated to $x$, because this subgroup is in natural bijection with all $2$-cells that contain $x$. Thus we have a local partial rigidity for automorphisms of $X_\sigma$. In $\Hig_\sigma$ the non-rigid horizontal lines of one Baumslag-Solitar group correspond to the rigid vertical lines of another Baumslag-Solitar group. This allows to propagate the local partial rigidity to obtain a global rigidity.
	To transfer the combinatorial rigidity from $X_\sigma$ to $\Theta_\sigma$, the key is to show there is a 1-1 correspondence between induced $k$-cycles in $\Theta_\sigma$ and 2-cells in $X_\sigma$; our proof of this fact exploits the non-positive curvature of $X_\sigma$ through the use of van Kampen diagrams to analyze these cycles (compare \cite[Section~6]{BKS}). 
	
\begin{rk}
An interested reader may want to consult \cite{Cri} for comparison. There are analogies between the spaces we consider: our intersection graph corresponds to Crisp's $\Theta$-graph, and our $X_\sigma$ corresponds to his modified Deligne complex. But as explained above, the use of the interlocking pattern of Baumslag-Solitar groups is a new ingredient in the present paper that does not appear in \cite{Cri}. The argumentation in \cite{Cri}, and the reason why the relevant objects are rigid, are in fact quite different. 
\end{rk}

\paragraph*{Structure of the paper.} Section~\ref{sec:background} contains background on generalized Higman groups and the polyhedral  complexes they act on. The next two sections aim at proving our combinatorial/geometric rigidity theorem (Theorem~\ref{theo:rigidity}): Section~\ref{sec:BS} establishes the required partial rigidity theorem for Baumslag-Solitar groups, and the proof of Theorem~\ref{theo:rigidity} is carried in Section~\ref{sec:combinatorial-rigidity}. The measured group-theoretic arguments are carried in Section~\ref{sec:me}, where our main theorem is proved, as well as Theorems~\ref{theointro:main-2} and~\ref{theo:combination}, and Corollary~\ref{corintro:lattice}. The arguments leading to the $W^*$-rigidity of free, ergodic, probability measure-preserving actions of Higman groups (Corollary~\ref{cor:oe-w}) are carried in Section~\ref{sec:w}. Finally, the paper contains two appendices, which contain the relevant background and terminology regarding polyhedral complexes (Appendix~\ref{subsec:polyhedral complex}) and measured groupoids (Appendix~\ref{sec:appendix-groupoids}).
	
	\paragraph*{Acknowledgments.} We thank Vincent Guirardel for interesting discussions around this project. We also thank the anonymous referee for their careful reading of our paper and valuable suggestions.

	\section{Background on Higman groups}\label{sec:background}

In this section, we review generalized Higman groups and polyhedral complexes on which they act. We refer to Appendix~\ref{subsec:polyhedral complex} or \cite{BH} for some general language of polyhedral complexes and their links.

	\subsection{Generalized Higman groups and associated complexes}
	\label{subsec:Higman background}
	
	Let $k\ge 4$, and let $\sigma=((m_1,n_1),\dots,(m_k,n_k))$ be a $k$-tuple of pairs of non-zero integers, with $|m_i|<|n_i|$ for every $i\in\{1,\dots,k\}$. Recall from the introduction that the \emph{generalized Higman group} $\Hig_\sigma$ is defined by the following presentation:
	\[\Hig_\sigma=\langle\{a_i\}_{i\in\mathbb Z/k\mathbb Z}\mid \{a_ia^{m_i}_{i+1}a^{-1}_{i}=a^{n_i}_{i+1}\}_{i\in\mathbb Z/k\mathbb Z}\rangle.\] Notice that we are not losing any generality by assuming $|m_i|<|n_i|$, instead of just $|m_i|\neq |n_i|$ as we did in the introduction: indeed, replacing the pair $(m_i,n_i)$ by $(n_i,m_i)$ in $\sigma$ yields an isomorphic group (via the isomorphism sending the generator $a_i$ to $a_i^{-1}$).  We also mention that replacing any of the tuples $(m_i,n_i)$ by $(-m_i,-n_i)$ yields an isomorphic group, because the defining relation $a_ia_{i+1}^{m_i}a_i^{-1}=a_{i+1}^{n_i}$ is equivalent to $a_ia_{i+1}^{-m_i}a_i^{-1}=a_{i+1}^{-n_i}$. In the rest of the paper, we will always adopt these conventions regarding the possible values of $\sigma$.
	 From now on we will fix $k\ge 4$ and $\sigma$ as above, and write $G=\Hig_\sigma$.
	
	Let $K$ be the 2-dimensional $k$-gon with its natural cell structure. We label its vertices cyclically by $\{\bx_i\}_{i\in\mathbb Z/k\mathbb Z}$, and let the edge between $\bx_i$ and $\bx_{i+1}$ be $\be_i$. To each vertex $\bx_i$ of $K$, we associate a \emph{vertex group} $G_{\bx_i}$, with presentation $\langle \alpha_{i-1},\alpha_{i}\mid \alpha_{i-1}\alpha^{m_i}_{i}\alpha^{-1}_{i-1}=\alpha^{n_i}_{i}\rangle$. To each edge $\be_i$ of $K$, we associate an \emph{edge group} $G_{\be_i}$, isomorphic to $\mathbb Z=\langle\beta_i\rangle$. There are natural embeddings  $G_{\be_i}\to G_{\bx_i}$ and $G_{\be_i}\to G_{\bx_{i+1}}$, sending the generator $\beta_i$ of $\mathbb{Z}$ to $\alpha_{i}$ on both sides.
	
	\begin{lemma}\label{lem:injective}
 For every $i\in\{1,\dots,k\}$, the natural homomorphism $G_{\bx_i}\to G$ (sending $\alpha_{i-1},\alpha_{i}$ to $a_{i-1},a_{i}$) is injective.
	\end{lemma}
	
	Here it is important to have assumed that $k\ge 4$: the lemma is not true when $k=3$, see \cite{Hig}.
	
	We will explain the proof of this lemma after the statement of Lemma~\ref{lem:sc}. From now on, we will treat vertex groups $G_{\bar x_i}$ and edge groups $G_{\bar e_i}$ as subgroups of $G$ via the above embedding.
	Let $\mathcal P$ be the poset (ordered by inclusion) of all left cosets of an edge group, a vertex group, or a trivial subgroup of $G$. Note that \emph{a priori} left cosets of two different vertex groups (or edge groups) might give the same subset of $G$, in which case they are viewed as the same element of $\mathcal P$ (though we will see later in Lemma~\ref{lemma:fundamental-domain} that this actually does not happen).
	Let $X_\Delta$ be the \emph{geometric realization} of $\mathcal P$, i.e.\ the simplicial complex whose vertices are the elements of $\mathcal P$, and whose simplices correspond to chains in $\mathcal P$. There is a natural simplicial action $G\actson X_\Delta$.  Note that the stabilizer of each vertex with respect to this action is a conjugate of a vertex group, or an edge group, or the trivial subgroup. Denoting by $K_\Delta$ the barycentric subdivision of $K$, there is a natural simplicial map $K_\Delta\to X_\Delta$, sending the vertex $\bar x_i$ to the vertex corresponding to $G_{\bar x_i}$, the midpoint of $\bar e_i$ to the vertex corresponding to $G_{\bar{e}_i}$, and the center of the $2$-cell of $K$ to the vertex corresponding to $\{1\}$. 
	
	\begin{lemma}\label{lemma:fundamental-domain}
		The natural map $K_\Delta\to X_\Delta$ is an embedding whose image is a strict fundamental domain for the $G$-action on $X_\Delta$, i.e.\ $K_\Delta$ meets each orbit in exactly one point.
	\end{lemma}
	
	Again we will explain the proof of Lemma~\ref{lemma:fundamental-domain} after the statement of Lemma~\ref{lem:sc}. From $X_\Delta$, we define another cell complex $X$, homeomorphic to $X_\Delta$, as follows. 
	
	The 1-skeleton $X^{(1)}$ of $X$ is homeomorphic to the full subgraph of $X^{(1)}_\Delta$ spanned by all vertices of $X_\Delta$ that correspond to left cosets of vertex groups or edge groups; however, only vertices of $X^{(1)}_\Delta$ corresponding to left cosets of vertex groups are vertices of $X^{(1)}$, while vertices of $X^{(1)}_\Delta$ corresponding to left cosets of edge groups give midpoints of edges of $X^{(1)}$.
	
	Notice that the closure of every connected component of $X_\Delta\setminus X^{(1)}$ is then homeomorphic to a closed disk, which comes with the natural structure of a polygon isomorphic to $K$. These polygons are the $2$-cells of $X$: more precisely, $X$ is obtained from $X^{(1)}$ by adding one $2$-cell, isomorphic to $K$, per element of $G$ (corresponding to a vertex $v\in X_\Delta^{(0)}$), and gluing its edges onto the finite subgraph of $X_\Delta^{(1)}$ spanned by all vertices at distance $1$ from $v$. The simplicial complex $X_{\Delta}$ is then just the barycentric subdivision of $X$. Moreover, there is an order-reversing isomorphism between the poset of closed cells in $X$ (under containment) and $\mathcal P$. We will denote by $\hat K$ the $2$-cell corresponding to the subgroup $\{\mathrm{id}\}$ under this isomorphism. By Lemma~\ref{lemma:fundamental-domain}, $G\backslash X$ is isomorphic to $K$, which induces a quotient map $\pi:X\to K$. We call $X$ the \emph{developed complex} of $G$.
	
	\begin{lemma}
		\label{lem:sc}
		The simplicial complex $X_\Delta$ is simply connected. Hence $X$ is simply connected.
	\end{lemma}
	
	The proofs of Lemmas~\ref{lem:injective},~\ref{lemma:fundamental-domain} and~\ref{lem:sc} rely on the language of complexes of groups, and follow from the literature in the case where all vertex groups are isomorphic to $\BS(1,n)$, see \cite{BH}, also \cite{stallings1991non,Mar}. The more general case follows exactly the same proof. Now we give more precise references.
	
	We refer to \cite[Chapter II.12]{BH} for the definition of simple complexes of groups and simple morphisms \cite[Definition~II.12.11]{BH}, their strict developability \cite[II.12.15]{BH} and their developments \cite[Theorem~II.12.18]{BH}.
	Let $\mathcal{Q}$ be the poset of closed cells of $K$. We consider a simple complex of groups $\mathcal K$ over $\mathcal{Q}$, such that the local group at $\be_i$ (resp.\ $\bx_i$) is $G_{\be_i}$ (resp.\ $G_{\bx_i}$), and the local group $G_K$ at the 2-cell of $K$ is the trivial group. Let $\widehat{G(\mathcal K)}$ be the group obtained from a direct limit of its local groups, see \cite[Definition II.12.12]{BH}. Then $\widehat{G(\mathcal K)}$ is isomorphic to the generalized Higman group $G$.
	
	Let $K_\Delta$ be the barycentric subdivision of $K$. Then $K_\Delta$ is isomorphic to the geometric realization $|\mathcal{Q}|$ of $\mathcal{Q}$. The \emph{rank} of a vertex of $|\mathcal{Q}|=K_\Delta$ is the dimension of the element of $\mathcal{Q}$ associated with this vertex. We metrize $|\mathcal{Q}|$ in such a way that each triangle in $|\mathcal{Q}|$ is a right-angled isoceles triangle with flat metric such that the angle at the rank 1 vertex is $\pi/2$ and the long side of the triangle has length 1. Now one analyzes the local development as in \cite[II.12.29]{BH} to prove that the complex of groups $\calk$ over $\calg$ with this metric is non-positively curved when $k\ge 4$. One can thus apply \cite[Theorem~II.12.28]{BH} to deduce that the complex of groups $\mathcal K$ is strictly developable when $k\ge 4$, in particular there is a group $G'$ and a simple morphism from $\mathcal K$ to $G'$ which is injective on local groups. This together with \cite[II.12.13]{BH} imply Lemma~\ref{lem:injective}. As $X_\Delta$ is the development of $K_\Delta$ with respect to the simple morphism $\mathcal K\to \widehat{G(\mathcal K)}=G$, Lemma~\ref{lemma:fundamental-domain} follows from \cite[Theorem~II.12.18(2)]{BH}, and Lemma~\ref{lem:sc} follows from \cite[Proposition II.12.20(4)]{BH}.

	\subsection{Metric on $X$ and some consequences}\label{sec:metric}
	
	We endow $X$ with a piecewise Euclidean (or hyperbolic) metric such that each copy of $K$ in $X$ is a right-angled regular polygon in a Euclidean plane (when $k=4$) or in a hyperbolic plane (when $k>4$) -- in the Euclidean case we further impose all edge lengths to be equal to $1$. The $G$-action on $X$ is by isometries. Moreover, the action \emph{does not have inversion}, i.e.\ if an element in $G$ preserves a cell setwise, then it fixes the cell pointwise. It follows from standard arguments in complexes of groups that $X$ is $\CAT(0)$ when $k=4$ and is $\CAT(-1)$ when $k=5$. Now we give an alternative explanation for readers not comfortable with complexes of groups, assuming Lemma~\ref{lem:sc}.
	
	Let $x\in X^{(0)}$ be a vertex corresponding to a left coset of a vertex group $G_{\bx}$ with $\bx\in K$. Let $\be_1$ and $\be_2$ be the edges of $K$ containing $\bx$.  We use $\lk(x,X)$ to denote the link of $x$ in $X$, equipped with the angular metric. The vertices in $\lk(x,X)$ are in 1-1 correspondence with left cosets of $G_{\be_1}$ and $G_{\be_2}$ in $G_{\bx}$, and two vertices of $\lk(x,X)$ are joined by an edge if and only if the corresponding left cosets have non-empty intersection. Note that $\lk(x,X)$ is a bipartite graph (by considering vertices of $\lk(x,X)$ corresponding to left cosets of $G_{\be_1}$, and vertices corresponding to left cosets of $G_{\be_2}$). The edges of $\lk(x,X)$ all have length $\pi/2$. Thus the girth of $\lk(x,X)$ as a  simplicial graph is at least 4, and each embedded cycle in $\lk(x,X)$ has length at least $2\pi$. Then $X$ is $\CAT(0)$ when $k=4$ and $\CAT(-1)$ when $k>4$ (cf. Lemma~\ref{lemma:link condition}). 
	
	As in \cite{Mar}, we orient edges of $K$ such that $\bar{e}_i$ is oriented from $\bx_{i}$ to $\bx_{i+1}$. We pull back the edge orientation of $K$ to $X$ via $X\to K\cong G\backslash X$.
	Let $c\subset X$ be a vertex or an edge. We define $G_c$ to be the stabilizer of $c$ under $G\actson X$. Then $G_c$ is a conjugate of a vertex group, or of an edge group of $G$. Thus $G_c$ actually fixes $c$ pointwise. We record below some basic properties of intersections of cell stabilizers.

	\begin{lemma}[{see \cite[Section 2.1]{Mar}}]
		\label{lemma:intersection of vertex stabilizer}
		The following assertions hold.
		\begin{enumerate}
			\item Let $x\in X^{(0)}$ be a vertex, and let $e\in X^{(1)}$ be an edge incident on $x$ and oriented towards $x$. Then for any edge $e'\neq e$ incident on $x$, one has $G_e\cap G_{e'}=\{1\}$.
			\item For two vertices $x_1\neq x_2$ of $X$, one has $G_{x_1}\cap G_{x_2}\neq\{1\}$ if and only if one of the following situations happens:
			\begin{enumerate}
				\item $x_1$ and $x_2$ are adjacent;
				\item both $x_1$ and $x_2$ are adjacent to a vertex $x$ such that the edge $\overline{xx_i}$ connecting $x$ and $x_i$ is oriented away from $x$ for every $i\in\{1,2\}$.
			\end{enumerate}
			\item Given any distinct integers $s\neq s'$ and any $i\in\mathbb{Z}/k\mathbb{Z}$, the intersection $(a^s_{i-1}G_{\bx_{i+1}}a^{-s}_{i-1})\cap (a^{s'}_{i-1}G_{\bx_{i+1}}a^{-s'}_{i-1})$ is infinite cyclic.	
			\item There is a bound on the size $m$ of a chain $x_1,\dots,x_m$ of vertices of $X$ such that for every $i\in\{1,\dots,m-1\}$, the pointwise stabilizer of $\{x_1,\dots,x_{i+1}\}$ has infinite index in the pointwise stabilizer of $\{x_1,\dots,x_i\}$.
		\end{enumerate}
	\end{lemma}
	
	\begin{proof}
	The first assertion extends \cite[Lemma~2.1]{Mar}, with a similar proof. It is a consequence of the following two facts. First, $\langle a_i\rangle$ does not intersect any conjugate of $\langle a_{i+1}\rangle$ in $G_i=\langle a_i,a_{i+1}|a_{i}a_{i+1}^{m_i}a_{i}^{-1}=a_{i+1}^{n_i}\rangle$: this is because $a_i$ is loxodromic in the Bass-Serre tree $T_i$ of $G_i$, while $a_{i+1}$ is elliptic. Second (and this crucially uses our assumption that $|m_i|<|n_i|$), the subgroup $\langle a_{i}\rangle$ is malnormal in $G_i$. We now prove this second fact. Let $g\in G_i$ be such that $g\langle a_i\rangle g^{-1}\cap\langle a_i\rangle$ is nontrivial; we aim to prove that $g\in\langle a_i\rangle$. Then $g$ preserves the axis $\ell$ of $a_i$ in the Bass-Serre tree $T_i$ of $G_i$. We first consider the case where $g$ acts elliptically on $T_i$. In fact in this case $g$ fixes $\ell$ pointwise (it cannot permute its endpoints): this is because edges of $T_i$ have an orientation coming from the projection map from the Cayley graph of $G_i$ to $T_i$, and the $G_i$-action preserves this orientation.  
	 But our assumption that $|m_i|\neq |n_i|$ ensures that $\ell$ contains at least a half-ray $[x_0,\xi)$ along which the stabilizers of the subsegments $[x_0,x_n]$ (where $x_n$ denote the successive vertices along the ray) form a strictly decreasing sequence of cyclic subgroups, so the pointwise stabilizer of this half-ray is trivial. Therefore $g=1$ in this case. 
		%two adjacent edges never have the same stabilizer, and as edge stabilizers are cyclic, we deduce that stabilizers of rays are trivial. Therefore $g=1$ in this case. 
		
		Now, if $g$ acts loxodromically on $T_i$, then there exists $k\in\mathbb{Z}$ such that $ga_i^{-k}$ fixes the axis of $a_i$ pointwise, and by the above it follows that $g=a_i^k$. 
		This completes our proof of Assertion~1.
		
		For the if direction of assertion 2, the only thing to observe in case (b) is that any two edges incident on $x$ and oriented away from $x$ have commensurate stabilizers (because $\langle a_{i+1}\rangle$ is a commensurated subgroup in $\langle a_i,a_{i+1}\rangle$). The only if direction is a consequence of Assertion~1, see \cite[Corollaries~2.2 and~2.3]{Mar}. 
		
	Before proving the last two assertions, we make the following observation: if $v\neq v'$, then $\Stab_G(v)\cap \Stab_G(v')$ is either infinite cyclic or trivial. Indeed $\Stab_G(v)\cap \Stab_G(v')$ fixes the geodesic segment $s$ from $v$ to $v'$ pointwise, hence fixes any cell of $X$ whose interior contains a point of $s$. Note that if a subgroup fixes a cell of $X$, then it fixes the cell pointwise. As $v\neq v'$, we know $\Stab_G(v)\cap \Stab_G(v')$ fixes an edge of $X$. The observation follows because edge stabilizers of $X$ are infinite cyclic. 
		
		Assertion 4 (with $m=3$) is an immediate consequence of the above observation. For Assertion~3, note that $(a^s_{i-1}G_{\bx_{i+1}}a^{-s}_{i-1})\cap (a^{s'}_{i-1}G_{\bx_{i+1}}a^{-s'}_{i-1})$ is the common stabilizer of the vertices $a_{i-1}^s\bx_{i+1}$ and $a_{i-1}^{s'}\bx_{i+1}$. As $s\neq s'$, these two vertices are distinct, so the observation made in the previous paragraph ensures that $(a^s_{i-1}G_{\bx_{i+1}}a^{-s}_{i-1})\cap (a^{s'}_{i-1}G_{\bx_{i+1}}a^{-s'}_{i-1})$ is at most cyclic.  
		On the other hand, this intersection contains a finite index subgroup of $\langle a_i\rangle$, so Assertion~3 follows.
	\end{proof}
	
	\begin{rk}\label{rk:X}
		Note that in Assertion~2.(b) of Lemma~\ref{lemma:intersection of vertex stabilizer}, the edges $\overline{xx_1}$ and $\overline{xx_2}$ are not contained in the same 2-cell by the description of $\lk(x,X)$ as above. Thus $\angle_x(x_1,x_2)\ge\pi$; hence $\overline{x_1x}$ and $\overline{xx_2}$  concatenate to a geodesic.
	\end{rk}
	
	\begin{lemma}
		\label{lemma:Higman property}
		The group $G$ is torsion-free and ICC, i.e.\ each non-trivial conjugacy class in $G$ is infinite.
	\end{lemma}
	
	\begin{proof}
		Assume towards a contradiction that $G$ contains a torsion element $g$.  Being a torsion element acting on a $\CAT(0)$ space, the element $g$ fixes a point in $X$ (\cite[Theorem~II.6.7]{BH}). Hence $g$ fixes a vertex as the action $G\actson X$ is without inversion. But this is impossible as each vertex stabilizer is a Baumslag-Solitar group, hence torsion-free. 
		
		To prove the ICC property, suppose that there exists a nontrivial element $g\in G$ whose conjugacy class is finite. Then $G$ has a finite index subgroup $G'$ such that $H:=\langle g\rangle$ is normalized by $G'$. As $G$ acts on $X$ by  cellular isometries and $X$ has finite shapes, the action of $G$ is semi-simple \cite[Exercise~II.6.6(2)]{BH}. By the classification of isometries of $\CAT(0)$ spaces in \cite[Chapter~II.6]{BH}, the action of $H$ is either elliptic (i.e.\ $H$ fixes a point) or hyperbolic (i.e.\ there is an $H$-invariant axis on which $H$ acts by translations).
		If $H$ fixes a point, then the fix point set $F$ of $H$ is bounded by Lemma~\ref{lemma:intersection of vertex stabilizer}(2). Then $F$ is $G'$-invariant, which is a contradiction as $G$ acts with unbounded orbits (and therefore so does $G'$). If $H$ translates an axis in $X$, then $G'$ has a finite index subgroup $G''$ containing $H$ as a direct factor by \cite[Theorem II.7.1(5)]{BH}. This is again impossible for the following reason. As $G''$ is finite index in $G$, we can find an infinite order element $b\in G''$ which fixes a vertex of $X$. Again the fix point set $F_b$ of $\langle b\rangle$ is bounded by Lemma~\ref{lemma:intersection of vertex stabilizer}(2). As $H$ commutes with $\langle b\rangle$, the set $F_b$ is $H$-invariant, which contradicts that $H$ acts hyperbolically on $X$.
	\end{proof}

	\subsection{Weak acylindricity of the action}
	
	Let $H$ be a group acting on a metric space $Z$. The $H$-action on $Z$ is said to be \emph{weakly acylindrical} if there exist $L>0,N>0$ such that for any two points $x,y\in Z$ with $d(x,y)\ge L$, the common stabilizer of $x$ and $y$ has cardinality at most $N$.
	
	Note that if $Z$ is an $M_{\kappa}$-polyhedral complex with finite shapes and $H$ acts by cellular isometries, then it suffices to check the special case when $x$ and $y$ are vertices when we prove that the action of $H$ is weakly acylindrical. Thus Lemma~\ref{lemma:higman acylindrical} is a direct consequence of Lemma~\ref{lemma:intersection of vertex stabilizer}(2). 
	
	\begin{lemma}
		\label{lemma:higman acylindrical}
		The action of a generalized Higman group $G$ on $X$ is weakly acylindrical.
		\qed
	\end{lemma}
	
	\begin{lemma}
		\label{lemma:weakly acylindrical}
		Suppose $H$ is a torsion-free group acting on a $\CAT(-1)$ piecewise hyperbolic complex $Z$ with finite shapes by cellular isometries. If the action is weakly acylindrical, then for any $z\in Z$ and $\xi\in \partial_\infty Z$, there exists a neighborhood $U_{\xi,z}$ of $\xi$ (with respect to the cone topology on $\partial_\infty Z$) such that any non-trivial element $g\in \Stab_H(z)$ satisfies $gU_{\xi,z}\cap U_{\xi,z}=\emptyset$.
	\end{lemma}
	
	\begin{proof}
		Let $L$ be the constant in the definition of a weakly acylindrical action. Take a point $y$ in the geodesic ray from $z$ to $\xi$ such that $d(y,z)> L$. As $H$ is torsion-free, any non-trivial element of $\Stab_H(z)$ does not fix $y$. As $Z$ has finite shapes, there exists $\epsilon>0$ such that for any non-trivial element $g\in\Stab_H(z)$, we have $d(gy,y)>\epsilon$. Now the lemma follows from the definition of the cone topology.
	\end{proof}
	
	\subsection{A finite extension of $\Hig_\sigma$}\label{sec:finite-extension}
	 
Let $F_\sigma$ be the finite group consisting of all translations $\tau$ of $\mathbb{Z}/k\mathbb{Z}$ with the property that for every $i\in\mathbb{Z}/k\mathbb{Z}$, either $m_{\tau(i)}=m_i$ and $n_{\tau(i)}=n_i$, or else $m_{\tau(i)}=-m_i$ and $n_{\tau(i)}=-n_i$. Then $F_\sigma$ naturally acts by automorphisms on $\Hig_\sigma$ (by letting $\tau$ send $a_i$ to $a_{\tau(i)}$), and we let $\widehat{\Hig}_\sigma=\Hig_\sigma\rtimes F_\sigma$.
	
	\begin{lemma}\label{lemma:icc+}
		The group $\widehat\Hig_\sigma$ is ICC.
	\end{lemma}
	
	\begin{proof}
		Let $g\in\widehat{\Hig}_\sigma$ be centralized by a finite-index subgroup $G'$ of $\widehat{\Hig}_\sigma$; we aim to prove that $g=\mathrm{id}$. Write $g=\tau h$, with $h\in\Hig_\sigma$ and $\tau\in F_\sigma$. Then for every $g'\in G'\cap\Hig_\sigma$, we have $\tau h g' h^{-1}\tau^{-1}=g'$. In other words, viewing $\tau$ as an automorphism of $\Hig_\sigma$, we have $\tau(hg'h^{-1})=g'$ for every $g'\in G'$, in particular $\tau$ sends $g'$ to a conjugate of itself. But $\tau(a_i^n)=a_{\tau(i)}^n$ for every $n\in\mathbb{Z}$, and nontrivial powers of $a_i$ and $a_{\tau(i)}$ are never conjugate unless $\tau(i)=i$ (this is a consequence of Lemma~\ref{lemma:fundamental-domain}). Therefore $\tau=\mathrm{id}$. It follows that $g$ belongs to $\Hig_\sigma$ and has an infinite conjugacy class, so $g=\mathrm{id}$ using that $\Hig_\sigma$ is ICC (Lemma~\ref{lemma:Higman property}).
	\end{proof}

	\section{Line-preserving bijections of Baumslag--Solitar groups}\label{sec:BS}
	
	This section, which is solely about the geometry of the Cayley graph of a Baumslag--Solitar group, serves as a preparatory section towards combinatorial rigidity theorems for the Higman groups that will be proved in Section~\ref{sec:combinatorial-rigidity}. Given $m,n\in\mathbb{Z}\setminus\{0\}$, the Baumslag--Solitar group $\BS(m,n)$, introduced in \cite{BS}, is defined as $$\BS(m,n)=\langle a,t \mid ta^mt^{-1}=a^n\rangle.$$ We will only consider the case where $|m|\neq |n|$, and as $\BS(m,n)$ is always isomorphic to $\BS(n,m)$ via an isomorphism sending the stable letter $t$ of the HNN extension to $t^{-1}$, it will be harmless to assume that $|m|<|n|$ throughout. 
	
	Let $\Upsilon_{m,n}$ be the Cayley graph of $\BS(m,n)$ with respect to its standard generating set $\{a,t\}$. Edges of $\Upsilon_{m,n}$ naturally come with an orientation and a label in $\{a,t\}$. A \emph{standard line} in $\Upsilon_{m,n}$ is a subset of $\Upsilon_{m,n}$ homeomorphic to the real line which is either a concatenation of edges labeled $t$ (in which case we call it a \emph{$t$-line}) or a concatenation of edges labeled $a$ (in which case we call it an \emph{$a$-line}). Vertices of $\Upsilon_{m,n}$ in a standard line have a linear order induced by the orientation of edges in $\Upsilon_{m,n}$. Note that $t$-lines (resp.\ $a$-lines) in $\Upsilon_{m,n}$ are in 1-1 correspondence with left cosets of the subgroup $\langle t\rangle$ (resp.\ $\langle a\rangle$) in $\BS(m,n)$. We identify $\BS(m,n)$ with the vertex set $\Upsilon^{(0)}_{m,n}$ of $\Upsilon_{m,n}$. 
	
	Let now $m,n,m',n'\in\mathbb{Z}\setminus\{0\}$ with $|m|<|n|$ and $|m'|<|n'|$. A bijection $\phi:\Upsilon^{(0)}_{m,n}\to \Upsilon^{(0)}_{m',n'}$ is \emph{line-preserving} if for every standard line $\ell$ of $\Upsilon_{m,n}$, there exists a standard line $\ell'$ of $\Upsilon_{m',n'}$ such that $\phi(\ell^{(0)})=(\ell')^{(0)}$. The main result of the present section is the following.
	
	\begin{prop}\label{prop:combinatorial rigid BS}
		Suppose $|m|<|n|$ and $|m'|<|n'|$. Let $\phi:\Upsilon^{(0)}_{m,n}\to \Upsilon^{(0)}_{m',n'}$ be a line-preserving bijection. 
		
		Then $\phi$ preserves labels of standard lines, and the linear order along each $t$-line. Moreover $|m|=|m'|$ and $|n|=|n'|$.
	\end{prop}
	
	\begin{rk}
		However, in general, it is not true that every line-preserving bijection $\Upsilon^{(0)}_{m,n}\to\Upsilon^{(0)}_{m',n'}$ preserves the linear order along each $a$-line. See Example~\ref{ex:12} below. 
	\end{rk}
	
	Let $\Lambda_{m,n}$ be the graph having one vertex per standard line in $\Upsilon_{m,n}$, two vertices being joined by an edge if the associated standard lines intersect nontrivially. Note that when two standard lines intersect nontrivially, then one is an $a$-line and the other is a $t$-line, and their intersection is reduced to one point: indeed, otherwise, we would have $t^k=a^{\ell}$ for nonzero integers $k,\ell$, contradicting that $t$ is loxodromic in the Bass--Serre tree of $\BS(m,n)$ while $a$ is elliptic. The \emph{type} of a vertex in $\Lambda_{m,n}$ is the label of the associated standard line. Note that $\Lambda_{m,n}$ is bipartite: every edge of $\Lambda_{m,n}$ connects a vertex of type $t$ and a vertex of type $a$. Edges of $\Lambda_ {m,n}$ are naturally in 1-1 correspondence with vertices of $\Upsilon_{m,n}$: indeed, an edge of $\Lambda_{m,n}$ corresponds to a pair $(\ell_1,\ell_2)$, where $\ell_1$ is an $a$-line and $\ell_2$ is a $t$-line, and the associated vertex of $\Upsilon_{m,n}$ is then $\ell_1\cap\ell_2$. Using this observation, one gets a 1-1 correspondence between line-preserving bijections $\Upsilon^{(0)}_{m,n}\to \Upsilon^{(0)}_{m',n'}$ and graph isomorphisms $\Lambda_{m,n}\to\Lambda_{m',n'}$.
	
	To simplify notation, in the following discussion, we write $\Lambda$ and $\Lambda'$ instead of $\Lambda_{m,n}$ and $\Lambda_{m',n'}$ respectively. Similarly we define $\Upsilon$ and $\Upsilon'$. Let $T$ and $T'$ be the Bass--Serre trees of $\BS(m,n)$ and $\BS(m',n')$, respectively. Each vertex $v\in T^{(0)}$ corresponds to an $a$-line (which is the unique $a$-line whose setwise stabilizer for the $\BS(m,n)$-action on $\Upsilon$ is equal to the stabilizer of $v$ with respect to the $\BS(m,n)$-action on $T$). Thus vertices of $T$ are in 1-1 correspondence with type $a$ vertices of $\Lambda$. There is a natural map $\pi:\Upsilon\to T$ sending each $a$-line to its associated vertex in $T$ and sending each $t$-edge to an edge in $T$. Note that if two $t$-edges of $\Upsilon$ are sent by $\pi$ to the same edge in $T$, then they induce the same orientation on this edge of $T$. Thus edges of $T$ inherit an orientation from $\Upsilon$ via $\pi$. This gives a partial order on the vertex set of $T$ (hence on the collection of type $a$ vertices of $\Lambda$), namely: for vertices $v_1$ and $v_2$ of $T$, we let $v_1\le v_2$ if each edge in the unique geodesic path connecting $v_1$ and $v_2$ is oriented from $v_1$ to $v_2$.
	
	We equip $\Lambda$ with the path metric such that each edge has length $1$.
	For a vertex $u\in\Lambda^{(0)}$, the \emph{link} $\lk_\Lambda(u)$ is the set of all vertices of $\Lambda$ at distance $1$ from $u$ (note that link defined here has a slightly different meaning compared to the definition in Section~\ref{subsec:polyhedral complex}, thus we choose a different notation).  Our first lemma shows that every graph automorphism of $\Lambda$ preserves the types of vertices. Through the correspondence between line-preserving bijections $\Upsilon^{(0)}\to (\Upsilon')^{(0)}$ and graph isomorphisms $\Lambda\to\Lambda'$, this yields the first conclusion of Proposition~\ref{prop:combinatorial rigid BS}.
	
	\begin{lemma}
		\label{lemma:link1}
		Suppose $|m|<|n|$. Let $u\in\Lambda^{(0)}$ be a vertex. The following assertions are equivalent.
		\begin{enumerate}	
			\item[(1)] The vertex $u$ is of type $a$.
			\item[(2)] There exists a finite subset $V_u\subseteq\Lambda^{(0)}$ consisting of vertices at distance $2$ from $u$, such that $\lk_\Lambda(u)\subseteq \bigcup_{v\in V_u} \lk_\Lambda(v)$.
		\end{enumerate}
	\end{lemma}
	
	\begin{proof}
		We first assume that $u$ is of type $a$, and prove that $(2)$ holds. Let $\ell_u$ be the $a$-line of $\Upsilon$ associated to $u$. Vertices in $\lk_\Lambda(u)$ correspond to $t$-lines of $\Upsilon$ which intersect $\ell_u$ in a point. Let $V_u$ be the set consisting of all type $a$ vertices $v\in\Lambda^{(0)}$ such that $u,v$ correspond to adjacent vertices $\bar u,\bar v$ in $T$, with $\bar v<\bar u$. The set $V_u$ is finite because $T$ is locally finite. For every $v\in V_u$, let $\ell_v$ be the $a$-line of $\Upsilon$ associated with $v$. Then each $t$-line $\ell$ which intersects $\ell_u$ intersects at least one of the lines $\ell_v$ with $v\in V_u$ (as can be seen by projecting $\ell$ to $T$). Thus $\lk_\Lambda(u)\subseteq \cup_{v\in V_u} \lk_\Lambda(v)$. We also have $d(u,v)=2$ for every $v\in V_u$, because there exists a $t$-line that intersects both $\ell_u$ and $\ell_v$.
		
		We now assume that $u$ is of type $t$, and prove that $(2)$ fails. Let $v\in\Lambda^{(0)}$ with $d(u,v)=2$. Using the bipartite structure of $\Lambda$, we see that $v$ is also of type $t$. Let $\ell_u$ and $\ell_v$ be $t$-lines in $\Upsilon$ corresponding to $u$ and $v$, respectively. Recall that $\pi:\Upsilon\to T$ denotes the projection map.
		
		Since $d(u,v)=2$, there exists an $a$-line $\ell_0$ intersecting both $\ell_u$ and $\ell_v$. Let $x_0=\pi(\ell_0)$, and let $d_0=d_{\ell_0}(\ell_0\cap\ell_u,\ell_0\cap\ell_v)$ (the line $\ell_0$ has a natural simplicial structure induced by $\Upsilon$, and $d_{\ell_0}$ is the simplicial distance on $\ell_0$). Let $x_v>x_0$ be any vertex such that $$d_T(x_0,x_v)>\frac{\log(d_0)}{\log(|n|/|m|)}.$$ We claim that $\pi(\ell_u)\cap\pi(\ell_v)$ does not contain any vertex $y\ge x_v$. Indeed, assume towards a contradiction that $\pi(\ell_u)\cap\pi(\ell_v)$ contains a vertex $y\ge x_v$. Let $x_0,x_1,\dots,x_k=y$ be the vertices, aligned in this order, along the segment from $x_0$ to $y$ in $T$ (in particular $k=d_T(x_0,y)$). For every $i\in\{1,\dots,k\}$, let $\ell_i$ be the $a$-line in $\Upsilon$ corresponding to $x_i$. Notice that all $x_i$ belong to the convex set $\pi(\ell_u)\cap\pi(\ell_v)$, so $\ell_i$ intersects both $\ell_u$ and $\ell_v$. Let $d_i=d_{\ell_i}(\ell_i\cap\ell_u,\ell_i\cap\ell_v)$. Then for every $i\in\{0,\dots,k-1\}$, we have $\frac{d_{i+1}}{d_i}=\frac{|m|}{|n|}$, so  $$d_k=\left(\frac{|m|}{|n|}\right)^kd_0<1,$$ contradicting that $d_k$ is a positive integer. This proves our claim. 
		
		Now, suppose towards a contradiction that a finite set $V_u\subseteq\Lambda^{(0)}$ as in $(2)$ exists. For every $v\in V_u$, let $\ell_v$ be the $t$-line of $\Upsilon$ corresponding to $v$. As $\lk_\Lambda(u)\subseteq \cup_{v\in V_u} \lk_\Lambda(v)$, each $a$-line intersecting $\ell_u$ also intersects at least one of the lines $\ell_v$ with $v\in V_u$. Thus $\pi(\ell_u)\subseteq \cup_{v\in V_u}\pi(\ell_v)$. But the claim made in the previous paragraph implies that $\cup_{v\in V_u}\pi(\ell_v)$ does not contain any vertex $x$ of $\pi(\ell_u)$ such that $x>x_0$ and $x$ is sufficiently far away from $x_0$. This contradiction concludes our proof.
	\end{proof}
	
	\begin{lemma}
		\label{lemma:link3}
		Let $u_1,u_2\in\Lambda^{(0)}$ be vertices of type $a$. If $\lk_\Lambda(u_1)\cap \lk_\Lambda(u_2)\neq\emptyset$, then either $u_1\ge u_2$ or $u_2\ge u_1$.
	\end{lemma}
	
	\begin{proof}
		Let $\ell_1$ and $\ell_2$ be the $a$-lines in $\Upsilon$ corresponding to $u_1$ and $u_2$ respectively. Since $\lk_\Lambda(u_1)\cap \lk_\Lambda(u_2)\neq\emptyset$, there exists a $t$-line $\ell$ which intersects both $\ell_1$ and $\ell_2$. The lemma follows by considering the images of $\ell,\ell_1$ and $\ell_2$ under the projection $\pi:\Upsilon\to T$.
	\end{proof}
	
	\begin{rk}
		The converse of Lemma~\ref{lemma:link3} is false in general.  For example, in $\BS(2,4)$, if $u_1$ corresponds to the $a$-line passing through identity and $u_2$ corresponds to the $a$-line passing through $tat$, then  $u_1\le u_2$ but $\lk_\Lambda(u_1)\cap \lk_{\Lambda}(u_2)=\emptyset$.
	\end{rk}
	
	\begin{lemma}\label{lemma:equally-spaced}
		Suppose $|m|<|n|$, let $h=\gcd(m,n)$, let $p=|m|/h$ and $q=|n|/h$. Let $u_1,u_2\in\Lambda^{(0)}$ be vertices of type $a$, let $\bar u_1,\bar u_2$ be the corresponding vertices in $T$, and let $\ell_1$ and $\ell_2$ be the corresponding standard lines in $\Upsilon$. 
		
		Assume that there exists a $t$-line which intersects both $\ell_1$ and $\ell_2$, and let $\ell^{12}_2\subseteq\ell_2$ be the subset consisting of all intersections of the form $\ell_2\cap\ell$, where $\ell$ varies over all $t$-lines that intersect both $\ell_1$ and $\ell_2$.
		
		Then $\ell^{12}_2$ is a subset of equally spaced vertices of $\ell_2$ whose gap (measured along $\ell_2$) is equal to $hq^{d_T(\bar u_1,\bar u_2)}$ if $u_2<u_1$, and $hp^{d_T(\bar u_1,\bar u_2)}$ if $u_1<u_2$.  
	\end{lemma}
	
	\begin{proof}
		We will assume that $u_2<u_1$. The case where $u_1<u_2$ is similar and left to the reader  (alternatively, it can be deduced from the case $u_2<u_1$ through the natural automorphism between $\BS(m,n)$ and $\BS(n,m)$, after observing that the present proof also works if $|n|<|m|$).  Let $d=d_T(\pi(\ell_1),\pi(\ell_2))$. 
		
		The existence of a $t$-line intersecting both $\ell_1$ and $\ell_2$ shows that $\ell^{12}_2\neq\emptyset$. We first observe that if $x\in\ell^{12}_2$ and $y\in\ell_2$ are such that  $d_{\ell_2}(x,y)=hq^{d}$, then $y\in\ell^{12}_2$. Indeed, say for instance that $y=x\cdot a^{hq^d}$ (the case where $x=y\cdot a^{hq^d}$ is similar). Observe that $t^da^{hp^d}t^{-d}=a^{\pm hq^d}$.  Assuming that $t^da^{hp^d}t^{-d}=a^{hq^d}$ (the case where it is equal to $a^{-hq^d}$ being similar), we deduce that 
		\begin{itemize}
			\item  $x,xt^d$ lie on a $t$-line, 
			\item $xt^d,xt^da^{hp^d}$ all lie on the $a$-line $\ell_1$, and
			\item $xt^da^{hp^d}, xt^da^{hp^d}t^{-d}=xa^{hq^d}=y$ lie on a $t$-line,
		\end{itemize}
		showing that $y\in\ell^{12}_2$, as desired.
		
		Now, there only remains to show that if $y\in\ell^{12}_2$, then $d_{\ell_2}(x,y)\ge hq^d$. Let $\ell$ and $\ell'$ be two $t$-lines intersecting both $\ell_1$ and $\ell_2$. For $i\in\{1,2\}$, let $d_i=d_{\ell_i}(\ell_i\cap\ell,\ell_i\cap\ell')$. Then $$d_2=d_1\left(\frac{|n|}{|m|}\right)^{d};\ \ \textrm{and}\ \ |m|\ \textrm{divides}\  d_1\left(\frac{|n|}{|m|}\right)^{i}\ \textrm{for}\ 0\le i\le d-1,$$
	 as illustrated in Figure~\ref{fig:bs1} below.
		\begin{center}
		\begin{figure}[h]
			\includegraphics[scale=1]{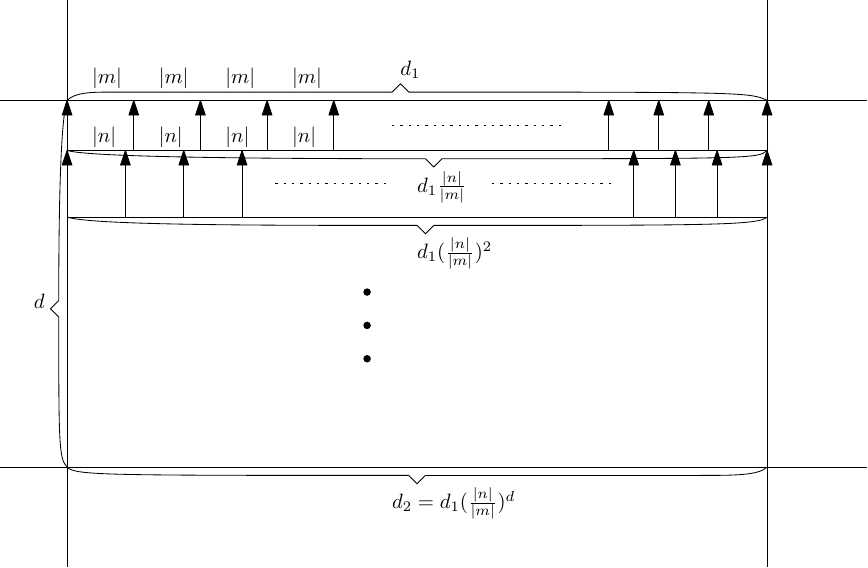}
			\caption{Illustration of the conditions in the proof of Lemma~\ref{lemma:equally-spaced}.}
			\label{fig:bs1}
		\end{figure}
		\end{center}
		From the first clause, we deduce that $d_1=p^ds$ for some positive integer $s$. Taking $i=d-1$ in the second clause, we have $ph\mid ps q^{d-1}$, hence $h\mid s q^{d-1}$. Taking $i=0$ in the second clause, we have $ph\mid s p^{d}$, hence $h\mid s p^{d-1}$. As $\gcd(p,q)=1$, we have $\gcd(s q^{d-1},s p^{d-1})=s$. Thus $h\mid s$. In particular $s\ge h$, so $d_1\ge hp^d$ and $d_2\ge hq^d$, as desired.
	\end{proof}
	
	\begin{cor}\label{cor:equally-spaced}
		Suppose $|m|<|n|$. Let $u_1,u_2,u_3\in\Lambda^{(0)}$ be  vertices of type $a$, with $u_1<u_2<u_3$. For $i,j\in\{1,2,3\}$, let $\Lambda_{ij}=\lk_\Lambda(u_i)\cap\lk_\Lambda(u_j)$. Assume that $\Lambda_{13}\neq\emptyset$. 
		
		Then $\Lambda_{13}\subsetneq\Lambda_{12}$ and $\Lambda_{13}\subseteq \Lambda_{23}$. Moreover, if $m$ does not divide $n$, then $\Lambda_{13}\subsetneq \Lambda_{23}$, otherwise $\Lambda_{13}=\Lambda_{23}$.
	\end{cor}
	
	\begin{proof}
		We adopt the same notations as in Lemma~\ref{lemma:equally-spaced} and its proof. By convexity of $T$, every $t$-line which intersects both $\ell_1$ and $\ell_3$ also intersects $\ell_2$, so $\Lambda_{13}\subseteq\Lambda_{12}$ and $\Lambda_{13}\subseteq\Lambda_{23}$. In particular $\Lambda_{12}\neq\emptyset$ and $\Lambda_{23}\neq\emptyset$.
		
		Let $\bar u_i\in T$ be the vertex corresponding to $u_i$. By Lemma~\ref{lemma:equally-spaced}, the subset $\ell^{13}_1\subseteq \ell_1$ is a subset of equally spaced vertices in $\ell_1$ with gap $hq^{d_T(\bar u_1,\bar u_3)}$. Likewise $\ell^{12}_1\subseteq \ell_1$ is a subset of equally spaced vertices in $\ell_1$ with gap $hq^{d_T(\bar u_1,\bar u_2)}$. As $d_T(\bar u_1,\bar u_3)>d_T(\bar u_1,\bar u_2)$ and $q>1$, there exists a $t$-line which intersects both $\ell_1$ and $\ell_2$ but not $\ell_3$, showing that $\Lambda_{13}\subsetneq\Lambda_{12}$.
		
		By Lemma~\ref{lemma:equally-spaced}, the subset $\ell^{13}_3\subseteq \ell_3$ is a subset of equally spaced vertices in $\ell_3$ with gap $hp^{d_T(\bar u_1,\bar u_3)}$. Likewise $\ell^{23}_3\subseteq \ell_3$ is a subset of equally spaced vertices in $\ell_3$ with gap $hp^{d_T(\bar u_2,\bar u_3)}$. If $m$ does not divide $n$, then $p>1$. Hence we deduce as before that $\Lambda_{13}\subsetneq \Lambda_{23}$. If $m\mid n$, then $p=1$. Then $hp^{d_T(\bar u_1,\bar u_3)}=hp^{d_T(\bar u_2,\bar u_3)}=h$. As $\Lambda_{13}\subseteq\Lambda_{23}$, it follows that in fact $\Lambda_{13}=\Lambda_{23}$, as desired.
	\end{proof}
	
	We will now use the above analysis to characterize adjacency along $t$-lines solely in terms of the combinatorics of $\Lambda$. We will argue separately regarding whether $m$ divides $n$ or not, and start with the case where it does not. 
	
	\begin{lemma}
		\label{lemma:link4}
		Suppose that $|m|<|n|$ and that $m$ does not divide $n$.	Let $u_1$ and $u_2$ be vertices of type $a$ in $\Lambda$ such that $u_1<u_2$ or $u_2<u_1$. Then $u_1$ and $u_2$ correspond to adjacent vertices in $T$ if and only if there does not exist any vertex $u_3\in\Lambda^{(0)}\setminus\{u_1,u_2\}$ such that $\lk_\Lambda(u_1)\cap \lk_\Lambda(u_2)\subseteq \lk_\Lambda(u_3)$.
	\end{lemma}

	\begin{proof}
		We assume without loss of generality that $u_2<u_1$. For every $i\in\{1,2\}$, let $\bar u_i$ be the vertex in $T$ corresponding to $u_i$ and let $\ell_i$ be the $a$-line in $\Upsilon$ corresponding to $u_i$. For $i,j\in\{1,2,3\}$, we write $\Lambda_{ij}=\lk_\Lambda(u_i)\cap \lk_\Lambda(u_j)$.
		
		We first assume that $\bar u_1$ and $\bar u_2$ are not adjacent, and show that $u_1,u_2$ fail to satisfy the condition from the lemma. As $\bar u_1$ and $\bar u_2$ are not adjacent, we can find $u_3\in\Lambda^{(0)}$ with $u_2<u_3<u_1$. If $\Lambda_{12}=\emptyset$, then the inclusion $\Lambda_{12}\subseteq\lk_\Lambda(u_3)$ is obvious, and otherwise this inclusion follows from Corollary~\ref{cor:equally-spaced}. 
		
		We now assume that $\bar u_1$ and $\bar u_2$ are adjacent, and prove that $u_1,u_2$ satisfy the condition from the lemma. In this case it is clear that $\Lambda_{12}\neq\emptyset$. We argue by contradiction and take $u_3$ such that $\Lambda_{12}\subseteq\lk_\Lambda(u_3)$. In particular $u_3$ is of type $a$, and for every $i\in\{1,2\}$, we have $\Lambda_{12}\subseteq \Lambda_{i3}.$
		Lemma~\ref{lemma:link3} thus ensures that $u_3$ is comparable to both $u_1$ and $u_2$. As $u_2<u_1$, and $\bar u_1,\bar u_2$ are adjacent, there are two possibilities left for the ordering of $\{u_1,u_2,u_3\}$. If $u_2<u_1<u_3$, then Corollary~\ref{cor:equally-spaced} ensures that $\Lambda_{23}\subsetneq \Lambda_{12}$, contradicting that $\Lambda_{12}\subseteq \Lambda_{23}$. If $u_3<u_2<u_1$, then as $m$ does not divide $n$, Corollary~\ref{cor:equally-spaced} ensures that $\Lambda_{13}\subsetneq\Lambda_{12}$, contradicting that $\Lambda_{12}\subseteq\Lambda_{13}$. 
	\end{proof}

	\begin{lemma}\label{lemma:link5}
		Suppose that $|m|<|n|$ and that $m$ divides $n$. Let $u_1$ and $u_2$ be vertices of type $a$ in $\Lambda$ such that $u_1<u_2$ or $u_2<u_1$. Then $u_1$ and $u_2$ correspond to adjacent vertices in $T$ if and only if $\lk_\Lambda(u_1)\cap\lk_\Lambda(u_2)\neq\emptyset$, and there do not exist any vertices $u_3,u_4\in\Lambda^{(0)}$ such that the following two assertions hold: 
		\begin{enumerate}
			\item we have $\lk_\Lambda(u_1)\cap \lk_\Lambda(u_2)\subseteq \lk_\Lambda(u_3)\cap \lk_\Lambda(u_4)$;
			\item there exists $i\in\{1,2\}$ such that $\lk_\Lambda(u_3)\cap \lk_\Lambda(u_4)=\lk_\Lambda(u_3)\cap \lk_\Lambda(u_i)$.
		\end{enumerate}
	\end{lemma}
	
	\begin{proof}
		We assume without loss of generality that $u_2<u_1$. As a matter of convention, we will always denote by $\bar u_i$ the vertex in $T$ corresponding to the vertex $u_i\in\Lambda^{(0)}$, and by $\ell_i$ the $a$-line in $\Upsilon$ corresponding to $u_i$. We will always write $\Lambda_{ij}=\lk_\Lambda(u_i)\cap \lk_\Lambda(u_j)$. 
		
		We first assume that $\bar u_1$ and $\bar u_2$ are not adjacent, and show that $u_1,u_2$ fail to satisfy the condition from the lemma. If $\Lambda_{12}=\emptyset$, we are done, so we will assume otherwise. As $\bar u_1$ and $\bar u_2$ are not adjacent, we can find $u_3\in\Lambda^{(0)}$ with $u_2<u_3<u_1$. Corollary~\ref{cor:equally-spaced} then ensures that $\Lambda_{12}\subset\Lambda_{23}$ (in particular $\Lambda_{12}\subseteq\lk_\Lambda(u_3)$).  
		Let now $u_4$ be such that $u_4<u_2$, chosen so that $\Lambda_{14}\neq\emptyset$ and $\Lambda_{34}\neq\emptyset$: this exists, by choosing a $t$-line $\ell$ that intersects $\ell_1,\ell_2,\ell_3$, and choosing $u_4<u_2$ corresponding to an $a$-line $\ell_4$ that also intersects $\ell$. As $m$ divides $n$, Corollary~\ref{cor:equally-spaced} applied to $u_4<u_2<u_1$ ensures that $\Lambda_{14}=\Lambda_{12}$ (in particular $\Lambda_{12}\subseteq\lk_\Lambda(u_4)$). 
		Finally, as $m$ divides $n$, Corollary~\ref{cor:equally-spaced} applied to $u_4<u_2<u_3$ ensures that $\Lambda_{34}=\Lambda_{23}$. We have thus found $u_3,u_4$ as in the statement of the lemma in this case.
		
		We now assume that $\bar u_1$ and $\bar u_2$ are adjacent, and prove that $u_1,u_2$ satisfy the condition from the lemma. In this case it is clear that $\Lambda_{12}\neq\emptyset$. We argue by contradiction and take $u_3,u_4$ as in the lemma. As $\Lambda_{12}\subseteq\lk_\Lambda(u_3)$, the vertex $u_3$ is of type $a$, and for every $i\in\{1,2\}$, we have $\Lambda_{12}\subseteq \Lambda_{i3}.$
		Lemma~\ref{lemma:link3} thus ensures that $u_3$ is comparable to both $u_1$ and $u_2$. As $u_2<u_1$, and $\bar u_1,\bar u_2$ are adjacent, there are two possibilities left for the ordering of $\{u_1,u_2,u_3\}$. If $u_2<u_1<u_3$, then Corollary~\ref{cor:equally-spaced} ensures that $\Lambda_{23}\subsetneq \Lambda_{12}$, contradicting that $\Lambda_{12}\subseteq \Lambda_{23}$. 
		So we can assume that $u_3<u_2<u_1$. Likewise $u_2<u_1<u_4$ leads to $\Lambda_{14}=\Lambda_{24}\subsetneq \Lambda_{12}$, a contradiction to the equality $\Lambda_{14}=\Lambda_{12}$ established in the previous paragraph. Thus we can assume $u_4<u_2<u_1$.
		
		As $\Lambda_{12}\neq\emptyset$ and $\Lambda_{12}\subseteq\Lambda_{34}$, we have $\Lambda_{34}\neq\emptyset$. So Lemma~\ref{lemma:link3} ensures that $u_3$ and $u_4$ are comparable. 
		
		Assume first that $u_4<u_3<u_2<u_1$. The first clause in the statement of the lemma ensures that $\Lambda_{14}\neq\emptyset$ and that $\Lambda_{24}\neq\emptyset$. Then Corollary~\ref{cor:equally-spaced}, applied to $u_4<u_3<u_2$, ensures that $\Lambda_{23}=\Lambda_{24}$ (using that $m$ divides $n$), which in turn is strictly contained in $\Lambda_{34}$. And when applied to $u_4<u_3<u_1$ it ensures that $\Lambda_{13}=\Lambda_{14}$ (using that $m$ divides $n$), which in turn is strictly contained in $\Lambda_{34}$. In particular $\Lambda_{13}$ and $\Lambda_{23}$ are strictly contained in $\Lambda_{34}$, and we get a contradiction to the last clause in the statement of the lemma. 
		
		Assume now that $u_3<u_4<u_2<u_1$. The first clause in the statement of the lemma ensures that $\Lambda_{13}\neq\emptyset$ and that $\Lambda_{23}\neq\emptyset$. Then Corollary~\ref{cor:equally-spaced}, applied to $u_3<u_4<u_2$, ensures that $\Lambda_{23}\subsetneq\Lambda_{34}$, and when applied to $u_3<u_4<u_1$ it ensures that $\Lambda_{13}\subsetneq\Lambda_{34}$. Again we get a contradiction to the last clause in the statement of the lemma. This completes our proof.
	\end{proof}
	
	Now we are ready to prove Proposition~\ref{prop:combinatorial rigid BS}.
	
	\begin{proof}[Proof of Proposition~\ref{prop:combinatorial rigid BS}]
		Recall that the line-preserving bijection $\phi:\Upsilon^{(0)}\to (\Upsilon')^{(0)}$ induces a graph isomorphism $\phi_\Lambda:\Lambda\to\Lambda'$. By Lemma~\ref{lemma:link1}, $\phi$ preserves labels of standard lines, and therefore it also induces a bijection $\phi_T: T^{(0)}\to (T')^{(0)}$ (recall indeed that type $a$ vertices of $\Lambda$ are identified with vertices of $T$). Take two adjacent vertices $\bar u_1,\bar u_2$ of $T$, and let $u_1,u_2$ be the corresponding vertices of $\Lambda$. Assume for now that $m$ does not divide $n$. Then $\lk_\Lambda(u_1)\cap \lk_\Lambda(u_2)\neq\emptyset$ and $u_1,u_2$ satisfy the condition from Lemma~\ref{lemma:link4}. Hence $\lk_{\Lambda'}(\phi_\Lambda(u_1))\cap \lk_{\Lambda'}(\phi_\Lambda(u_2))\neq\emptyset$. Lemma~\ref{lemma:link3} implies that $\phi_T(\bar u_1)$ and $\phi_T(\bar u_2)$ are comparable. As $u_1$ and $u_2$ satisfy the condition of Lemma~\ref{lemma:link4}, so do $\phi_\Lambda(u_1)$ and $\phi_\Lambda(u_2)$. Lemma~\ref{lemma:link4} therefore implies that $\phi_T(\bar u_1)$ and $\phi_T(\bar u_2)$ are adjacent in $T'$. The same conclusion also holds when $m$ divides $n$, using Lemma~\ref{lemma:link5} instead of Lemma~\ref{lemma:link4}. Thus $\phi_T$ preserves adjacency, and as $T$ and $T'$ are trees, it also preserves non-adjacency, i.e.\ $\phi_T$ extends to an isomorphism from $T$ to $T'$. As $T$ has valence $|m|+|n|$ and $T'$ has valence $|m'|+|n'|$, we deduce that $|m|+|n|=|m'|+|n'|$.
		
		We claim that $\phi_T$ preserves the orientation of edges. Say that two vertices $\bar v_1,\bar v_2\in T^{(0)}$ are \emph{strongly comparable} if $\lk_\Lambda(v_1)\cap\lk_\Lambda(v_2)\neq\emptyset$. Let $\bar u_1>\bar u_2$ be two adjacent vertices of $T$, and for every $i\in\{1,2\}$, let $V_i$ be the set consisting of all vertices of $T\setminus\{\bar u_1,\bar u_2\}$ which are adjacent to $\bar u_i$ and are strongly comparable to both $\bar u_1$ and $\bar u_2$. 
		
		Write $|m|/|n|=p/q$, where $p,q\ge 1$ satisfy $\gcd(p,q)=1$. We now show that $|V_1|=q$ and $|V_2|=p$. We refer to Figure~\ref{fig:bs2} below, where  $\ell_j$ is the $a$-line corresponding to $u_j$ for $j\in\{1,2\}$, and the lines $\ell_3^i$ correspond to vertices in $V_2$.
		
		Let $v$ be a vertex in $\ell_2$ which belongs to some $t$-line that intersects $\ell_1$. For every $i\in\{0,\dots,|m|-1\}$, let $v_i$ be the vertex at distance $i$ from $v$ on $\ell_2$, going positively in the $a$-direction (in particular $v_0=v$). Then there is a natural bijection between the set of all $a$-lines $\ell_3$ corresponding to vertices $\bar u_3<\bar u_2$ of $T$ that are adjacent to $\bar u_2$, and the set $\{v_0,\dots,v_{|m|-1}\}$: this bijection sends an $a$-line $\ell_3$ to the unique vertex $v_i$ such that there is a $t$-line through $v_i$ intersecting $\ell_3$. We denote by $\ell_3^i$ the $a$-line corresponding to $v_i$. Then the existence of a $t$-line intersecting both $\ell_3^i$ and  $\ell_1$ is equivalent to the existence of a pair of integers $k_1,k_2$ such that $k_1|n|-k_2|m|=i$. In other words, the set $V_2$ is in 1-1 correspondence with the set of all integers $i\in\{0,\dots,|m|-1\}$ such that the equation  $k_1|n|-k_2 |m|=i$ has an integer solution for $k_1$ and $k_2$. This happens if and only if $i$ is a multiple of $\gcd(|m|,|n|)$, so $|V_2|=p$.   
		
		Likewise $V_1$ is in 1-1 correspondence with the set of all integers $i\in\{0,\dots,|n|-1\}$ such that the equation  $k_1|m|-k_2|n|=i$ has an integer solution for $k_1$ and $k_2$, so $|V_1|=q$.

		\begin{figure}[htb]
			\begin{center}
			\includegraphics[scale=1]{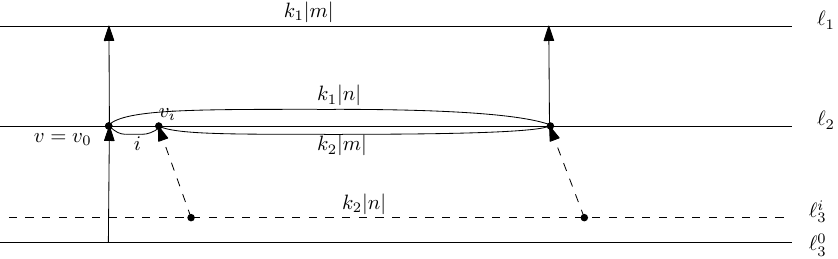}
			\caption{Illustration of the configuration in the proof of Proposition~\ref{prop:combinatorial rigid BS}.}
			\label{fig:bs2}
				\end{center}
			\end{figure}

		It follows from the definition that $\phi_T(V_i)$ is the set of all vertices of $T'$ which are adjacent to $\phi_T(\bar u_i)$ and are strongly comparable to both $\phi_T(\bar u_1)$ and $\phi_T(\bar u_2)$. We claim that $\phi_T(\bar u_1)>\phi_T(\bar u_2)$. Indeed, suppose towards a contradiction that  $\phi_T(\bar u_1)<\phi_T(\bar u_2)$. Then repeating the above computation, we have $|\phi_T(V_1)|=p'$ and $|\phi_T(V_2)|=q'$ with $p',q'$ being positive integers such that $\gcd(p',q')=1$ and $\frac{|m'|}{|n'|}=\frac{p'}{q'}$. As $|\phi_T(V_i)|=|V_i|$ for any $i\in\{1,2\}$, we deduce that $q=p'$ and $p=q'$, which is impossible as $p<q$ and $p'<q'$. Thus we must have $\phi_T(\bar u_1)>\phi_T(\bar u_2)$. Moreover, $p=p'$ and $q=q'$. This together with $|m|+|n|=|m'|+|n'|$ implies that $|m|=|m'|$ and $|n|=|n'|$. Hence Proposition~\ref{prop:combinatorial rigid BS} follows.
	\end{proof}

	\paragraph*{Some remarks, examples and questions.} Proposition~\ref{prop:combinatorial rigid BS} suggests the following more general question: given $m,n\in\mathbb{Z}$, what is the group of all line-preserving bijections of $\BS(m,n)$? Is it significantly larger than $\BS(m,n)$ itself? Let us now illustrate a possible behaviour on an example.
	
	\begin{ex}\label{ex:12}
		In the case of $\BS(1,2)$, let us now exhibit examples of line-preserving bijections of $\BS(1,2)$ which do not preserve the linear order along $a$-lines. Let $\Upsilon=\Upsilon_{1,2}$ be the Cayley graph of $\BS(1,2)$ for its usual generating set. 
		
		Let $v_0$ be the vertex of $\Upsilon$ corresponding to the identity element $e$ of $\BS(1,2)$. Let $\ell_0$ be the $t$-line through $e$, and let $\call$ be the set of all $a$-lines that intersect $\ell_0$. Let $P\subseteq\Upsilon$ be the subgraph of $\Upsilon$ spanned by all vertices that belong to a line in $\call$. Observe that every $t$-line in $\Upsilon_{1,2}$ has at least one half-line that is contained in $P$.
		
		Fix $n\ge 0$, and let $\call_{\ge n}$ be the subset of $\call$ consisting of all $a$-lines that contain $t^k$ for some $k\ge n$.
		
		We now define a map $\phi:\Upsilon^{(0)}\to\Upsilon^{(0)}$ in the following way. Say that a $t$-line is \emph{high} if it intersects a line in $\call_{\ge n}$. Say that a vertex is \emph{high} if it belongs to some high $t$-line, and \emph{low} otherwise. Observe that a $t$-line is high if and only if it is of the form $a^{k2^{n}}\ell_0$ for some $k\in\mathbb{N}$. Given a vertex $x\in\Upsilon^{(0)}$, we set $\phi(x)=x$ if $x$ is high, and $\phi(x)=a^{2^{n}}x$ if $x$ is low.
		
		Observe that $\phi$ is a bijection of $\Upsilon^{(0)}$: indeed it coincides with the identity on one orbit of $t$-lines under $\langle a^{2^{n}}\rangle$, and with the left multiplication by  $a^{2^{n}}$ on all other orbits. The map $\phi$ clearly sends every $t$-line onto a $t$-line, and we will now prove that $\phi$ sends every $a$-line onto an $a$-line. It is clear that if $x,y$ are either two high vertices on a common $a$-line, or two low vertices on a common $a$-line, then their images also belong to a common $a$-line. Let finally $x$ be a high vertex, and $y$ a low vertex, that belong to the same $a$-line. Necessarily $x$ and $y$ belong to $P$, and there exists $k<n$ such that $x$ and $y$ both belong to the $a$-line corresponding to the coset $t^k\langle a\rangle$. Write $x=t^ka^p$ and $y=t^ka^q$, with $p,q\in\mathbb{Z}$. Then $\phi(x)=t^ka^p$, while  $\phi(y)=a^{2^{n}}t^ka^q$. Using that $k<n$, we see that  $\phi(y)=t^ka^{2^{n-k}}a^q$, showing that $\phi(x)$ and $\phi(y)$ both belong to the $a$-line corresponding to the coset $t^k\langle a\rangle$. This shows that $\phi$ sends every $a$-line into an $a$-line, and arguing similarly with $\phi^{-1}$ ensures that $\phi$ actually sends bijectively $a$-line onto $a$-line. And $\phi$ does not preserve the linear order along $a$-lines.
	\end{ex}

	\section{Combinatorial rigidity of the intersection graph}\label{sec:combinatorial-rigidity}
	
	The intersection graph of the Higman group was introduced by Martin \cite[Definition~2.6]{Mar}. It can be viewed as a Higman group analogue of the curve graph for surfaces, of Crisp's fixed set graph in the context of $2$-dimensional Artin groups of hyperbolic type \cite{Cri}, or of Kim and Koberda's extension graph for right-angled Artin groups \cite{KK}. For later use, we give a definition in a more general context than generalized Higman groups.
	
	\begin{de}[Intersection graph]\label{de:intersection-graph}
		Let $G$ be a group, and let $X$ be an $M_\kappa$-polyhedral complex (for some $\kappa\in\mathbb{R}$) equipped with an action of $G$ by cellular isometries. The \emph{intersection graph} $\Theta_{G\actson X}$ of the $G$-action on $X$ is the graph with the same vertex set as $X$, in which two vertices are joined by an edge if their stabilizers have nontrivial intersection. 
		
		The \emph{intersection graph} of a generalized Higman group is the intersection graph of its action on the polyhedral complex $X$ defined in Section~\ref{sec:metric} (its developed complex). 
	\end{de}
	
	The graph $\Theta_{G\actson X}$ comes with a natural $G$-action induced by the $G$-action on $X$. Let now specify this to the case of a generalized Higman group $\Hig_\sigma$, acting on its developed complex $X_\sigma$. Recall the finite subgroup $F=F_\sigma$ of $\mathfrak{S}(\mathbb{Z}/k\mathbb{Z})$ introduced in Section~\ref{sec:finite-extension}. Notice that $F_\sigma$ acts on $X_\sigma$, by sending a vertex $gG_{\bar x_i}$ to $\tau(gG_{\bar x_i})=\tau(g)G_{\bar x_{\tau(i)}}$, sending an edge $gG_{\bar e_i}$ to $\tau(gG_{\bar e_i})=\tau(g)G_{\bar e_{\tau(i)}}$, and a $2$-cell $g\hat{K}$ to the $2$-cell $\tau(g)\hat{K}$. Through this, the action of $\Hig_\sigma$ on $X_\sigma$ extends to an action of $\widehat{\Hig}_\sigma=\Hig_\sigma\rtimes F_\sigma$. This also yields an action of $\widehat{\Hig}_\sigma$ on the intersection graph $\Theta_\sigma$ of $\Hig_\sigma$. We denote by $\Aut(X_\sigma)$ the group of all cellular automorphisms of $X_\sigma$, and by $\Aut(\Theta_\sigma)$ the group of all graph automorphisms of $\Theta_\sigma$.
	
	\begin{theo}
		\label{theo:crigidity0}
		Let $\Hig_\sigma$ be a generalized Higman group, let $X_\sigma$ be its developed complex, and let $\Theta_\sigma$ be its intersection graph. 
		
		Then the natural maps $\widehat{\Hig}_\sigma\to\Aut(X_\sigma)$ and $\widehat{\Hig}_\sigma\to \Aut(\Theta_\sigma)$ are isomorphisms.
	\end{theo}
	
	Recall from Section~\ref{subsec:Higman background} that $\calp$ denotes the poset of all cosets of the vertex groups, edge groups and the trivial subgroup of $G$. Recall also that $\calq$ denotes the poset of closed cells of $X$, ordered by containment. Then the natural map $\calp\to\calq$ is order-reversing. We start by proving the combinatorial rigidity statement for $\Aut(X_\sigma)$.
	
	\begin{prop}
		\label{prop:combinatorial rigid X}
		Let $\Hig_\sigma$ be a generalized Higman group, and let $X_\sigma$ be its developed complex. 
		
		Then the natural map $\widehat{\Hig}_\sigma\to \Aut(X_\sigma)$ is an isomorphism.
	\end{prop}
	
	\begin{proof}
	For simplicity of notation, we will write $G=\Hig_\sigma$, $\hat{G}=\widehat{\Hig}_\sigma$, and simply write $X,\Theta$ instead of $X_\sigma,\Theta_\sigma$.
	
	Injectivity of the map $\hat{G}\to\Aut(X)$ follows directly from its construction, so we focus on proving its surjectivity.
		
		Let $\Phi:X\to X$ be an automorphism of $X$. As $G$ acts transitively on the collection of 2-cells of $X$, by post-composing $\Phi$ by an element of $G$, we assume that $\Phi$ fixes the fundamental domain $\hat K$ (cf. Section~\ref{subsec:Higman background}) setwise. Through the identification $\mathcal P\to \mathcal Q$, $\Phi$ induces a bijection $\phi:G\to G$ which sends elements in a left coset of a vertex group (resp.\ edge group) bijectively to elements in a (possibly different) left coset of a vertex group (resp.\ edge group). 
		
		Let $\Upsilon_G$ be the Cayley graph of $G$ (with respect to the generating set $\{a_1,\dots,a_k\}$) with the usual edge labeling and orientation of Cayley graphs. For $i\in\mathbb{Z}/k\mathbb{Z}$, a \emph{standard BS-subgraph} of type $\langle a_i,a_{i+1}\rangle$ in $\Upsilon_G$ is the subgraph spanned by all vertices in a left coset of $\langle a_i,a_{i+1}\rangle$ (this subgraph is isomorphic to the Cayley graph of a Baumslag-Solitar group $\BS(m_i,n_i)$ for its standard generating set). A \emph{standard line} of type $a_i$ in $\Upsilon_G$ is a line in $\Upsilon_G$ spanned by all vertices in a left coset of $\langle a_i\rangle$. By the observation made in the previous paragraph, $\phi$ sends every standard BS-subgraph onto a standard BS-subgraph, and every standard line onto a standard line. We claim that $\phi$ preserves the linear order induced by the orientation of edges along each standard line. Indeed, if $\ell$ is a standard line of type $a_i$, then there is a standard BS-subgraph $\Upsilon_1\subset\Upsilon_G$ of type $\langle a_i,a_{i+1}\rangle$ such that $\ell\subset\Upsilon_1$. As $\phi_{|(\Upsilon_1)^{(0)}}$ is a line-preserving bijection between two Baumslag-Solitar groups, Proposition~\ref{prop:combinatorial rigid BS} implies that $\phi$ respects the linear order along $\ell$.
		
		It follows from the above claim that if $\Phi$ fixes a vertex $x$ of $X$ and 2-cell $K'$ of $X$ containing $x$, then $\Phi$ fixes the closed star of $x$ pointwise. Indeed, $x$ corresponds to a standard BS-subgraph $\Upsilon_x$ in $\Upsilon_G$. Moreover, $\phi_{|(\Upsilon_x)^{(0)}}$ is a line-preserving bijection to itself  which respects the linear order along each line (by the above) and label of each standard line (by Proposition~\ref{prop:combinatorial rigid BS}), and fixes one element (corresponding to $K'$), hence is the identity map. Thus $\Phi$ fixes the closed star of $x$ pointwise. We can deduce from this that if $\Phi$ fixes a 2-cell $K_1$ of $X$ pointwise, then $\Phi$ fixes $X$ pointwise, as for any other 2-cell $K_2$, there is a chain of 2-cells from $K_1$ to $K_2$ such that any two adjacent members in the chain have non-empty intersection (indeed $X$ is connected by Lemma~\ref{lem:sc}).
		
		Let $f$ be a permutation of elements of $\mathbb Z/k\mathbb Z$ such that $\Phi(\bar x_i)=\bar x_{f(i)}$ for each vertex $\bar x_i$ of $K$ (viewed as a vertex of $\hat K$). We now show that $f$ belongs to $F$. Note that $\phi_{|G_{\bar x_i}}:G_{\bar x_i}\to G_{\bar x_{f(i)}}$ is a line-preserving bijection which sends identity to identity. Thus by Proposition~\ref{prop:combinatorial rigid BS}, $\phi_{|G_{\bar x_i}}$ sends $a_i$-lines (resp.\ $a_{i-1}$-lines) to $a_{f(i)}$-lines (resp.\ $a_{f(i)-1}$-lines). In particular $f(i)-1=f(i-1)$, i.e.\ $f$ is a translation. As $\phi_{|G_{\bar x_i}}$ respects the linear order along each standard line by the previous claim, we deduce that $\phi_{|G_{\bar x_i}}$ is an isomorphism. Therefore, either $m_i=m_{f(i)}$ and $n_i=n_{f(i)}$, or $m_i=-m_{f(i)}$ and $n_i=-n_{f(i)}$ (see \cite{Mol}). Thus $f\in F$. By post-composing $\Phi$ by $f^{-1}\in F$, we can assume $\Phi$ fixes $\hat K$ pointwise. Then $\Phi$ is  the identity by the previous paragraph. Thus the map in Proposition~\ref{prop:combinatorial rigid X} is surjective, which completes our proof.  
	\end{proof}
	
	As before, we will now let $G=\Hig_\sigma$, $\hat{G}=\widehat{\Hig}_\sigma$, $X=X_\sigma$ and $\Theta=\Theta_\sigma$. Before proving the rigidity statement for $\Aut(\Theta)$, we now quickly review disk diagrams and the Gauss-Bonnet formula for application in the proof of the next lemma. We will only need the Gauss-Bonnet formula in the realm of $M_{\kappa}$-polyhedral complexes, though it is more natural to set it up in the more general context of combinatorial CW complexes. We refer to \cite[Definition 2.1]{MW} for the definition of combinatorial CW complexes and combinatorial maps between them. We recall from \cite[Definition~2.6]{MW} that a \emph{(singular) disk diagram} $D$ is a finite contractible 2-dimensional combinatorial CW complex with a fixed embedding in the plane $\mathbb R^2$. A \emph{boundary cycle} of $D$ is a combinatorial map from a polygon $P$ to $D$ whose image is an edge-path in the graph $D^{(1)}$ corresponding to going around $D$ once in the clockwise direction along the boundary of the unbounded complementary region $\mathbb R^2\setminus D$ (see also \cite[p.~150]{lyndon2015combinatorial}).
	
	Let $P \to X$ be a closed null-homotopic edge path in a 2-dimensional combinatorial CW complex $X$. A \emph{singular disk diagram in $X$ for $P$}
	is a singular disk diagram $D$ together with a map $D\to X$ such that the closed path
	$P\to X$ factors as $P\to D\to X$ where $P\to D$ is a boundary cycle of $D$.
	It is a theorem of Van Kampen that every null-homotopic closed edge path $P\to X$ is the boundary cycle of a singular disk diagram $D\to X$; moreover, we can assume this singular disk diagram is \emph{reduced}, i.e. $D-D^{(0)}\to X$ is an immersion, see \cite[V.2.1]{lyndon2015combinatorial} or \cite[Lemma~2.17]{MW}. We caution the reader that $D$ is usually not homeomorphic to a 2-dimensional disk (thus the name ``singular'' disk diagram), e.g.\ it is not if $P\to X$ is not an embedding. 
	Also even if $P\to X$ is an embedding, there might not exist a singular disk diagram for $P$ such that $D\to X$ is an embedding. If $X$ has a piecewise Riemannian metric, then we equip the singular disk diagram $D$ with the natural piecewise Riemannian metric induced by $D\to X$.
	
	We will use the following version of the Gauss-Bonnet formula for a singular disk diagram $D$ which is a special case of the Gauss-Bonnet formula in \cite[Section 2]{ballmann1996nonpositively}. We assume that $D$ has an $M_\kappa$-polyhedral complex structure. For a vertex $v\in D^{(0)}$, let $\chi(v)$ be the Euler characteristic of $\lk(v,D)$. Recall that the length of an edge of $\lk(v,D)$ is the interior angle at $v$ of the 2-cell of $D$ corresponding to this edge. Let $\alpha(v)$ be the sum of the lengths of all edges in $\lk(v,D)$. Define $\kappa(v)=(2-\chi(v))\pi-\alpha(v)$. Given a $2$-cell $C\in D^{(2)}$, we denote by $\mathrm{Area}(C)$ the area of $C$. Then 
	\begin{equation}
	\label{eq:GB}
	\sum_{v\in D^{(0)}}\kappa(v)+\kappa\sum_{C\in D^{(2)}}\mathrm{Area}(C)=2\pi.
	\end{equation}
	
	Given $n\in\mathbb{N}$, an \emph{induced $n$-cycle} in a simplicial graph $\Theta$ is a reduced closed edge path $\omega$ of length $n$ such that two vertices of $\omega$ are joined by an edge in $\Theta$ if and only if they are adjacent in $\omega$. Let $k$ be the number of generators in the standard presentation of the generalized Higman group $G$. Now we show that each induced $k$-cycle of the intersection graph $\Theta$ corresponds to a 2-cell of $X$ in the following way.	
	
	\begin{lemma}
		\label{lemma:cycle}
		Let $\omega$ be an induced $k$-cycle in $\Theta$ with consecutive vertices of $\omega$ denoted by $\{v_i\}_{i\in\mathbb Z/k\mathbb Z}$. Let $x_i$ be the vertex of $X$ corresponding to $v_i$. Then $\{x_i\}_{i\in\mathbb Z/k\mathbb Z}$ is the collection of consecutive vertices in the boundary of a 2-cell of $X$.
	\end{lemma}
	
	\begin{proof}
		Let $r_i$ be the geodesic segment (in $X$) between $x_i$ and $x_{i+1}$. As $v_i$ and $v_{i+1}$ are adjacent in $\Theta$, the intersection $G_{x_i}\cap G_{x_{i+1}}$ is nontrivial. It thus follows from Lemma~\ref{lemma:intersection of vertex stabilizer}(2) that $d(x_i,x_{i+1})\le 2$ and $r_i$ is contained in the 1-skeleton of $X$ (see Remark~\ref{rk:X}). 
		
		We claim that $\angle_{x_i}(x_{i-1},x_{i+1})\neq 0$. This is clear if $x_i$ is adjacent to both $x_{i-1}$ and $x_{i+1}$, as $x_{i-1}\neq x_{i+1}$. If $x_i$ is adjacent to exactly one of $x_{i-1}$ and $x_{i+1}$, say $x_{i-1}$, then $x_{i-1}\in r_i$ if $\angle_{x_i}(x_{i-1},x_{i+1})=0$. As $G_{x_i}\cap G_{x_{i+1}}$ fixes $r_i$ pointwise, $G_{x_i}\cap G_{x_{i+1}}\cap G_{x_{i-1}}$ is infinite, contradicting that $\omega$ is an induced cycle. If $x_i$ is adjacent to neither $x_{i-1}$ nor $x_{i+1}$, then the midpoints of $r_{i-1}$ and $r_i$ coincide if $\angle_{x_i}(x_{i-1},x_{i+1})=0$. We denote the common midpoint by $x$. Lemma~\ref{lemma:intersection of vertex stabilizer}(2) implies $\overline{xx_j}$ is oriented away from $x$ for $i-1\le j\le i+1$, so $G_{x_{i+1}}\cap G_{x_{i-1}}$ is nontrivial, which contradicts that $\omega$ is an induced cycle. 
		
		The concatenation of the geodesic segments $r_i$ forms an edge loop in $X$, which we view as a map $f_0$ from a suitable polygon $P$ to $X^{(1)}$ sending edges to edges. By the previous paragraph, $f_0:P\to X$ is a local embedding. When $k=4$, as $\angle_{x_i}(x_{i-1},x_{i+1})\ge\pi/2$, by the Flat Quadrilateral Theorem (cf.\ \cite[Theorem II.2.11]{BH}), the map $P\to X$ is an embedding and its image must be the boundary of a convex region $R$ in $X$ which is isometric to a Euclidean rectangle. As the sides of $R$ have length at most $2$, it follows that $R$ is made of one 2-cell of $X$, or two 2-cells, or four 2-cells. As the orientation of sides of $R$ with length $2$ has to satisfy Lemma~\ref{lemma:intersection of vertex stabilizer}(2b), moreover, the boundary of each 2-cell has a consistent edge orientation, the only possibility of $R$ satisfying these two constraints is $R$ being a single 2-cell of $X$ (see Figure~\ref{fig:rectangle}), hence the lemma follows.
		\begin{figure}
			\centering
			\includegraphics[scale=1]{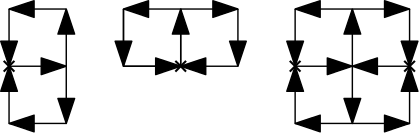}
			\caption{The vertices with a cross indicate a contradiction with Lemma~\ref{lemma:intersection of vertex stabilizer} (2).}
			\label{fig:rectangle}
		\end{figure}
		
		Suppose now that $k\ge 5$. Let $f:D\to X$ be a reduced singular disk diagram for $f_0:P\to X$ (i.e.\ $P\to X$ factors through $P\to D\to X$). Then $D$ inherits an $M_{-1}$-polyhedral complex structure from $X$ and the above Gauss-Bonnet formula (Equation~\eqref{eq:GB}) becomes \[\sum_{v\in D^{(0)}}\kappa(v)-\sum_{C\in D^{(2)}}\mathrm{Area}(C)=2\pi.\] 
		We denote the map $P\to D$ by $f_1$. Define $\partial D = \textrm{im}\ f_1$. The polygon $P$ is a concatenation of subpaths $r'_i$ which are mapped to the geodesic segments $r_i$ under $f_0$. We let $x'_i\in P$ be the intersection of $r'_{i-1}$ and $r'_i$ (with indices taken modulo $k$), and note that $f_0(x'_i)=x_i$.
		
		We now make the following observations.
		\begin{enumerate}
			\item If $v\in D^{(0)}\setminus \partial D$, then $\chi(v)=0$ (as $\lk(v,D)$ is a circle). Moreover, as the disk diagram is reduced, $f:D\to X$ induces a local embedding $\lk(v,D)\to \lk(f(v),X)$. As the length of any embedded cycle in $\lk(f(v),X)$ is at least $2\pi$  (see Lemma~\ref{lemma:link condition}), we have $\alpha(v)\ge 2\pi$, so $\kappa(v)\le 0$.
			\item Suppose $v\in D^{(0)}$ satisfies that $v\in \partial D$ and $v$ is one of the $f_1(x'_i)$. Then either $\chi(v)=2$ (i.e. $\lk(v,D)$ is disconnected) or $\chi(v)=1$ (i.e. $\lk(v,D)$ is connected). In the former case $\kappa(v)\le 0$; in the latter case, as $\angle_{x_i}(x_{i-1},x_{i+1})\ge\pi/2$ in $X$, $f$ induces a map from the path $\lk(v,D)$ to a path in $\lk(x_i,X)$ joining two vertices at distance at least $\pi/2$, implying $\alpha(v)\ge\pi/2$ and $\kappa(v)\le \pi/2$. Thus we always have $\kappa(v)\le\pi/2$.
			\item Suppose $v\in D^{(0)}$ satisfies that $v\in \partial D$ and $v$ is not one of the $f_1(x'_i)$. If $\lk(v,D)$ is disconnected, then we deduce as before that $\kappa(v)\le 0$. If $\lk(v,D)$ is connected, as $f(v)$ lies in the interior of one of the geodesic segment $r_i$, $f$ induces a map from the path $\lk(v,D)$ to a path in $\lk(f(v),X)$ joining two vertices at distance at least $\pi$, implying $\alpha(v)\ge\pi$ and $\kappa(v)\le 0$. Thus we always have $\kappa(v)\le 0$.
		\end{enumerate}
		Applying the Gauss-Bonnet formula to a single $2$-cell of $X$ shows that the area of any $2$-cell is equal to $\frac{k\pi}{2}-2\pi$. Now, when applying the Gauss-Bonnet formula to $D$, the above analysis shows that $\sum_{v\in D^{(0)}}\kappa(v)\le \frac{k\pi}{2}$, and from this we deduce that $D$ has at most one $2$-cell. The case where $D$ has no 2-cells can be ruled out using that $\angle_{x_i}(x_{i-1},x_{i+1})\ge\pi/2$. We deduce that $P$ is the boundary of a single 2-cell in $X$ (in particular $x_i$ and $x_{i+1}$ are adjacent in $X^{(1)}$) and the lemma follows.			
	\end{proof}

	\begin{rk}
		We record a geometric observation which follows from a similar use of Gauss-Bonnet formula as above. Let $Z$ be a  2-dimensional $\CAT(-1)$ piecewise hyperbolic complex with finite shapes such that every 2-cell of $Z$ is a right-angled hyperbolic polygon. Suppose $P$ is an immersed (i.e.\ locally embedded) geodesic polygon in $Z^{(1)}$ which is a concatenation of $k$ geodesic segments of $Z$, where $k$ is the minimal number of sides of  a 2-cell in $Z$. Then $P$ is embedded and $P$ is the boundary of a 2-cell of $Z$.
	\end{rk}
	
	\begin{proof}[Proof of Theorem~\ref{theo:crigidity0}]
As before, we will let $G=\Hig_\sigma$, $\hat{G}=\widehat{\Hig}_\sigma$, $X=X_\sigma$ and $\Theta=\Theta_\sigma$. The statement for $\Aut(X)$ was already proved in Proposition~\ref{prop:combinatorial rigid X}, so we focus on the statement for $\Aut(\Theta)$. Take $\alpha\in \Aut(\Theta)$. Then $\alpha$ induces a bijection $\alpha':X^{(0)}\to X^{(0)}$. If two vertices $x_1$ and $x_2$ of $X$ are adjacent in $X^{(1)}$, then there is a 2-cell $C$ containing $x_1$ and $x_2$. Vertices of the 2-cell give rise to an induced $k$-cycle $\omega$ in $\Theta$, where $x_1$ and $x_2$ correspond to adjacent vertices $v_1$ and $v_2$ in $\omega$.  As $\alpha(\omega)$ is an induced $k$-cycle in $\Theta$, Lemma~\ref{lemma:cycle} implies that $\alpha(v_1)$ and $\alpha(v_2)$ correspond to two adjacent vertices in the boundary of some 2-cell of $X$. Thus $\alpha'(x_1)$ and $\alpha'(x_2)$ are adjacent. By working with $\alpha^{-1}$ instead of $\alpha$, we also deduce that non-adjacent vertices are sent to non-adjacent vertices by $\alpha'$. Thus $\alpha'$ extends to a graph isomorphism $X^{(1)}\to X^{(1)}$. Moreover, Lemma~\ref{lemma:cycle} implies that $\alpha'$ sends boundaries of 2-cells in $X$ to boundaries of 2-cells (as they correspond to induced $k$-cycles in $\Theta$). Thus $\alpha'$ further extends to an automorphism of $X$. This defines a homomorphism $\Aut(\Theta)\to\Aut(X)$ (sending $\alpha$ to $\alpha'$), and it is injective by construction. In addition, the natural homomorphism $\hat{G}\to\Aut(X)$ factors through the natural homomorphism $\hat{G}\to\Aut(\Theta)$. Since the former is an isomorphism (Proposition~\ref{prop:combinatorial rigid X}), so is the latter, which completes our proof.
	\end{proof}

\begin{rk}
\label{rk:example}
Here we leave a remark on the assumptions of Theorem~\ref{theo:crigidity0}. This remark will not be used in the later part of the article.
Note that Theorem~\ref{theo:crigidity0} relies crucially on the interlocking pattern of the Baumslag-Solitar groups in the generalized Higman groups. If we change the pattern of how the BS subgroups fit together, then the theorem fails. For example, we can change one relator in $\Hig_5$ from $a_2a_3a^{-1}_2=a^2_3$ to $a_3a_2a^{-1}_3=a^2_2$
to obtain a new group
$$
G=\langle a_1,\dots,a_5|a_1a_2a^{-1}_1=a^2_2,a_3a_2a^{-1}_3=a^2_2,a_3a_4a^{-1}_3=a^2_4,a_4a_5a^{-1}_4=a^2_5,a_5a_1a^{-1}_5=a^2_1\rangle
$$
We can define the developed complex $X$ and the intersection graph $\Theta$ for $G$ similarly.   However, as we will now explain, elements in $\Aut(X)$ do not arise from the action of an element of $G$ in general (in fact not even from an automorphism of $G$).

Let $\Upsilon_G$ be the Cayley graph of $G$. We define standard lines and standard BS-subgraphs in $\Upsilon_G$ as before, and their intersection with $G$ are defined to be standard lines and \emph{standard BS-subsets} in $G$. 
A bijection of $G\to G$ is \emph{admissible} if it preserves standard lines and standard BS-subsets. Elements in $\Aut(X)$ are in 1-1 correspondence with admissible bijections of $G$.
Now we give an admissible bijection $G\to G$ which is not an automorphism.

Note that $G$ admits a splitting $G=H_1*_L H_2$, where $H_1=\langle a_1,a_2,a_3\rangle$, $H_2=\langle a_1,a_4,a_5,a_3\rangle$, and $L=\langle a_1,a_3\rangle$. Moreover, $H_1$ is a doubling of $\langle a_1,a_2\rangle$ in the sense that it can be thought of as the amalgamation of two identical copies of $\langle a_1,a_2\rangle$ along the subgroup generated by $a_2$. Let $f:\langle a_1,a_2\rangle\to \langle a_1,a_2\rangle$ be the line-preserving bijection as in Example~\ref{ex:12}. We can use the amalgam structure of $H_1$ to double the map $f$ to obtain an admissible bijection $g:H_1\to H_1$. Note that $g$ does not respect orders along standard lines of type $a_2$, though it does respect orders along standard lines of type $a_1$ and $a_3$. In particular, $g$ permutes the left cosets of $L$ in $H_1$, and for each such left coset $E$, there exists $h\in H_1$ (depending on $E$) such that $g(k)=hk$ for any $k\in E$.

Now we use the splitting $G=H_1*_L H_2$ to extend $g$ to an admissible bijection $G\to G$ as follows. Let $T$ be the associated Bass-Serre tree. Then each vertex of $T$ corresponds to a left coset of $H_1$ or $H_2$, and each edge of $T$ corresponds to a left coset of $L$. Let $v_0\in T$ be the vertex associated with $H_1$. Then vertices in $T$ adjacent to $v_0$ correspond to left cosets of $H_2$ that have non-empty intersection with $H_1$. Take one such left coset, denoted by $E$. As we already know $g|_{E\cap H_1}$ coincides with left multiplication with an element $h_E\in H_1$. We can thus define the extension of $g$ to $G$ by letting it coincide with the left multiplication by $h_E$ on all left cosets of $H_1$ or $H_2$ that are represented by vertices of $T$ that lie in the same connected component of $T\setminus\{v_0\}$ as $E$. This defines an admissible bijection of $G$, which does not coincide with a left translation (because it does not coincide with a left translation on $H_1$).  
\end{rk}	
	
	\section{Measure equivalence rigidity of the generalized Higman groups}\label{sec:me}
	
In this section, we prove our main theorem (Theorem~\ref{theointro:main} from the introduction). We will actually formulate a general rigidity statement for groups acting on piecewise hyperbolic $\mathrm{CAT}(-1)$ simplicial complexes, from which the case of the Higman groups is derived. Most arguments in the section are phrased in the language of measured groupoids. The reader is referred to Appendix~\ref{sec:appendix-groupoids} for background and terminology. Let us mention that in this appendix, the only new notion is that of an \emph{action-like} cocycle; however, it is harmless to simply think about the natural cocycle associated to a probability measure-preserving action.
	
	\subsection{Statements}
	
	The reader is referred to Section~\ref{sec:combinatorial-rigidity} for the notion of the intersection graph of an action. Whenever $\Theta$ is a simplicial graph on countably many vertices, the group $\Aut(\Theta)$ of all graph automorphisms of $\Theta$ is equipped with the topology of pointwise convergence, which turns it into a Polish group. 
	
	\begin{theo}\label{theo:full}
		Let $X$ be a connected $\mathrm{CAT}(-1)$ piecewise hyperbolic polyhedral complex with countably many cells in finitely many isometry types. Let $G$ be a torsion-free countable group acting by cellular isometries on $X$  without inversion, and let $\Theta$ be the intersection graph of the action $G\actson X$. Assume that
		\begin{enumerate}
			\item[1)] (Vertex stabilizers). %\Ccom{the first part of this assumption is in fact unnecessary, there is a way to bypass it using a little fact from my recent paper with Amandine; should we do this?} The stabilizer of every vertex of $X$ is either amenable or contains a nonabelian free subgroup. In addition, 
			The stabilizer of every vertex of $X$ contains a commensurated infinite amenable subgroup.
			\item[2)] (Edge stabilizers). Edge stabilizers for the $G$-action on $X$ are amenable and of infinite index in the incident vertex stabilizers.
			\item[3)] (Weak acylindricity). The $G$-action on $X$ is weakly acylindrical.
			\item[4)] (Non-isolation of amenable vertex stabilizers). For each vertex $v\in X^{(0)}$ such that $\Stab_G(v)$ is amenable, there exists an infinite subgroup of $\Stab_G(v)$ which either fixes a vertex with nonamenable stabilizer, or fixes  two distinct vertices of $X$ which are different from $v$.
			\item[5)] (Chain condition). There is a bound on the size $k$ of a chain $v_1,\dots,v_k$ of vertices of $X$ such that for every $i\in\{1,\dots,k-1\}$, the pointwise stabilizer of $\{v_1,\dots,v_{i+1}\}$ has infinite index in the pointwise stabilizer of $\{v_1,\dots,v_i\}$.
		\end{enumerate}
		Then for every self measure equivalence coupling $\Sigma$ of $G$, there exists a $(G\times G)$-equivariant Borel map $\Sigma\to\Aut(\Theta)$, where the $(G\times G)$-action on $\Aut(\Theta)$ is via $(g_1,g_2)\cdot f(x)=g_1f(g_2^{-1}x)$.
	\end{theo}
	
	\begin{rk}\label{rk:coupling-cocycle}
		In fact, in this section, we will prove the following. Let $\calg$ be a measured groupoid over a standard finite measure space $Y$ with two action-like cocycles $\rho_1,\rho_2:\calg\to G$ (see the paragraph above Lemma~\ref{lemma:action-like} in Appendix~\ref{sec:appendix-groupoids} for the definition). Then the two cocycles are \emph{$\Aut(\Theta)$-cohomologous}, i.e.\ there exists a Borel map $\varphi:Y\to\Aut(\Theta)$ and a conull Borel subset $Y^*\subseteq Y$ such that for every $g\in\calg_{|Y^*}$, one has $\rho_2(g)=\varphi(r(g))\rho_1(g)\varphi(s(g))^{-1}$. Theorem~\ref{theo:full} as stated then follows for instance from \cite[Proposition~5.11(2)]{HH1}, which relies on techniques that originate in work of Furman \cite{Fur-oe}. To be precise, notice that \cite[Proposition~5.11(2)]{HH1} was stated with action-type (and not action-like) cocycles, but its proof only involves natural cocycles associated to restrictions of actions, which are action-like by Lemma~\ref{lemma:action-like}.
	\end{rk}

\begin{rk}\label{rk:other-assumptions-0}
The last assumption (Chain condition) is in fact unnecessary in the case where all vertex groups are amenable, which is why it does not appear in Theorem~\ref{theo:combination} from the introduction. The changes that need to be made in order to prove Theorem~\ref{theo:combination} will be explained in Section~\ref{sec:other-assumptions}.
\end{rk}
	
	Before we prove Theorem~\ref{theo:full}, let us explain how our main theorems from the introduction follow. 
	
	\begin{proof}[Proof of Theorem~\ref{theointro:main-2}]
		As observed at the beginning of Section~\ref{subsec:Higman background}, in proving Theorem~\ref{theointro:main}, we do not lose any generality by assuming that $|m_i|<|n_i|$ for each $i$ (instead of just $|m_i|\neq |n_i|$). Let $G=\Hig_\sigma$, and denote by $X$ the developed complex of $\Hig_\sigma$, constructed in Section~\ref{sec:metric}. As $k\ge 5$, the complex $X$,  equipped with the metric defined in Section~\ref{sec:metric}, is connected (Lemma~\ref{lem:sc}), piecewise hyperbolic and $\CAT(-1)$, and it has countably many cells in finitely many isometry types.
		Let $\Theta$ be the intersection graph of the $G$-action on $X$. The group $\Hig_\sigma$ is torsion-free by Lemma~\ref{lemma:Higman property}, and we claim that the action of $\Hig_\sigma$ on $X$ satisfies all assumptions from Theorem~\ref{theo:full}. First, this action is without inversion by construction. Assumptions~1 and~2 are satisfied because all vertex stabilizers are Baumslag-Solitar groups, and edge stabilizers are infinite cyclic. Assumption~3 (weak acylindricity) comes from Lemma~\ref{lemma:higman acylindrical}. By Lemma~\ref{lemma:intersection of vertex stabilizer}(3), for each vertex $v\in X^{(0)}$, there exist countably many mutually different vertices $\{v_i\}_{i=1}^\infty$ such that $\Stab_G(v)\cap \Stab_G(v_i)$ is infinite, thus Assumption~4 follows. Assumption~5 follows from Lemma~\ref{lemma:intersection of vertex stabilizer}(4). Therefore, Theorem~\ref{theo:full} ensures that for every self measure equivalence coupling $\Sigma$ of $G$, there exists a $(G\times G)$-equivariant Borel map $\Sigma\to\Aut(\Theta)$. This is exactly what we want, since $\Aut(\Theta)$ is isomorphic to $\widehat{\Hig}_\sigma$ (Theorem~\ref{theo:crigidity0}).
	\end{proof}
	
We now complete the proof of the measure equivalence rigidity of $\Hig_\sigma$.

\begin{proof}[Proof of Theorem~\ref{theointro:main}]
As $\Hig_\sigma$ is ICC (Lemma~\ref{lemma:Higman property}),  Theorem~\ref{theointro:main} follows from Theorem~\ref{theointro:main-2} in view of e.g.\ \cite[Corollary~5.4]{HH1}. 
\end{proof}
	
	The rest of the section is devoted to the proof of Theorem~\ref{theo:full}. 
	
	\paragraph*{Standing assumption.} From now on $G$ and $X$ are fixed once and for all to satisfy the assumptions of Theorem~\ref{theo:full}.

	\subsection{Elliptic and $\partial_\infty X$-elementary subgroupoids}
	
	We denote by $\partial_\infty X$ the visual boundary of $X$, equipped with the cone topology. Let $\calp_{\le 2}(\partial_\infty X)$ be the space of all  nonempty subsets of $\partial_\infty X$ of cardinality at most $2$: this is naturally identified to $(\partial_\infty X)^2/\mathfrak{S}_2$, and is equipped with the quotient topology of the product topology on $(\partial_\infty X)^2$. Let $\calp_{<\infty}(\partial_\infty X)$ denote the collection of all nonempty finite subsets of $\partial_\infty X$. This also has a natural topology, by taking the direct limit, as $n$ goes to $+\infty$, of the sets of all nonempty finite subsets of cardinality at most $n$. We also denote by $\calp_{=2}(\partial_\infty X)$ and $\calp_{\ge 3,<\infty}(\partial_\infty X)$ the subspace made of subsets of cardinality equal to $2$, and at least $3$, respectively.
	
	\begin{de}
		Let $\calh$ be a measured groupoid over a standard finite measure space $Y$, and let $\rho:\calh\to G$ be a strict cocycle. We say that $(\calh,\rho)$ is 
		\begin{itemize}
			\item \emph{$X$-elliptic} if there exist a conull Borel subset $Y^*\subseteq Y$ and a vertex $v\in X^{(0)}$ such that $\rho(\calh_{|Y^*})\subseteq\Stab_G(v)$;
			\item \emph{stably $X$-elliptic} if there exists a countable Borel partition $Y=\dunion_{i\in I}Y_i$ such that for every $i\in I$, $(\calh_{|Y_i},\rho)$ is $X$-elliptic;
			\item \emph{$\partial_\infty X$-elementary} if there exists an $(\calh,\rho)$-equivariant Borel map $Y\to \calp_{\le 2}(\partial_\infty X)$. 
		\end{itemize}
	\end{de}
	
	The following lemma says that up to a countable Borel partition of the base space, amenable groupoids equipped with a cocycle towards $G$ fall into one of the above categories. 
	
	\begin{lemma}\label{lemma:elliptic-loxodromic}
		Let $\cala$ be an amenable measured groupoid over a standard finite measure space $Y$, equipped with a strict cocycle $\rho:\cala\to G$. Then there exists a Borel partition $Y=Y_1\dunion Y_2$ such that 
		\begin{enumerate}
			\item $(\cala_{|Y_1},\rho)$ is stably $X$-elliptic, and 
			\item $(\cala_{|Y_2},\rho)$ is $\partial_\infty X$-elementary, and in fact there exists a stably $(\cala_{|Y_2},\rho)$-equivariant Borel map $\varphi:Y_2\to\calp_{\le 2}(\partial_\infty X)$ such that for every positive measure Borel subset $V\subseteq Y_2$, every stably $(\cala_{|V},\rho)$-equivariant Borel map $\psi:V\to\calp_{<\infty}(\partial_\infty X)$, and almost every $y\in V$, one has $\psi(y)\subseteq\varphi(y)$. 
		\end{enumerate}
	\end{lemma}
	
\begin{rk}
In the following proof and throughout the section, we will work with spaces of probability measures. Given a separable metrizable space $X$, we equip the space $\Prob(X)$ of all Borel probability measures on $X$ with the topology generated by the maps $\mu\mapsto\int fd\mu$, where $f$ varies over all real-valued continuous bounded functions on $X$. When $X$ is compact, this topology coincides with the weak-$*$ topology coming from the identification of $\Prob(X)$ to a subspace of the dual of $C(X;\mathbb{R})$ (the space of real-valued continuous functions on $X$), by the Riesz--Markov--Kakutani theorem. When $X$ is countable, the above topology on $\Prob(X)$ coincides with the topology of pointwise convergence. By \cite[Theorem~17.24]{Kec}, the Borel $\sigma$-algebra on $\Prob(X)$ is generated by the maps $\mu\mapsto\mu(A)$, where $A$ varies among all Borel subsets of $X$. The reader is refered to \cite[Section~17]{Kec} for more information, justifying the measurability of all maps involving spaces of probability measures in the following proof.
\end{rk}	
	
	\begin{proof}
		The proof is based on an argument of Adams \cite{Ada}. We let $Y_1$ be a Borel subset of $Y$ of maximal measure such that $(\cala_{|Y_1},\rho)$ is stably $X$-elliptic: this exists because if $(Y_{1,n})_{n\in\mathbb{N}}$ is a measure-maximizing sequence of such subsets, then their union again has the property. Let $Y_2=Y\setminus Y_1$. Our goal is to show that $Y_2$ satisfies the conclusion from the lemma. Notice that, by construction, there does not exist any positive measure Borel subset $U\subseteq Y_2$ such that $(\cala_{|U},\rho)$ is $X$-elliptic. 
		
		Let $X_\Delta$ be the second barycentric subdivision of $X$ (cf.\ \cite[Chapter I.7]{BH}): this is a simplicial complex, which is still $\mathrm{CAT}(-1)$ and piecewise hyperbolic, with countably many simplices in finitely many isometry types, and on which $G$ acts with amenable edge stabilizers. The visual boundaries $\partial_\infty X$ and $\partial_\infty X_\Delta$ are canonically ($G$-equivariantly) isomorphic. Fix a basepoint $x_0\in X_\Delta$, and let $\overline{X}_\Delta^h$ be the horocompactification of $X_\Delta$, defined as the closure of the image of the embedding 
		\begin{displaymath}
		\begin{array}{cccc}
		X_\Delta & \to & \mathcal{C}(X_\Delta,\mathbb{R}) \\
		x & \mapsto & d(x,\cdot)-d(x,x_0)
		\end{array}
		\end{displaymath}
		where $\mathcal{C}(X_\Delta,\mathbb{R})$ denotes the space of all real-valued continuous functions on $X_\Delta$, equipped with the compact-open topology. Denote by $\mathcal{S}(X_\Delta)$ the countable set of all simplices of $X_\Delta$, equipped with the discrete topology. By \cite[Theorem~I.7.19]{BH}, the space $X_\Delta$ is a complete metric space, and it is also separable because it contains countably many simplices. We can therefore apply \cite[Proposition~3.3]{HH1} (based on earlier work of Maher and Tiozzo \cite{MT}) and deduce that there exist a Borel partition $\overline{X}_\Delta^h=(\overline{X}_{\Delta}^h)_\infty\dunion (\overline{X}_{\Delta}^h)_{\mathrm{bdd}}$ together with $G$-equivariant Borel maps $\theta_\infty:(\overline{X}^h_\Delta)_\infty\to\partial_\infty X$ and $\theta_{\mathrm{bdd}}:(\overline{X}_\Delta^h)_{\mathrm{bdd}}\to \mathcal{S}(X_\Delta)$.
		
		The space $\overline{X}_\Delta^h$ is compact and metrizable (see \cite[Proposition~3.1]{MT}). Since $\cala$ is amenable, it follows (using \cite[Proposition~4.14]{Kid-survey}) that there exists an $(\cala,\rho)$-equivariant Borel map $\nu:Y\to\Prob(\overline{X}_\Delta^h)$. In the rest of the proof, we will usually write $\nu_y$ instead of $\nu(y)$.
		
		We claim that for almost every $y\in Y_2$, one has $\nu_y((\overline{X}^h_\Delta)_{\infty})=1$. Indeed, assume towards a contradiction that there exists a positive measure Borel subset $U\subseteq Y_2$ such that for every $y\in U$, the probability measure $\nu_y$ gives positive measure to $(\overline{X}^h_\Delta)_{\bdd}$. By restricting $\nu_y$ to $(\overline{X}^h_\Delta)_{\bdd}$, renormalizing it to a probability measure, and pushforwarding it through the Borel map $\theta_\bdd$, we deduce an $(\cala_{|U},\rho)$-equivariant Borel map $U\to\Prob(\cals(X_\Delta))$. As $\cals(X_\Delta)$ is countable, there is a natural $G$-equivariant Borel map $\Prob(\cals(X_\Delta))\to\calp_{<\infty}(\cals(X_\Delta))$, sending a probability measure $\mu$ to the nonempty finite subset of $\cals(X_\Delta)$ consisting of all simplices with maximal $\mu$-measure. There is also a $G$-equivariant \emph{circumcenter map} $\calp_{<\infty}(\cals(X_\Delta))\to \cals(X_\Delta)$, defined as follows: given a nonempty finite set $F$ of simplices of $X_\Delta$, the set $\hat{F}$ of all vertices of simplices in $F$, viewed as a subset of $X_\Delta$, has a unique circumcenter $c$ (see e.g.\ \cite[Proposition~II.2.7]{BH}), and we map $F$ to the smallest simplex of $X_\Delta$ that contains $c$. Finally, denoting by  $\mathrm{Cell}(X)$ the countable set consisting of all cells of $X$, there is a $G$-equivariant map $\mathcal{S}(X_\Delta)\to\mathrm{Cell}(X)$, sending a simplex to the smallest cell of $X$ that contains it. Combining these maps yields an $(\cala_{|U},\rho)$-equivariant Borel map $U\to \mathrm{Cell}(X)$. As the $G$-action on $X$ is without inversion, this implies that $(\cala_{|U},\rho)$ is stably $X$-elliptic, a contradiction to the definition of $Y_2$. This proves our claim.
		
		By restricting the probability measures $\nu_y$ to $(\overline{X}^h_\Delta)_\infty$ and pushforwarding them through the Borel map $\theta_\infty$, we obtain an $(\cala_{|Y_2},\rho)$-equivariant Borel map $\alpha:Y_2\to\Prob(\partial_\infty X)$.
		
		We now claim that for every positive measure Borel subset $U\subseteq Y_2$, every $(\cala_{|U},\rho)$-equivariant Borel map $\mu:U\to\Prob(\partial_\infty X)$ and almost every $y\in U$, the support of $\mu_y$ has cardinality at most $2$. Indeed, assume towards a contradiction that there exists a positive measure Borel subset  $V\subseteq U$ such that for every $y\in V$, the support of $\mu_y$ has cardinality at least $3$. Then for every $y\in V$, the probability measure  $\mu(y)\otimes\mu(y)\otimes\mu(y)$ on $(\partial_\infty X)^3$ gives positive measure to the Borel subset $(\partial_\infty X)^{(3)}$ consisting of pairwise distinct triples. This yields an $(\cala_{|V},\rho)$-equivariant map $V\to\Prob((\partial_\infty X)^{(3)})$, and this map is Borel as $\mu$ is Borel. Using the existence of a Borel \emph{barycenter map} $(\partial_\infty X)^{(3)}\to \cals(X_\Delta)$ (see \cite[Proposition~3.5]{HH1}), we derive an $(\cala_{|V},\rho)$-equivariant Borel map $V\to\Prob(\cals(X_\Delta))$. As in the previous paragraph, this implies that $(\cala_{|V},\rho)$ is stably $X$-elliptic, a contradiction. 
		
		Notice in particular that the above claim has the following consequence (which is proved by first mapping every nonempty finite set to the uniform probability measure on this set):
		
		\medskip
		\noindent\textbf{Fact:} For every positive measure Borel subset $U\subseteq Y_2$, every $(\cala_{|U},\rho)$-equivariant Borel map $U\to\calp_{<\infty}(\partial_\infty X)$ essentially takes its values in $\calp_{\le 2}(\partial_\infty X)$.
		\medskip
		
		By considering the supports of the measures $\alpha_y$, we see that there exists an $(\cala_{|Y_2},\rho)$-equivariant Borel map $\varphi':Y_2\to\calp_{\le 2}(\partial_\infty X)$, showing that $(\cala_{|Y_2},\rho)$ is $\partial_\infty X$-elementary. Let now $U\subseteq Y_2$ be a Borel subset of maximal measure such that there exists a stably $(\cala_{|U},\rho)$-equivariant Borel map $\varphi_U:U\to\calp_{=2}(\partial_\infty X)$. 
		
		Given $\varphi_U$ as above, we finally prove that the map $\varphi$ which coincides with $\varphi_U$ on $U$ and with $\varphi'$ on $Y_2\setminus U$ satisfies the conclusion of the lemma. Indeed, let $V\subseteq Y_2$ be a Borel subset of positive measure, and let $\psi:V\to\calp_{<\infty}(\partial_\infty X)$ be a stably $(\cala_{|V},\rho)$-equivariant Borel map. Assume that there exists a Borel subset $W\subseteq V$ of positive measure such that for every $y\in W$, one has $\psi(y)\nsubseteq\varphi(y)$. Then the stably $(\cala_{|W},\rho)$-equivariant Borel map $\psi\cup\varphi$ either takes essentially its values in $\calp_{\ge 3,<\infty}(\partial_\infty X)$ (which would contradict the above fact), or it takes its values in $\calp_{=2}(\partial_\infty X)$ on a positive measure subset of $Y_2\setminus U$, which contradicts the maximality of $U$. This completes our proof.  
	\end{proof} 
	
	\subsection{Exploiting amenable subgroupoids}
	
	In the following lemma, we will crucially exploit the weak acylindricity of the $G$-action on $X$.
	
	\begin{lemma}
		\label{lemma:elliptic vs elementary}
		Let $\calh$ be a measured groupoid over a standard finite measure space $Y$, let $\rho:\calh\to G$ be a strict cocycle with trivial kernel, and assume that $\calh$ is nowhere trivial.
		
		Then $(\calh,\rho)$ cannot be both stably $X$-elliptic and $\partial_\infty X$-elementary.
	\end{lemma}
	
	\begin{proof}
		Assume towards a contradiction that $(\calh,\rho)$ is both stably $X$-elliptic and $\partial_\infty X$-elementary. Up to replacing $Y$ by a positive measure Borel subset, we can assume that there exists $v\in X^{(0)}$ such that $\rho(\calh)\subseteq\Stab_G(v)$. Let also $\theta:Y\to\calp_{\le 2}(\partial_\infty X)$ be an $(\calh,\rho)$-equivariant Borel map. Up to replacing $Y$ by a conull Borel subset, we can assume that for every $h\in\calh$, one has $\theta(r(h))=\rho(h)\theta(s(h))$.
		
		Since the $G$-action on $X$ is weakly acylindrical, by Lemma~\ref{lemma:weakly acylindrical}, every element $\xi\in\calp_{\le 2}(\partial_\infty X)$ has an open neighborhood $U_\xi\subseteq\calp_{\le 2}(\partial_\infty X)$ such that for every nontrivial element $g\in\Stab_G(v)$, one has $gU_\xi\cap U_\xi=\emptyset$. Let $\nu$ be the finite Borel measure on $\calp_{\le 2}(\partial_\infty X)$ obtained by pushforwarding the probability measure on $Y$ through the Borel map $\theta$. Let $\xi\in\calp_{\le 2}(\partial_\infty X)$ be a point in the support of $\nu$. Then the Borel set $Y_\xi\subseteq Y$ consisting of all points $y\in Y$ such that $\theta(y)\in U_\xi$ has positive measure. But for every $h\in\calh_{|Y_{\xi}}$, the element $\rho(h)$ must both fix $v$ and send an element in $U_\xi$ to another element in $U_\xi$, so by the above $\rho(h)=1$. As $\rho$ has trivial kernel, it follows that the groupoid $\calh_{|Y_\xi}$ is trivial, contradicting our assumption that $\calh$ is nowhere trivial.      
	\end{proof}
	
	We say that a measured subgroupoid $\calh\subseteq\calg$ is \emph{stably maximal} with respect to some property if it satisfies the property, and for every measured subgroupoid $\hat\calh\subseteq\calg$ satisfying the property, if $\calh$ is stably contained in $\hat\calh$, then they are stably equal.
	
	\begin{lemma}
		\label{lemma:maximal amenable}
		Let $\calg$ be a measured groupoid over a standard finite measure space $Y$, let $\rho:\calg\to G$ be a strict cocycle with trivial kernel, and let $\cala$ be a stably maximal amenable measured subgroupoid of $\calg$.
		
		Then there exists a Borel partition $Y=Y_1\dunion Y_2$ such that $(\cala_{|Y_1},\rho)$ is stably $X$-elliptic, and there exists a Borel map $\varphi:Y_2\to\calp_{\le 2}(\partial_\infty X)$ such that $\cala_{|Y_2}$ is stably equal to the $(\calg_{|Y_2},\rho)$-stabilizer of $\varphi$.
	\end{lemma}
	
	\begin{proof}
	 Lemma~\ref{lemma:elliptic-loxodromic} ensures that there exists a Borel partition $Y=Y_1\dunion Y_2$ such that $(\cala_{|Y_1},\rho)$ is stably $X$-elliptic, and $\cala_{|Y_2}$ is contained in the $(\calg_{|Y_2},\rho)$-stabilizer of some Borel map $\varphi:Y_2\to\calp_{\le 2}(\partial_\infty X)$. It is therefore enough to prove that the $(\calg_{|Y_2},\rho)$-stabilizer of $\varphi$ is amenable, as the stable maximality assumption will then conclude the proof.
	
		As in the proof of Lemma~\ref{lemma:elliptic-loxodromic}, let $X_\Delta$ be the second barycentric subdivision of $X$. This is a connected, piecewise hyperbolic $\mathrm{CAT}(-1)$ simplicial complex with countably many simplices in finitely many isometry types, on which $G$ acts by simplicial isometries. As edge stabilizers of the $G$-action on $X$ are amenable (Assumption~2 from Theorem~\ref{theo:full}), so are edge stabilizers for the $G$-action on $X_\Delta$. The visual boundaries $\partial_\infty X$ and $\partial_\infty X_\Delta$ are naturally identified. It thus follows from  \cite[Theorem~1(2)]{HH1} that the $G$-action on $\partial_\infty X$ is Borel amenable, and therefore so is the $G$-action on $\calp_{\le 2}(\partial_\infty X)$ (see e.g.\ \cite[Lemma~6.14]{HH1}). As $\rho$ has trivial kernel, it follows from \cite[Proposition~3.38]{GH} (recasting \cite[Proposition~4.33]{Kid-memoir}) that for every Borel subset $U\subseteq Y$ of positive measure, the $(\calg_{|U},\rho)$-stabilizer of every Borel map $U\to\calp_{\le 2}(\partial_\infty X)$ is amenable. This concludes our proof.  
	\end{proof}
	
	\begin{lemma}
		\label{lemma:intersection of four gropuoids}
		Let $\calg$ be a measured groupoid over a standard finite measure space $Y$, and let $\rho:\calg\to G$ be a strict cocycle with trivial kernel. Let  $\cala_1,\cala_2,\cala_3$ be three stably maximal amenable subgroupoids of $\calg$, and assume that there do not exist $i\neq j$ and a Borel subset $U\subseteq Y$ of positive measure such that $(\cala_i)_{|U}$ has a finite-index subgroupoid contained in $(\cala_j)_{|U}$. Let $\cala=\cala_1\cap\cala_2\cap\cala_3$.
		
		Then $(\cala,\rho)$ is stably $X$-elliptic. 
	\end{lemma}
	
	\begin{proof}
		Without loss of generality, we can (and shall) assume that $\cala$ is nowhere trivial. Indeed, if $U\subseteq Y$ is a Borel subset of maximal measure such that $\cala_{|U}$ is stably trivial, then $(\cala_{|U},\rho)$ is stably $X$-elliptic, and $\cala_{|Y\setminus U}$ is nowhere trivial.
		
	 The groupoid $\cala$ is amenable. Therefore, by Lemma~\ref{lemma:elliptic-loxodromic}, there exists a Borel partition $Y=Y_1\dunion Y_2$ such that $(\cala_{|Y_1},\rho)$ is stably $X$-elliptic, and there exists an $(\cala_{|Y_2},\rho)$-equivariant Borel map $\varphi:Y_2\to\calp_{\le 2}(\partial_\infty X)$. By Lemma~\ref{lemma:elliptic vs elementary}, for every Borel subset $U\subseteq Y_2$ of positive measure, $(\cala_{|U},\rho)$ is not $X$-elliptic (so in particular $((\cala_i)_{|U},\rho)$ is not $X$-elliptic for any $i\in\{1,2,3\}$). Therefore, Lemma~\ref{lemma:maximal amenable} ensures that there exist three Borel maps $\varphi_1,\varphi_2,\varphi_3:Y_2\to\calp_{\le 2}(\partial_\infty X)$ such that for every $i\in\{1,2,3\}$, the groupoid $(\cala_i)_{|Y_2}$ is stably equal to the $(\calg_{|Y_2},\rho)$-stabilizer of $\varphi_i$. 
	 
 Our assumptions on the subgroupoids $\cala_i$ ensure that there do not exist any Borel subset $V\subseteq Y_2$ of positive measure and $i\neq j$ such that $(\varphi_i)_{|V}=(\varphi_j)_{|V}$. More generally, we claim that there do not exist any Borel subset $V\subseteq Y_2$ of positive measure and $i\neq j$ such that $(\varphi_i)_{|V}\subseteq (\varphi_j)_{|V}$. Indeed, otherwise, we can assume without loss of generality that $(\varphi_j)_{|V}$ takes its values in $\calp_{=2}(\partial_\infty X)$, that $(\varphi_i)_{|V}$ takes its values in $\calp_{=1}(\partial_\infty X)$, and we let $\varphi'_i(y)=\varphi_j(y)\setminus\varphi_i(y)$. Then observe that the stabilizer of the Borel map $((\varphi_i)_{|V},(\varphi'_i)_{|V})$ has index at most $2$ (almost everywhere) in $(\cala_j)_{|V}$. Thus, a finite index subgroupoid of $(\cala_j)_{|V}$ is contained in $(\cala_i)_{|V}$, contradicting our assumption. 
	 
	 In other words, we have just shown that there exists a conull Borel subset $Y_2^*\subseteq Y_2$ such that for every $y\in Y_2^*$ and any two distinct $i,j\in\{1,2,3\}$, we have $\varphi_i(y)\nsubseteq\varphi_j(y)$. It follows that the stably $(\cala_{|Y_2^*},\rho)$-equivariant map $\varphi_1\cup\varphi_2\cup\varphi_3$ takes its values in $\calp_{\ge 3,<\infty}(\partial_\infty X)$. Notice that there is a Borel map $\calp_{\ge 3,<\infty}(\partial_\infty X)\to\calp_{<\infty}((\partial_\infty X)^{(3)})$, sending a finite set $F$ to the (finite, nonempty) set of all triples of elements in $F$. There is also a Borel $G$-equivariant \emph{barycenter map} $\calp_{<\infty}((\partial_\infty X)^{(3)})\to\calp_{<\infty}(\cals(X_\Delta))$ (as in the proof of Lemma~\ref{lemma:elliptic-loxodromic}), and a Borel $G$-equivariant \emph{circumcenter map} $\calp_{<\infty}(\cals(X_\Delta))\to\cals(X_\Delta)$ (again, as in the proof of Lemma~\ref{lemma:elliptic-loxodromic}). Combining all these maps together, we derive a stably $(\cala_{|Y_2^*},\rho)$-equivariant Borel map $Y_2^*\to\cals(X_\Delta)$, which shows that $(\cala_{|Y_2},\rho)$ is stably $X$-elliptic. This completes our proof.  
	\end{proof}
	
	\subsection{Exploiting quasi-normalized amenable subgroupoids}

	\begin{de}
		A measured groupoid $\calg$ is  \emph{QNA}\footnote{standing for Quasi-Normalizing an Amenable subgroupoid} if it contains an amenable subgroupoid $\cala$ of infinite type, such that $\cala$ is stably quasi-normalized by $\calg$. 
	\end{de}
	
	Notice that this notion is stable under restriction and stabilization. More precisely, if $\calh$ is a QNA measured subgroupoid of $\calg$ and $U\subseteq Y$ is a positive measure Borel subset, then $\calh_{|U}$ is a QNA measured subgroupoid of $\calg_{|U}$. And if $Y^*=\dunion Y_i$ is a countable Borel partition of a conull Borel subset $Y^*\subseteq Y$, such that for every $i\in I$, the subgroupoid $\calh_{|Y_i}$ of $\calg_{|Y_i}$ is QNA, then $\calh$ is QNA. Notice also that every amenable groupoid of infinite type is QNA. Our goal here is to prove the following.
	
	\begin{lemma}\label{lemma:nonamenable-BS-elliptic}
		Let $\calh$ be a measured groupoid over a standard finite measure space $Y$, and let $\rho:\calh\to G$ be a strict cocycle with trivial kernel.
		
		If $\calh$ is everywhere nonamenable and QNA, then $(\calh,\rho)$ is stably $X$-elliptic.
	\end{lemma}
	
	\begin{proof}
		This is a consequence of Lemma~\ref{lemma:loxodromicity-quasi-normalizer} and Corollary~\ref{cor:ellipticity-quasi-normalizer} below.
	\end{proof}
	
	\subsubsection{Quasi-normalized amenable subgroupoids are stably elliptic.}

	\begin{lemma}\label{lemma:loxodromicity-quasi-normalizer}
		Let $\calh$ be a measured groupoid over a standard finite measure space $Y$, and let $\rho:\calh\to G$ be a strict cocycle with trivial kernel. Assume that $\calh$ is everywhere nonamenable, and let $\cala$ be an amenable subgroupoid of $\calh$ which is stably quasi-normalized by $\calh$.
		
		Then $(\cala,\rho)$ is stably $X$-elliptic.
	\end{lemma}
	
	\begin{proof}
		We follow an argument of Kida \cite[Theorem~5.1]{Kid-BS}, which is itself a variation over an argument of Adams \cite{Ada} (see also \cite[Theorem~4.1]{HR}).
		
		Assume towards a contradiction that $(\cala,\rho)$ is not stably $X$-elliptic. By Lemma~\ref{lemma:elliptic-loxodromic}, there exists a positive measure Borel subset $U\subseteq Y$ such that $(\cala_{|U},\rho)$ is $\partial_\infty X$-elementary; in fact $U$ can be chosen so that there exists a stably $(\cala_{|U},\rho)$-equivariant Borel map $\varphi:U\to\calp_{\le 2}(\partial_\infty X)$ such that for every positive measure Borel subset $V\subseteq U$, every (stably) $(\cala_{|V},\rho)$-equivariant Borel map $\psi:V\to\calp_{<\infty}(\partial_\infty X)$, and almost every $y\in V$, one has $\psi(y)\subseteq\varphi(y)$. 
		
		Up to replacing $U$ by a positive measure Borel subset, we will assume that $\cala$ is quasi-normalized by $\calh$ (not just stably quasi-normalized).
		
		We now claim that $\varphi$ is  $(\calh_{|U},\rho)$-equivariant. Let $B\subseteq\calh_{|U}$ be a bisection which determines an isomorphism $\theta:V\to W$ between two Borel subsets $V,W\subseteq U$, and such that $\cala$ is quasi-$B$-invariant. Up to subdividing $B$ into countably many Borel subsets and working on each of these subsets separately, we will assume that the $\rho$-image of $B$ contains a single element $g\in G$, and prove that for almost every $x\in V$, one has $\varphi(\theta(x))=g\varphi(x)$. As $\calh$ is covered by countably many bisections that all leave $\cala$ quasi-invariant, this will be enough to prove our claim.
		
		The map $\psi:W\to\calp_{\le 2}(\partial_\infty X)$ sending $y\in W$ to $g\varphi(\theta^{-1}(y))$ is $(B\cala_{|V}B^{-1},\rho)$-equivariant. As $\cala_{|W}\cap (B\cala_{|V}B^{-1})$ has finite index in $\cala_{|W}$, by applying Lemma~\ref{lemma:fi} with  $\Delta=\calp_{\le 2}(\partial_\infty X)$, we know that there exists a stably $(\cala_{|W},\rho)$-equivariant Borel map $\Psi:W\to\calp_{<\infty}(\partial_\infty X)$ such that for every $y\in W$, one has $\psi(y)\subseteq\Psi(y)$. By definition of $\varphi$, we also have $\Psi(y)\subseteq\varphi(y)$ for almost every $y\in W$, so $\psi(y)\subseteq\varphi(y)$. Thus for almost every $y\in W$, one has $\varphi(\theta^{-1}(y))\subseteq g^{-1}\varphi(y)$. By working with the bisection $B^{-1}$ instead of $B$ (and exchanging the roles of $V$ and $W$), we know that for almost every $x\in V$, one has $\varphi(\theta(x))\subseteq g\varphi(x)$. Thus for almost $y\in W$, one has $\varphi(y)\subseteq g\varphi(\theta^{-1}(y))$, i.e.\ $g^{-1}\varphi(y)\subseteq \varphi(\theta^{-1}(y))$. We have proved both inclusions, so for almost every $y\in W$, one has $g^{-1}\varphi(y)=\varphi(\theta^{-1}(y))$, so $\varphi$ is  $(B,\rho)$-equivariant. This finishes the proof of our claim.
		
		As the $G$-action on the second barycentric subdivision $X_\Delta$ has amenable edge stabilizers, it follows from \cite[Theorem~1]{HH1} that the $G$-action on $\partial_\infty X$ is Borel amenable, and therefore so is the $G$-action on $\calp_{\le 2}(\partial_\infty X)$. As $\rho$ has trivial kernel, the existence of an  $(\calh_{|U},\rho)$-equivariant Borel map $U\to\calp_{\le 2}(\partial_\infty X)$ ensures that $\calh_{|U}$ is amenable by \cite[Proposition~3.38]{GH} (recasting \cite[Proposition~4.33]{Kid-memoir}). This contradiction finishes our proof. 
	\end{proof}
	
	\subsubsection{Ellipticity, canonical vertex sets, quasi-normality}
	
	The following is a version of \cite[Section~3.3]{HH2}, inspired from Kida's work for surface mapping class groups \cite[Chapter~4, Section~5]{Kid-memoir}, in the presence of quasi-normalized (instead of just normalized) subgroupoids.
	
	\begin{de}[Canonical vertex set]\label{de:cvs}
		Let $\calh$ be a measured groupoid over a standard finite measure space $Y$, and let $\rho:\calh\to G$ be a strict cocycle. A \emph{canonical vertex set} for $(\calh,\rho)$ is a (possibly infinite) set  $\mathsf{V}\subseteq X^{(0)}$ such that 
		\begin{enumerate}
			\item every vertex in $\mathsf{V}$ is $(\calh,\rho)$-quasi-invariant;
			\item for every positive measure Borel subset $U\subseteq Y$, no vertex in $X^{(0)}\setminus \mathsf{V}$ is $(\calh_{|U},\rho)$-quasi-invariant. 
		\end{enumerate}
	\end{de}
	
	A canonical vertex set, if it exists, is unique. The following lemma, which crucially relies on the chain condition on vertices of $X$ (Assumption~5 from Theorem~\ref{theo:full}), gives the existence of canonical vertex sets after a countable partition of the base space $Y$.
	
	\begin{lemma}\label{lemma:canonical-vertex-set}
		 Let $\calh$ be a measured groupoid over a standard finite measure space $Y$, and let $\rho:\calh\to G$ be a strict cocycle.
		
		Then there exists a countable Borel partition $Y=\dunion_{i\in I}Y_i$ such that for every $i\in I$, there exists a subset $\mathsf{V}_i\subseteq X^{(0)}$ which is a canonical vertex set for $(\calh_{|Y_i},\rho)$.
	\end{lemma}
	
	\begin{proof}
		Let $U\subseteq Y$ be a Borel subset of maximal measure such that there exist a countable Borel partition $U=\dunion_{j\in J}U_j$ and for every $j\in J$, a vertex $v_j\in X^{(0)}$ which is $(\calh_{|U_j},\rho)$-quasi-invariant. This exists because if $(U_n)_{n\in\mathbb{N}}$ is a measure-maximizing sequence of such sets, then their union again has the property. We now observe that $\emptyset$ is a canonical vertex set for $(\calh_{|Y\setminus U},\rho)$. Indeed, otherwise, there exist a positive measure Borel subset $V\subseteq Y\setminus U$ and a vertex $v\in X^{(0)}$ which is $(\calh_{|V},\rho)$-quasi-invariant, but then the set $U\cup V$ contradicts the maximality of the measure of $U$ with the above property.
		
		Let now $j\in J$. Let $\mathsf{V}_{j}$ be the set of all vertices of $X$ which are fixed by a finite-index subgroup of $\Stab_G(v_j)$ (notice that the second assertion of Theorem~\ref{theo:full} ensures that $\mathsf{V}_j=\{v_j\}$, but we prefer to denote it this way to make the resemblance with the rest of this proof more transparent). All vertices in $\mathsf{V}_{j}$ are $(\calh_{|U_{j}},\rho)$-quasi-invariant. 
		Let $U'_j\subseteq U_j$ be a Borel subset of maximal measure such that there exist a countable Borel partition $U'_j=\dunion_{k\in K}U_{j,k}$ and for every $k\in K$, a vertex $v_{j,k}\in X^{(0)}\setminus \mathsf{V}_j$ which is $(\calh_{|U_{j,k}},\rho)$-quasi-invariant. We observe that $\mathsf{V}_j$ is a canonical vertex set for $(\calh_{|U_j\setminus U'_j},\rho)$. Indeed, it verifies the first point of Definition~\ref{de:cvs} by the above (and \cite[Lemma~3.5(ii)]{Kid-BS}). And the second point of Definition~\ref{de:cvs} is checked as follows: if some vertex $w\in X^{(0)}\setminus \mathsf{V}_j$ were  $(\calh_{|V},\rho)$-quasi-invariant for some positive measure Borel subset $V\subseteq U_j\setminus U'_j$, then $U'_j\cup V$ would contradict the maximality of the measure of $U'_j$. Finally, we note that for every $k\in K$, the intersection $\Stab_G(v_j)\cap\Stab_G(v_{j,k})$ is infinite index in $\Stab_G(v_j)$, otherwise $v_{j,k}\in \mathsf{V}_j$.
		
		Let now $k\in K$. Repeating the above, let $\mathsf{V}_{j,k}$ be the set of all vertices of $X$ which are fixed by a finite-index subgroup of $\Stab_G(v_j)\cap\Stab_G(v_{j,k})$. Notice that all vertices in $\mathsf{V}_{j,k}$ are $(\calh_{|U_{j,k}},\rho)$-quasi-invariant, using \cite[Lemma~3.7]{Kid-BS}.
		 Let $U'_{j,k}\subseteq U_{j,k}$ be a Borel subset of maximal measure such that there exist a countable Borel partition $U'_{j,k}=\dunion_{\ell\in L}U_{j,k,\ell}$ and for every $\ell\in L$, a vertex $v_{j,k,\ell}\in X^{(0)}\setminus \mathsf{V}_{j,k}$ which is $(\calh_{|U_{j,k,\ell}},\rho)$-quasi-invariant. Then $\mathsf{V}_{j,k}$ is a canonical vertex set for $(\calh_{|U_{j,k}\setminus U'_{j,k}},\rho)$. Notice also that the pointwise stabilizer of $\{v_j,v_{j,k},v_{j,k,\ell}\}$ has infinite index in the pointwise stabilizer of $\{v_j,v_{j,k}\}$. We denote by $\mathsf{V}_{j,k,\ell}$ the set of all vertices of $X$ which are fixed by a finite-index subgroup of $\Stab_G(v_j)\cap\Stab_G(v_{j,k})\cap\Stab_G(v_{j,k,\ell})$.
		
		Repeating this argument, we build a sequence $v_j,v_{j,k},v_{j,k,\ell},\dots$ such that the common stabilizer of the first $k+1$ vertices always has infinite index in the common stabilizer of the first $k$ vertices. Our chain condition (Assumption~5 from Theorem~\ref{theo:full}) implies that this process has to stop after finitely many steps, and at the end we have constructed a countable Borel partition of $Y$ with the required property.  
	\end{proof}
	
	\begin{lemma}\label{lemma:canonical-vertex-set-quasi-normalizer}
		Let $\calh$ be a measured groupoid over a standard finite measure space $Y$, and let $\rho:\calh\to G$ be a strict cocycle. Let $\cala$ be a subgroupoid of $\calh$ which is quasi-normalized by $\calh$, and assume that $(\cala,\rho)$ has a canonical vertex set $\mathsf{V}$.
		
		Then $\mathsf{V}$ is $(\calh,\rho)$-invariant.
	\end{lemma}
	
	\begin{proof}
		Write $\calh$ as a union of countably many bisections $B_n$, with source $U_n$ and range $V_n$, such that $(B_n\cala_{|U_n}B_n^{-1})\cap\cala_{|V_n}$ has finite index in both $B_n\cala_{|U_n}B_n^{-1}$ and $\cala_{|V_n}$. Without loss of generality, we can (and will) assume that the $\rho$-image of each $B_n$ is a single element $g_n$. It suffices to show that $g_n\in\Stab_G(\mathsf{V})$ for each $n$.
		
		Let $v\in \mathsf{V}$. Let $\cala^0$ be a finite-index subgroupoid of $\cala$ such that $v$ is $(\cala^0,\rho)$-invariant. Thus $g_nv$ is $(B_n\cala^0_{|U_n}B_n^{-1},\rho)$-invariant. Then $\cala'_n=(B_n\cala^0_{|U_n}B_n^{-1})\cap\cala_{|V_n}$ is a finite-index subgroupoid of $\cala_{|V_n}$ such that $g_nv$ is $(\cala'_n,\rho)$-invariant. It follows that $g_nv\in \mathsf{V}$, showing that $g_n\mathsf{V}\subseteq \mathsf{V}$. Working with the bisection $B_n^{-1}$ gives $g^{-1}_n\mathsf{V}\subseteq \mathsf{V}$, hence $g_n\mathsf{V}=\mathsf{V}$, i.e.\ $g_n\in\Stab_G(\mathsf{V})$. 
	\end{proof}
	
	In the following corollary, we use the weak acylindricity assumption on the $G$-action on $X$ to deduce that canonical vertex sets are bounded.
	
	\begin{cor}\label{cor:ellipticity-quasi-normalizer}
		Let $\calh$ be a measured groupoid over a standard finite measure space $Y$, and let $\rho:\calh\to G$ be a strict cocycle with trivial kernel. Let $\cala$ be a measured subgroupoid of $\calh$ of infinite type, and which is quasi-normalized by $\calh$.
		
		If $(\cala,\rho)$ is stably $X$-elliptic, then so is $(\calh,\rho)$.
	\end{cor}
	
	\begin{proof}
		By Lemma~\ref{lemma:canonical-vertex-set}, up to a countable Borel partition of the base space $Y$, we can (and will) assume that $(\cala,\rho)$ has a canonical vertex set $\mathsf{V}$. As $(\cala,\rho)$ is stably $X$-elliptic, $\mathsf{V}$ is non-empty.
		
		We first claim that $\mathsf{V}$ is bounded in $X$. Indeed, as the $G$-action on $X$ is weakly acylindrical (Assumption~3 from Theorem~\ref{theo:full}), there exists $L>0$ such that any two vertices of $X$ at distance at least $L$ have a finite common stabilizer (whence trivial since $G$ is torsion-free). So if $\mathrm{diam}_X(\mathsf{V})>L$, then $\cala$ would contain the trivial subgroupoid as a finite-index subgroupoid (using that $\rho$ has trivial kernel),  contradicting that $\cala$ is of infinite type. This proves our claim. 
		
		Therefore $\mathsf{V}$ has a circumcenter $c$ (cf.\ \cite[Proposition II.2.7]{BH}), which is contained in a unique cell $P$ of minimal dimension. By Lemma~\ref{lemma:canonical-vertex-set-quasi-normalizer}, the set $\mathsf{V}$ is $(\calh,\rho)$-invariant, and therefore so is $P$. As the $G$-action on $X$ is without inversion, every vertex of $P$ is $(\calh,\rho)$-invariant, showing that $(\calh,\rho)$ is (stably) $X$-elliptic. 
	\end{proof}
	
	\subsection{Characterization of vertex stabilizers}
	\label{subsec:vertex}
	
	\begin{lemma}\label{lemma:uniqueness-vertex}
		Let $\calg$ be a measured groupoid over a standard finite measure space $Y$, and let $\rho:\calg\to G$ be a strict action-like cocycle. Let $v,v'\in X^{(0)}$ be distinct vertices.
		
		Then for every positive measure Borel subset $U\subseteq Y$, no finite index subgroupoid of $\rho^{-1}(\Stab_G(v))_{|U}$ is contained in  $\rho^{-1}(\Stab_G(v'))_{|U}$.
	\end{lemma}
	
	\begin{proof}
		Assumption~2 from Theorem~\ref{theo:full} ensures that $\Stab_G(v)\cap\Stab_G(v')$ has infinite index in $\Stab_G(v)$. The lemma therefore follows from the second point in the definition of an action-like cocycle. 
	\end{proof}
	
	Let $\calg$ be a measured groupoid over a standard finite measure space $Y$, and let $\rho:\calg\to G$ be a strict cocycle. We say that a measured subgroupoid $\calh$ of $\calg$ is \emph{of vertex stabilizer type with respect to $\rho$} if there exist a conull Borel subset $Y^*\subseteq Y$ and a partition $Y^*=\dunion_{i\in I}Y_i$ into at most countably many Borel subsets such that for every $i\in I$, there exists a vertex $v_i\in X^{(0)}$ such that $\calh_{|Y_i}$ is equal to the $(\calg_{|Y_i},\rho)$-stabilizer of $v_i$, i.e.\  $\calh_{|Y_i}=(\rho^{-1}(\Stab_G(v_i)))_{|Y_i}$. 
	
	A subgroupoid $\calh\subseteq\calg$ is \emph{stably maximal QNA} if it is QNA, and any QNA subgroupoid $\hat\calh$ of $\calg$ which stably contains $\calh$ is stably equal to $\calh$. It is \emph{stably maximal amenable} if it is amenable, and any amenable subgroupoid $\hat\calh$ which stably contains $\calh$ is stably equal to $\calh$. Notice that since amenable groupoids of infinite type are QNA, every amenable subgroupoid of $\calg$ which is stably maximal QNA, is also stably maximal amenable.
	
	In the following statement, we give a purely groupoid-theoretic characterization of subgroupoids of $\calg$ of vertex stabilizer type with respect to $\rho$, without any reference to the action-like cocycle $\rho$. The reader is referred to Section~\ref{subsec:groupoid-argument} in the introduction for the intuition behind the statement, in particular the comparison with the group-theoretic situation. Roughly, in the statement below, Assertion~2.(a) deals with non-amenable Baumslag-Solitar subgroups, and Assertion~2.(b) deals with the amenable ones, and states their ``non-isolation''.
	
	\begin{lemma}\label{lemma:characterization-vertex}
		Let $\calg$ be a measured groupoid over a standard finite measure space $Y$, and let $\rho:\calg\to G$ be a strict  action-like cocycle. Let $\calh$ be a measured subgroupoid of $\calg$. The following assertions are equivalent.
		\begin{enumerate}
			\item The subgroupoid $\calh$ is of vertex stabilizer type with respect to $\rho$. 
			\item The subgroupoid $\calh$ is a stably maximal QNA subgroupoid of $\calg$, and there exists a Borel partition $Y=Y_1\dunion Y_2$ such that
			\begin{enumerate}
				\item there exists an everywhere nonamenable QNA subgroupoid $\calh'$ of $\calg_{|Y_1}$ such that $\calh_{|Y_1}\cap\calh'$ is nowhere trivial, and 
				\item the groupoid $\calh_{|Y_2}$ is amenable, and there exist two stably maximal amenable measured subgroupoids $\calh_1,\calh_2$ of $\calg_{|Y_2}$ such that letting $\calh_0=\calh_{|Y_2}$, the following two assertions hold: 
				\begin{enumerate}
					\item given any two distinct $i,j\in\{0,1,2\}$ and any Borel subset $U\subseteq Y$ of positive measure, no finite index subgroupoid of $(\calh_i)_{|U}$  is contained in $(\calh_j)_{|U}$;
					\item the groupoid $\calh_0\cap\calh_1\cap\calh_2$ is nowhere trivial.
				\end{enumerate}
			\end{enumerate} 
		\end{enumerate}
	\end{lemma}
	
	\begin{proof}
		We first show that $(1)\Rightarrow (2)$, so let us assume that $(\calh,\rho)$ is of vertex stabilizer type with respect to $\rho$. Up to a countable partition of the base space and passing to a conull Borel subset, we can (and will) assume that there exists $v\in X^{(0)}$ such that $\calh=\rho^{-1}(\Stab_G(v))$.
		
		We first prove that $\calh$ is a stably maximal QNA subgroupoid of $\calg$. By assumption (Assumption~1 from Theorem~\ref{theo:full}), the group $\Stab_G(v)$ contains a quasi-normalized infinite amenable subgroup $A$. Letting $\cala=\rho^{-1}(A)$, the groupoid $\cala$ is amenable (because $\rho$ has trivial kernel), of infinite type (using that $\rho$ is action-like) and quasi-normalized by $\calh$. It remains to prove the stable maximality of $\calh$ as a QNA subgroupoid of $\calg$. So let $\hat\calh$ be a QNA subgroupoid of $\calg$ such that $\calh$ is stably contained in $\hat\calh$. Up to passing to a conull Borel subset of $Y$ and taking a countable partition, we will assume that $\calh\subseteq\hat\calh$. 
		
		We aim to prove that $(\hat\calh,\rho)$ is stably $X$-elliptic: this will be enough to deduce the stable maximality of $\calh$ in view of Lemma~\ref{lemma:uniqueness-vertex}. Let $Y=Y_1\dunion Y_2$ be a Borel partition such that $\hat\calh_{|Y_1}$ is everywhere nonamenable, and $\hat\calh_{|Y_2}$ is amenable (this exists by choosing for $Y_2$ a Borel subset of maximal measure such that $\hat\calh_{|Y_2}$ is amenable). Lemma~\ref{lemma:nonamenable-BS-elliptic} ensures that $(\hat\calh_{|Y_1},\rho)$ is stably $X$-elliptic. If $(\hat\calh_{|Y_2},\rho)$ is not stably $X$-elliptic, then Lemma~\ref{lemma:elliptic-loxodromic} ensures that there exists a Borel subset $U\subseteq Y_2$ of positive measure such that there exists an $(\hat\calh_{|U},\rho)$-equivariant Borel map $\varphi:U\to\calp_{\le 2}(\partial_\infty X)$. But $(\calh_{|U},\rho)$ is then both $X$-elliptic and $\partial_\infty X$-elementary (since $\calh_{|U}\subseteq\hat\calh_{|U}$). Lemma~\ref{lemma:elliptic vs elementary} thus ensures that there exists a positive measure Borel subset $V\subseteq U$ such that $\calh_{|V}$ is trivial, which contradicts the assumption that $\rho$ is an action-like cocycle. This proves the stable maximality of $\calh$.
		
		 If $\Stab_G(v)$ is nonamenable, then since $\rho$ is action-like, 
		 the groupoid $\calh$ is everywhere nonamenable. So $Y=Y_1$ and  $\calh'=\calh$ show that $(2)$ holds. So assume that $\Stab_G(v)$ is amenable (in particular $\calh$ is amenable). If an infinite subgroup of $\Stab_G(v)$ fixes another vertex $v'$ with nonamenable stabilizer, then taking $Y_1=Y$ and  $\calh'=\rho^{-1}(\Stab_G(v'))$ shows that $(2)$ holds. Otherwise, by assumption (Assumption~4 from Theorem~\ref{theo:full}), we can find two distinct vertices $v_1,v_2$ of $X$ with amenable stabilizers, both different from $v$, such that the common stabilizer of $v,v_1,v_2$ is infinite.
		Let $v_0=v$, and for $j\in\{0,1,2\}$, let $\calh_j=\rho^{-1}(\Stab_G(v_j))$. By the argument handling  $\hat\calh_{|Y_2}$ in the previous paragraph, we know that the subgroupoids $\calh_j$ are stably maximal amenable subgroupoids of $\calg$. As $v,v_1,v_2$ are mutually distinct, Lemma~\ref{lemma:uniqueness-vertex} ensures that Condition~$(2b.i)$ holds. As the common stabilizer of $v,v_1,v_2$ is infinite and $\rho$ is action-like, we know that $(\calh_0\cap\calh_1\cap\calh_2)_{|U}$ is of infinite type for any positive measure Borel subset $U$, thus Condition~$(2b.ii)$ holds true.
		
		We now prove that $(2)\Rightarrow (1)$. We will prove that $(\calh,\rho)$ is stably $X$-elliptic; the stable maximality condition, together with $(1)\Rightarrow (2)$, will then show that $\calh$ is of vertex stabilizer type with respect to $\rho$. 
		
		By Lemma~\ref{lemma:nonamenable-BS-elliptic},  $(\calh',\rho)$ is stably $X$-elliptic, so $\calh_{|Y_1}\cap\calh'$ is  stably $X$-elliptic. Let $Y_1=Y'_1\dunion Y''_1$ be a Borel partition such that $\calh_{|Y'_1}$ is everywhere nonamenable, and $\calh_{|Y''_1}$ is amenable. By Lemma~\ref{lemma:nonamenable-BS-elliptic}, $(\calh_{|Y'_1},\rho)$ is stably $X$-elliptic. And $(\calh_{|Y''_1},\rho)$ is also stably $X$-elliptic: otherwise, by Lemma~\ref{lemma:elliptic-loxodromic}, there would exist a positive measure Borel subset $U\subseteq Y''_1$ such that $(\calh_{|U},\rho)$ is $\partial_\infty X$-elementary; but then  $(\calh\cap\calh')_{|U}$ is nowhere trivial and both stably $X$-elliptic and $\partial_\infty X$-elementary, contradicting Lemma~\ref{lemma:elliptic vs elementary}. 
		
		Notice that $\calh_{|Y_2}$ is a stably maximal amenable subgroupoid of $\calg_{|Y_2}$. We claim that $(\calh_{|Y_2},\rho)$ is stably $X$-elliptic. Indeed, otherwise, by Lemma~\ref{lemma:elliptic-loxodromic}, there exists a positive measure Borel subset $U\subseteq Y_2$ such that  $(\calh_{|U},\rho)$ is $\partial_\infty X$-elementary. Let now $\calh_0=\calh_{|Y_2}$, and  $\cala=\calh_0\cap\calh_1\cap\calh_2$. Lemma~\ref{lemma:intersection of four gropuoids} implies that $(\cala_{|U},\rho)$ is stably $X$-elliptic. But it is also nowhere trivial (Condition~$(2b.ii)$) and $\partial_\infty X$-elementary, contradicting Lemma~\ref{lemma:elliptic vs elementary}. Therefore $(\calh_{|Y_2},\rho)$ is also stably $X$-elliptic, which completes our proof. 
	\end{proof}

	Keeping the above notations, saying that $\calh$ is a measured subgroupoid of vertex stabilizer type with respect to $\rho$ means that there exists a countable Borel partition $Y^*=\dunion_{i\in I}Y_i$ of a conull Borel subset $Y^*\subseteq Y$ such that for every $i\in I$, one has $\calh_{|Y_i}=\rho^{-1}(\Stab_G(v_i))_{|Y_i}$ for some vertex $v_i\in X^{(0)}$. Of course, this Borel partition is not unique: one can always pass to a finer partition. But we observe that the map $Y\to X^{(0)}$ sending a point $y\in Y$ to the vertex $v_i$ is entirely determined by the pair $(\calh,\rho)$, up to changing its value on a null set. Indeed, otherwise, there would exist two distinct vertices $v,v'\in X^{(0)}$ and a positive measure Borel subset $U\subseteq Y$ such that $\rho^{-1}(\Stab_G(v))_{|U}=\rho^{-1}(\Stab_G(v'))_{|U}$, which would contradict  Lemma~\ref{lemma:uniqueness-vertex}. 
	The above map $Y\to X^{(0)}$ is called the \emph{vertex map} of $(\calh,\rho)$. We insist that, although being of vertex stabilizer type is a notion that does not depend on the choice of an action-like cocycle $\rho:\calg\to G$ (as a consequence of Lemma~\ref{lemma:characterization-vertex}), the vertex map of $(\calh,\rho)$ depends on $\rho$ in an essential way.
	
	\subsection{Conclusion}
	
The main theorem of the section (Theorem~\ref{theo:full}) will be a consequence of our next lemma. In its statement, contrary to the previous sections, we do not need to impose that the $G$-action on $X$ satisfies all the assumptions from Theorem~\ref{theo:full}. Notice that the notion of a measured subgroupoid of $\calg$ of vertex stabilizer type with respect to a cocycle $\rho$ still makes sense in this broader context. 
	
	\begin{lemma}
		\label{lemma:endgame}
		Let $G$ be a torsion-free countable group and let $X$ be an $M_\kappa$-polyhedral complex with countably many cells where $G$ acts by cellular automorphisms. Let $\Theta$ be the associated intersection graph. Assume that
		\begin{enumerate}
			\item for any two distinct vertices $v,v'\in X^{(0)}$, $\Stab_G(v)\cap \Stab_G(v')$ is infinite index in both $\Stab_G(v)$ and $\Stab_G(v')$;
			\item for any measured groupoid $\calg$ over a standard finite measure space $Y$ equipped with a pair of strict action-like cocycles $\rho_1:\calg\to G$ and $\rho_2:\calg\to G$, a measured subgroupoid of $\calg$ is of vertex stabilizer type with respect to $\rho_1$ if and only if it is of vertex stabilizer type with respect to $\rho_2$.
		\end{enumerate}
		Then for every self measure equivalence coupling $\Sigma$ of $G$, there exists a $(G\times G)$-equivariant Borel map $\Sigma\to\Aut(\Theta)$, where the $(G\times G)$-action on $\Aut(\Theta)$ is via $(g_1,g_2)\cdot f(x)=g_1f(g_2^{-1}x)$.
	\end{lemma}
	
		Theorem~\ref{theo:full} is an immediate consequence of Lemma~\ref{lemma:endgame}, as Assumption~1 of Lemma~\ref{lemma:endgame} follows from Assumption~2 of Theorem~\ref{theo:full}, and Assumption 2 of Lemma~\ref{lemma:endgame} follows from Lemma~\ref{lemma:characterization-vertex}. So there only remains to prove Lemma~\ref{lemma:endgame}. Our proof is inspired from Kida's proof of measure equivalence rigidity of surface mapping class groups \cite{Kid-me}, and has become a common argument in proofs of measure equivalence rigidity. 
	
	\begin{proof}[Proof of Lemma~\ref{lemma:endgame}]
		Using \cite[Proposition~5.11(2)]{HH1} (see Remark~\ref{rk:coupling-cocycle}), it is enough to consider a measured groupoid $\calg$ over a standard finite measure space $Y$, and two action-like cocycles $\rho_1,\rho_2:\calg\to G$, and prove that they are $\Aut(\Theta)$-cohomologous, i.e.\ there exists a Borel map $\varphi:Y\to\Aut(\Theta)$ and a conull Borel subset $Y^*\subseteq Y$ such that for every $g\in\calg_{|Y^*}$, one has $\rho_2(g)=\varphi(r(g))\rho_1(g)\varphi(s(g))^{-1}$.
		
		Note that Assumption 1 ensures that for every action-like cocycle $\rho:\calg\to G$, and every measured subgroupoid $\calh$ of $\calg$ of vertex stabilizer type with respect to $\rho$, the pair $(\calh,\rho)$ has a well-defined vertex map as in the discussion at the end of Section~\ref{subsec:vertex}.
		
		\begin{claim}\label{claim}
			Let $\rho:\calg\to G$ be an action-like cocycle. Let $\calh,\calh'$ be two measured subgroupoids of $\calg$ of vertex stabilizer type with respect to $\rho$, and let $v_\calh,v_{\calh'}:Y\to X^{(0)}$ be the vertex maps of $(\calh,\rho)$ and $(\calh',\rho)$, respectively. Then the following conditions are equivalent.
			\begin{enumerate}
				\item For a.e.\ $y\in Y$, the vertices $v_\calh(y)$ and $v_{\calh'}(y)$ are either equal or adjacent in $\Theta$.
				\item The groupoid $\calh\cap\calh'$ is nowhere trivial.
			\end{enumerate}
		\end{claim}	
		
		\begin{proof}[Proof of the claim]
			Up to passing to a conull Borel subset of $Y$ and taking a countable Borel partition, we will assume that the maps $v_\calh$ and $v_{\calh'}$ are constant, with respective values $v,v'\in X^{(0)}$, and that $\calh=\rho^{-1}(\Stab_G(v))$ and $\calh'=\rho^{-1}(\Stab_G(v'))$. 
			
			If $v$ and $v'$ are either equal or adjacent in $\Theta$, then $\Stab_G(v)\cap\Stab_G(v')$ is nontrivial (by definition of $\Theta$), whence infinite because $G$ is torsion-free. Using that $\rho$ is action-like, it follows that $\calh\cap\calh'$ is nowhere trivial. 
			
			Conversely, if $v$ and $v'$ are neither equal nor adjacent in $\Theta$, then $\Stab_G(v)\cap\Stab_G(v')$ is trivial. As $\rho$ has trivial kernel, it follows that $\calh\cap\calh'$ is the trivial groupoid in this case.
		\end{proof}
		
		For every vertex $v$ of $\Theta$ (which can also be viewed as a vertex of $X$), we let $\calh_{1,v}=\rho_1^{-1}(\Stab_G(v))$. Then $\calh_{1,v}$ is a measured subgroupoid of $\calg$ of vertex stabilizer type with respect to $\rho_1$, so Assumption 2 of the lemma ensures that it is also of vertex stabilizer type with respect to $\rho_2$. Let $\phi_{2,v}:Y\to \Theta^{(0)}$ be the vertex map of $(\calh_{1,v},\rho_2)$. We now define a Borel map $\theta:Y\times \Theta^{(0)}\to \Theta^{(0)}$ by letting $\theta(y,v)=\phi_{2,v}(y)$ for every $y\in Y$ and every $v\in \Theta^{(0)}$. 
		
		We first observe that for a.e.\ $y\in Y$, the map $v\mapsto\theta(y,v)$ is injective. Otherwise, as $\Theta^{(0)}$ is countable, there exist a positive measure Borel subset $U\subseteq Y$ and two distinct vertices $v,v'\in \Theta^{(0)}$ such that for a.e.\ $y\in U$, one has $\theta(y,v)=\theta(y,v')$. Up to restricting to a positive measure Borel subset of $U$, we can assume that the maps $\theta(\cdot,v)$ and $\theta(\cdot,v')$ are constant, with value a common vertex $w$. Then $\rho_1^{-1}(\Stab_G(v))_{|U}$ and $\rho_1^{-1}(\Stab_G(v'))_{|U}$ are both stably equal to $\rho_2^{-1}(\Stab_G(w))_{|U}$, contradicting Assumption~1 (as in Lemma~\ref{lemma:uniqueness-vertex}).
		
		Second, for a.e.\ $y\in Y$, the map $v\mapsto\theta(y,v)$ is surjective. Indeed, let $w\in\Theta^{(0)}$, and let  $\calh_{2,w}=\rho_2^{-1}(\Stab_G(w))$. Then $\calh_{2,w}$ is a measured subgroupoid of $\calg$ of vertex stabilizer type with respect to $\rho_2$, so Assumption 2 ensures that it is also of vertex stabilizer type with respect to $\rho_1$. In other words, there exist a partition $Y^*=\dunion_{i\in I} Y_i$ of a conull Borel subset $Y^*\subseteq Y$ into at most countably many Borel subsets and for every $i\in I$, a vertex $v_i\in\Theta^{(0)}$, such that $(\calh_{2,w})_{|Y_i}=\rho_1^{-1}(\Stab_G(v_i))_{|Y_i}$. We then have $\theta(y,v_i)=w$ for almost every $y\in Y_i$. It follows that there exists a conull Borel subset $Y^{**}\subseteq Y$ such that for every $y\in Y^{**}$, the vertex $w$ is in the image of $\theta(y,\cdot)$. As $\Theta^{(0)}$ is countable, we deduce that $\theta(y,\cdot)$ is surjective for a.e.\ $y\in Y$, as desired. 
		
		Third, Claim~\ref{claim} above ensures that for any two distinct vertices $v,v'\in\Theta^{(0)}$, and a.e.\ $y\in Y$, the vertices $\theta(y,v)$ and $\theta(y,v')$ are adjacent if and only if $v$ and $v'$ are. As $\Theta^{(0)}$ is countable, it follows that for a.e.\ $y\in Y$, the map $\theta(y,\cdot)$ preserves both adjacency and non-adjacency. 
		
		In summary, we have proved that there exists a conull Borel subset $Y^*\subseteq Y$ such that for every $y\in Y^*$, the map $\theta(y,\cdot):\Theta^{(0)}\to\Theta^{(0)}$ induces an automorphism of $\Theta$. This yields a Borel map $\varphi:Y^*\to\Aut(\Theta)$, defined by letting $\varphi(y)=\theta(y,\cdot)$ for every $y\in Y^*$. 
		
		We are left with proving that the cocycles $\rho_1$ and $\rho_2$ are cohomologous via $\varphi$. Let $B$ be a bisection of $\calg$ which induces an isomorphism $\alpha:U\to V$ between two Borel subsets of $Y$, and such that $\rho_1$ and $\rho_2$ are both constant in restriction to $B$, with respective values $g_1,g_2\in G$. We will prove that for every $v\in \Theta^{(0)}$ and almost every $y\in U$, one has $\varphi(\alpha(y))(g_1v)=g_2(\varphi(y)(v))$. This will be enough to conclude as $\calg$ can be written as a countable union of such bisections.
		
		Up to a countable Borel partition of $U$, we can assume that $\varphi(y)(v)$ is constant on $U$, with value $w$. This means that, up to replacing $U$ by a conull Borel subset, we have  $\rho_2^{-1}(\Stab_G(w))_{|U}=\rho_1^{-1}(\Stab_G(v))_{|U}$. Transporting this through the bisection $B$, we deduce that $\rho_2^{-1}(\Stab_G(g_2w))_{|V}=\rho_1^{-1}(\Stab_G(g_1v))_{|V}$. This implies that $\varphi(\alpha(y))(g_1v)=g_2w$ for almost every $y\in U$, as desired.
	\end{proof}

\subsection{The case of amenable vertex groups}\label{sec:other-assumptions}

In this section, we explain how to adapt our proof of Theorem~\ref{theo:full} in order to prove the version given in the introduction (Theorem~\ref{theo:combination}), where all vertex stabilizers are assumed to be amenable, but no chain condition is assumed. This section will not be used in the sequel of the paper.

Notice that throughout Section~\ref{sec:me}, the chain condition (Assumption~5 from Theorem~\ref{theo:full}) was only used in the proof of Lemma~\ref{lemma:nonamenable-BS-elliptic}: more precisely, it occurs in the proof of Lemma~\ref{lemma:canonical-vertex-set}, which is one of the key steps towards Lemma~\ref{lemma:nonamenable-BS-elliptic}. Lemma~\ref{lemma:characterization-vertex} has to be replaced by the following statement.

\begin{lemma}\label{lemma:alternative-version}
Let $G$ and $X$ be as in Theorem~\ref{theo:combination}. Let $\calg$ be a measured groupoid over a standard finite measure space $Y$, and let $\rho:\calg\to G$ be a strict action-like cocycle. Let $\calh$ be a measured subgroupoid of $\calg$. The following assertions are equivalent.
\begin{enumerate}
\item The subgroupoid $\calh$ is of vertex stabilizer type with respect to $\rho$.
\item The subgroupoid $\calh$ is a stably maximal amenable subgroupoid of $\calg$, and there exist two stably maximal amenable subgroupoids $\calh_1,\calh_2$ of $\calg$ such that letting $\calh_0=\calh$, the following two assertions hold:
\begin{enumerate}
\item given any two distinct $i,j\in\{0,1,2\}$ and any Borel subset $U\subseteq Y$ of positive measure, no finite index subgroupoid of $(\calh_i)_{|U}$ is contained in $(\calh_j)_{|U}$;
\item the groupoid $\calh_0\cap\calh_1\cap\calh_2$ is nowhere trivial.
\end{enumerate}
\end{enumerate}
\end{lemma}

The proof of Lemma~\ref{lemma:alternative-version} is in fact a simplified version of the proof of Lemma~\ref{lemma:characterization-vertex}. We provide it for completeness, and also to properly highlight how the various assumptions of Theorem~\ref{theo:combination} are used.

\begin{proof}
We first prove that $(1)\Rightarrow (2)$. Up to a countable partition of $Y$ and passing to a conull Borel subset, we will assume that $\calh=\rho^{-1}(\Stab_G(v))$ for some $v\in X^{(0)}$.

By assumption (Assumption~1 from Theorem~\ref{theo:combination}), $\Stab_G(v)$ is amenable. As $\rho$ has trivial kernel, it follows that $\calh$ is amenable. We now prove that $\calh$ is a stably maximal amenable subgroupoid of $\calg$. So let $\hat\calh$ be an amenable subgroupoid of $\calg$ which stably contains $\calh$. Up to passing to a conull subset of $Y$ and taking a countable Borel partition, we will assume that $\calh\subseteq\hat\calh$.  

In view of Lemma~\ref{lemma:uniqueness-vertex} (whose proof uses Assumption~2 from Theorem~\ref{theo:combination}), it is enough to prove that $(\hat\calh,\rho)$ is stably $X$-elliptic. Otherwise, Lemma~\ref{lemma:elliptic-loxodromic} ensures that there exist a positive measure Borel subset $U\subseteq Y$ such that $(\hat\calh_{|U},\rho)$ is $\partial_\infty X$-elementary. Then $(\calh_{|U},\rho)$ is both $X$-elliptic and $\partial_\infty X$-elementary. Lemma~\ref{lemma:elliptic vs elementary} (which exploited the weak acylindricity assumption, which is Assumption~3 from Theorem~\ref{theo:combination}) ensures that there exists a positive measure Borel subset $V\subseteq U$ such that $\calh_{|V}$ is trivial. This contradicts the fact that $\rho$ is an action-like cocycle, and proves the stable maximality of $\calh$.

By Assumption~4 from Theorem~\ref{theo:combination}, we can find two distinct vertices $v_1,v_2\in X^{(0)}\setminus\{v\}$, such that the common $G$-stabilizer of $v,v_1,v_2$ is infinite. Let $v_0=v$, and for $j\in\{0,1,2\}$, let $\calh_j=\rho^{-1}(\Stab_G(v_j))$. By the above paragraph, all subgroupoids $\calh_j$ are stably maximal amenable. And Lemma~\ref{lemma:uniqueness-vertex} ensures that for $i\neq j$, and for every positive measure Borel subset $U\subseteq Y$, no finite index subgroupoid of $(\calh_i)_{|U}$ is contained in $(\calh_j)_{|U}$. Using that $\rho$ is action-like, we also see that $\calh_0\cap\calh_1\cap\calh_2$ is nowhere trivial.

We now prove that $(2)\Rightarrow (1)$. We will prove that $(\calh,\rho)$ is stably $X$-elliptic, which will be enough to conclude in view of the stable maximality condition and $(1)\Rightarrow (2)$.

Assume towards a contradiction that $(\calh,\rho)$ is not stably $X$-elliptic. Then Lemma~\ref{lemma:elliptic-loxodromic} ensures that there exists a positive measure Borel subset $U\subseteq Y$ such that $(\calh_{|U},\rho)$ is $\partial_\infty X$-elementary. Let $\cala=\calh\cap\calh_1\cap\calh_2$, which is nowhere trivial by assumption. Lemma~\ref{lemma:intersection of four gropuoids} ensures that $(\cala,\rho)$ is stably $X$-elliptic. Then $(\cala_{|U},\rho)$ is both stably $X$-elliptic and $\partial_\infty X$-elementary, contradicting Lemma~\ref{lemma:elliptic vs elementary}. This contradiction completes our proof.
\end{proof}

\begin{proof}[Proof of Theorem~\ref{theo:combination}]
Theorem~\ref{theo:combination} is a direct consequence of Lemma~\ref{lemma:endgame}. Indeed, Assumption~1 from Lemma~\ref{lemma:endgame} follows from Assumption~2 of Theorem~\ref{theo:combination}, and Assumption~2 of Lemma~\ref{lemma:endgame} follows from Lemma~\ref{lemma:alternative-version}.
\end{proof}

	\subsection{Some consequences}
	
	In this section, we prove Corollary~\ref{corintro:lattice} from the introduction. For the first part, we also recall that the \emph{abstract commensurator} $\Comm(G)$ of a group $G$ is the group of equivalence classes of isomorphisms between finite-index subgroups of $G$, where two such isomorphisms are equivalent if they coincide in restriction to some common finite-index subgroup of their domain. 
	
	\begin{cor}\label{cor:consequences}
		Let $k\ge 5$, and let $\sigma=((m_1,n_1),\dots,(m_k,n_k))$ be a $k$-tuple of pairs of non-zero integers, where for every $i\in\{1,\dots,k\}$, one has $|m_i|<|n_i|$.
		\begin{enumerate}
			\item The natural homomorphisms $\widehat{\Hig}_\sigma\to\Aut(\Hig_\sigma)$ and  $\widehat{\Hig}_\sigma\to\Comm(\Hig_\sigma)$ are isomorphisms.
			\item For every locally compact second countable group $G$ and every lattice embedding $\alpha:\Hig_\sigma\to G$, there exists a continuous homomorphism $\pi:G\to\widehat{\Hig}_\sigma$ with compact kernel such that $\pi\circ\alpha=\mathrm{id}$.
			\item For every finite generating set $S$ of $\Hig_\sigma$, the inclusion $\Hig_\sigma\hookrightarrow\widehat{\Hig}_\sigma$ extends to an injective homomorphism $\Aut(\Cay(\Hig_\sigma,S))\hookrightarrow \widehat{\Hig}_\sigma$. 
		\end{enumerate}
	\end{cor}
	
	\begin{proof}
		Let $\Theta_\sigma$ be the intersection graph of $\Hig_\sigma$. Recall that $\Hig_\sigma$ has finite index in $\widehat{\Hig}_\sigma$, which is isomorphic to $\Aut(\Theta_\sigma)$ by Theorem~\ref{theo:crigidity0}. Thus, as explained in Remark~\ref{rk:coupling-cocycle}, we have proved that $\widehat{\Hig}_\sigma$ is rigid with respect to action-like cocycles in the sense of \cite[Definition~4.1]{GH} (after changing action-type cocycles to action-like cocycles in \cite[Definition~4.1]{GH} -- all results recorded in \cite[Section~4]{GH} only involve working with natural cocycles associated to group actions and therefore hold if one works with this version of \cite[Definition~4.1]{GH}, as mentioned in Remark~\ref{rk:action-type-action-like}). The injectivity of the natural homomorphism $\widehat{\Hig}_\sigma\to\Comm(\Hig_\sigma)$  (and $\widehat{\Hig}_\sigma\to\Aut(\Hig_\sigma)$) is a consequence of the fact that $\widehat{\Hig}_\sigma$ is ICC (Lemma~\ref{lemma:icc+}). We now prove its surjectivity. Take a commensuration of $\Hig_\sigma$, represented by an isomorphism $\theta:H\to\theta(H)$ between two finite-index subgroups of $\Hig_\sigma$. This gives rise to two monomorphisms $\rho_1:H\to \Hig_\sigma\subseteq \widehat{\Hig}_\sigma$ (given by the inclusion) and $\rho_2:H\to \Hig_\sigma\subseteq \widehat{\Hig}_\sigma$ (given by $\theta$), and both $\rho_1$ and $\rho_2$ have finite index image. We view $H$ as a groupoid over a point, and view $\rho_1$ and $\rho_2$ as two cocycles towards $\widehat{\Hig}_\sigma\approx\Aut(\Theta_\sigma)$. These cocycles are action-like because $\rho_1$ and $\rho_2$ have finite-index image. Now the rigidity with respect to action-like cocycles (as in \cite[Definition~4.1]{GH}) implies that $\rho_1$ and $\rho_2$ are conjugate in $\widehat{\Hig}_\sigma$. Thus $\widehat{\Hig}_\sigma\to\Comm(\Hig_\sigma)$ is surjective, which concludes the proof of the first assertion.
		
		The second assertion follows from \cite[Theorem~4.7]{GH}, and the third from \cite[Corollary~4.8(2)]{GH}, using that $\Hig_\sigma$ is torsion-free (Lemma~\ref{lemma:Higman property}). 
	\end{proof}

	\section{$W^*$-superrigidity of ergodic actions of generalized Higman groups}\label{sec:w}
	
	As a consequence of  Theorem~\ref{theo:full}, we derive the orbit equivalence rigidity of free, ergodic, probability measure-preserving actions of generalized Higman groups, via e.g.\ \cite[Lemma~4.18]{Fur-survey} -- see also \cite[Theorems~1.1 and~1.2(ii)]{Kid-oe} and their proof.
	
	\begin{cor}\label{cor:oe}
		Let $k\ge 5$, and let $\sigma=((m_1,n_1),\dots,(m_k,n_k))$ be a $k$-tuple of pairs of non-zero integers, where for every $i\in\{1,\dots,k\}$, one has $|m_i|<|n_i|$. Let  $\Hig_\sigma\actson X$ be a free, ergodic, measure-preserving action of $\Hig_\sigma$ on a standard Borel probability space $X$. Let $\Gamma\actson Y$ be a free, ergodic, probability measure-preserving action of a countable group $\Gamma$ on a standard probability space $Y$.
		\begin{enumerate}
		\item If the actions $\Hig_\sigma\actson X$ and $\Gamma\actson Y$ are stably orbit equivalent, then they are virtually conjugate.
		\item If the action $\Hig_\sigma\actson X$ is further assumed to be aperiodic (i.e.\ every finite-index subgroup acts ergodically), and if the actions $\Hig_\sigma\actson X$ and $\Gamma\actson Y$ are orbit equivalent, then they are conjugate. \qed
		\end{enumerate}
	\end{cor}
	
	Notice that when $\Hig_\sigma=\Hig_k$ is one of the groups considered by Higman in \cite{Hig}, all ergodic actions of $\Hig_k$ are automatically aperiodic, simply because $\Hig_k$ has no proper finite-index subgroup.
	
	In this section, we will obtain a $W^*$-superrigidity statement, by combining the above corollary with the uniqueness of the Cartan subalgebra $L^\infty(X)$ inside $L(\Hig_\sigma\actson X)$ up to unitary conjugacy. This Cartan-rigidity statement will be derived from a theorem of Ioana \cite[Theorem~1.1]{Ioa}; in order to apply it, we will need the following lemma. 
	
	\begin{lemma}\label{lemma:von-neumann}
		Let $G$ be a generalized Higman group whose standard generating set has at least $5$ elements. Then $G$ admits a splitting as an amalgamated product $G=G_1*_{A} G_2$, with $A$ of infinite index in both $G_1$ and $G_2$, and such that there exists $g\in G$ such that $A\cap gA g^{-1}=\{1\}$. 
	\end{lemma}
	
	\begin{proof}
		 Let $X_\Delta$ be as in Section~\ref{subsec:Higman background}, i.e.\ $X_\Delta$ is the barycentric subdivision of the developed complex of $\Hig_\sigma$.
		A vertex of $X_\Delta$ \emph{has rank $r$} if this vertex is the barycenter of an $r$-dimensional cell of $X$. If an edge of $X_\Delta$ joins a rank 0 vertex and a rank 2 vertex, then there are exactly two triangles of $X_\Delta$ containing this edge. We group all such pairs of triangles into squares. These squares cover $X_\Delta$, and their interiors have pairwise empty intersection. This gives a square complex, which we denote by $X_\square$. We endow $X_\square$ with the piecewise Euclidean metric such that each square is isometric to a unit square in the Euclidean plane. This makes $X_\square$ a $\CAT(0)$ square complex by  Lemma~\ref{lemma:link condition} (see the argument in Section~\ref{sec:metric} for why links of rank $0$ vertices do not contain any cycle of length smaller than $2\pi$). 
		Let $K_\Delta$ be the first barycentric subdivision of $K$, and let $K_\square$ be a cubical subdivision of $K$ obtained from $K_\Delta$ by combining pairs of triangles into squares in a similar fashion as above.
		Then $G\backslash X_\square$ is isomorphic to $K_\square$. Let $q:X_\square\to K_\square$ be the quotient map. 
		
		Let $\bar U_{i,i+2}\subset K_\square$ be the edge path from $b(\bar e_i)$ to $b(K)$, then to $b(\bar e_{i+2})$, where $b(K)$ and $b(\bar e_i)$ denote the barycenters of $K$ and $\bar e_i$ respectively. Then $\bar U_{i,i+2}$ cuts $K_\square$ into two connected components, the closures of which are denoted by $K_1$ and $K_2$. Choices are made so that $K_1$ contains the vertices $\bar x_{i+1}$ and $\bar x_{i+2}$, and $K_2$ contains all the remaining vertices of $K$. For $j\in\{1,2\}$, let $L_j=q^{-1}(K_j)$, and let $U_{i,i+2}=q^{-1}(\bar U_{i,i+2})$. 
		For a rank $2$ vertex $x\in U_{i,i+2}$, there are exactly two edges of $U_{i,i+2}$ containing $x$ forming an angle of $\pi$ at $x$; for a rank $1$ vertex $x\in U_{i,i+2}$, there are countably many edges of $U_{i,i+2}$ containing $x$, and any two of them form an angle of $\pi$ at $x$. Thus by the moreover part of Lemma~\ref{lemma:link condition}, each connected component of $U_{i,i+2}$ is a convex subcomplex of $X_\square$. Hence this component is contractible, i.e.\ it is a tree. 
		We call any such tree a \emph{tree of type $(i,i+2)$}. Similarly, each connected component of $L_1$ or $L_2$ is a convex subcomplex of $X_\square$.
		It follows from the construction that each tree $T$ of type $(i,i+2)$ has a small open neighborhood $O_T$ in $X$ which is an open interval bundle over $T$. This bundle is trivial as $T$ is contractible, and $O_T$ is homeomorphic to $T\times(-\epsilon,\epsilon)$. This together with the simple connectedness of $X_\square$ imply that $T$ separates $X_\square$ into exactly two connected components.
		
		We define a graph $T_i$ as follows. Let $\{T_\lambda\}_{\lambda\in\Lambda}$ be the collection of all trees of type $(i,i+2)$ in $X_\square$. Vertices of $T_i$ are in 1-1 correspondence with connected components of $X_\square\setminus (\cup_{\lambda\in\Lambda}T_\lambda)$. Let $C$ be one such connected component and let $\bar C$ be the closure of $C$ in $X_\square$. Note that $q(\bar C)$ is contained in $K_1$ or $K_2$. Thus $\bar C$ is a connected component of $L_1$ or $L_2$. In particular $\bar C$ is a convex subcomplex of $X_\square$.
		Then the previous paragraph implies that $\bar C\setminus C$ is a disjoint union of trees of type $(i,i+2)$, each of these trees is called a \emph{side} of $C$. Moreover, convexity of $\bar C$ implies that two components of $X_\square\setminus (\cup_{\lambda\in\Lambda}T_\lambda)$ can share at most one side.
		Edges of $T_i$ are in 1-1 correspondence with trees of type $(i,i+2)$ in $X_\square$. 
		An edge connects two vertices of $T_i$ if the tree of type $(i,i+2)$ associated with this edge is a side of both connected components associated with $v_1$ and $v_2$. By the previous discussion, there is at most one edge between two different vertices of $T_i$.
		
		Note that $T_i$ is a tree as the midpoint of each edge of $T_i$ separates $T_i$. Indeed, take an edge $e\in T_i$ and let $T$ be the tree of type $(i,i+2)$ corresponding to $e$. Take two vertices $v_1$ and $v_2$ corresponding to two components $C_1$ and $C_2$ of $X_\square\setminus (\cup_{\lambda\in\Lambda}T_\lambda)$ separated by $T$. Then any edge path in $T_i$ from $v_1$ and $v_2$ must pass through $e$, otherwise we can produce a path in $X_\square$ from a point in $C_1$ to a point in $C_2$ avoiding $T$.
		
		The group $G$ acts on $T_i$ without inversions. The quotient $G\backslash T_i$ has a single edge and two different vertices (note that connected components of $X_\square\setminus (\cup_{\lambda\in\Lambda}T_\lambda)$ are mapped to components of $K_\square\setminus \bar U_{i,i+2}$ under $q:X_\square\to K_\square$, and $G$-orbits of vertices of $T_i$ are in 1-1 correspondence with components of $X_\square\setminus U_{i,i+2}$). This gives a splitting $G=G_1*_{A} G_2$, where stabilizers of components of $L_i$ are conjugates of $G_i$ (for $i\in\{1,2\}$) and stabilizers of trees of type $(i,i+2)$ are conjugates of $A$. Note that if $g\in\langle a_{i+1}\rangle$ is a nontrivial element, then $g$ does not stabilize any tree of type $(i,i+2)$, but $g$ stabilizes a component of $L_1$. Thus $A$ is infinite index in $G_1$. Similarly $A$ is infinite index in $G_2$. It also follows that $T_i$ has infinite diameter.
		We are done if we can find two trees of type $(i,i+2)$, denoted $T_{\lambda}$ and $T_{\lambda'}$, such that the intersection of their stabilizers is trivial. 
		
		The space $X_{\square}$ is quasi-isometric to $X$, and both are geodesic metric spaces, so $X_{\square}$ is Gromov hyperbolic. Let $Y_{\lambda}=\{x\in T_{\lambda}\mid d(x,T_{\lambda'})=d(T_{\lambda},T_{\lambda'})\}$ and $Y_{\lambda'}=\{x\in T_{\lambda'}\mid d(x,T_{\lambda})=d(T_{\lambda},T_{\lambda'})\}$. 
		It follows from \cite[Lemma~2.10]{Hua} (adapting \cite[Lemma~2.3]{BKS}) that $Y_{\lambda}$ and  $Y_{\lambda'}$ are nonempty and they bound convex subset isometric to $Y_{\lambda}\times [0,d(Y_\lambda,Y_{\lambda'})]$; hyperbolicity of $X_\square$ thus ensures that $Y_{\lambda}$ and  $Y_{\lambda'}$ are bounded as long as $d(T_{\lambda},T_{\lambda'})$ is large enough. Note that $G':=\Stab_G(T_{\lambda})\cap \Stab_G(T_{\lambda'})$ stabilizes both $Y_{\lambda}$ and $Y_{\lambda'}$. As $Y_{\lambda}$ and $Y_{\lambda'}$ are bounded, $G'$ fixes the circumcenters $c,c'$ of $Y_{\lambda},Y_{\lambda'}$ (cf.\ \cite[Proposition II.2.7]{BH}). Thus $G'$ fixes the minimal cells of $X_{\square}$ that contain $c$ and $c'$ respectively.
		By Lemma~\ref{lemma:intersection of vertex stabilizer}(2), as long as $d(Y_{\lambda},Y_{\lambda'})=d(T_{\lambda}, T_{\lambda'})$ is large enough, $G'$ is trivial.  Moreover, it is possible to find two trees of type $(i,i+2)$ whose distance in $X_\square$ is arbitrarily large. Indeed, we can take two such trees $T_{\lambda},T_{\lambda'}$ corresponding to two edges $e$ and $e'$ of $T_i$ that are far away (this is possible as $T_i$ is unbounded), and any path from $T_{\lambda}$ to $T_{\lambda'}$ has to pass through trees of type $(i,i+2)$ corresponding to edges of $T_i$ that separate $e$ from $e'$.
	\end{proof}
	
Recall that a \emph{Cartan subalgebra} in a von Neumann algebra $M$ is a maximal abelian subalgebra whose normalizer generates $M$. By combining Lemma~\ref{lemma:von-neumann} and \cite[Theorem~1.1]{Ioa}, we reach the following corollary.
	
	\begin{cor}
		Let $k\ge 5$, and let $\sigma=((m_1,n_1),\dots,(m_k,n_k))$ be a $k$-tuple of pairs of non-zero integers, where for every $i\in\{1,\dots,k\}$, one has $|m_i|<|n_i|$. Let $\Hig_\sigma\actson X$ be a free, ergodic, probability measure-preserving action of the generalized Higman group $\Hig_\sigma$ on a standard Borel probability space $X$.
		
		Then $L(\Hig_\sigma\actson X)$ contains $L^\infty(X)$ as its unique Cartan subalgebra, up to unitary conjugacy. \qed
	\end{cor}
	
	Finally, by combining the orbit equivalence rigidity statement with the uniqueness of the Cartan subalgebra up to unitary conjugacy, we reach a $W^*$-rigidity statement for all actions of a generalized Higman group. Indeed, as was observed by Singer \cite{Sin}, two free ergodic measure-preserving actions of countable groups on standard probability spaces $X,Y$ are orbit equivalent if and only if there exists an isomorphism between their von Neumann algebras sending $L^\infty(X)$ to $L^\infty(Y)$.
	
	\begin{cor}\label{cor:w}
		Let $k\ge 5$, and let $\sigma=((m_1,n_1),\dots,(m_k,n_k))$ be a $k$-tuple of pairs of non-zero integers, where for every $i\in\{1,\dots,k\}$, one has $|m_i|<|n_i|$. Let $\Hig_\sigma\actson X$ be a free, ergodic, measure-preserving action of $\Hig_\sigma$ on a standard Borel probability space $X$. Let $\Gamma\actson Y$ be a free, ergodic, probability measure-preserving action of a countable group $\Gamma$ on a standard probability space $Y$.
		\begin{enumerate}
		\item If the actions $\Hig_\sigma\actson X$ and $\Gamma\actson Y$ are stably $W^*$-equivalent, then they are virtually conjugate. 
		\item If the action $\Hig_\sigma\actson X$ is aperiodic, and if the actions $\Hig_\sigma\actson X$ and $\Gamma\actson Y$ are $W^*$-equivalent, then they are conjugate.
		\end{enumerate}
	\end{cor}
	
	\begin{proof}
	   The fact that the two actions are stably $W^*$-equivalent means that there exists $t>0$ such that $L(\Hig_\sigma\actson X)$ is isomorphic to $L(\Gamma\actson Y)^t$. The von Neumann algebra $L(\Gamma\actson Y)^t$ is also the von Neumann algebra of a measured equivalence relation $\mathcal{R}$ obtained by restricting the (product) measure-preserving action of $\Gamma\times\mathbb{Z}/n\mathbb{Z}$ on $Y\times\{1,\dots,n\}$, for some $n\in\mathbb{N}$, to a positive measure Borel subset $U$. As $L^\infty(X)$ is, up to unitary conjugacy, the unique Cartan subalgebra of $L(\Hig_\sigma\actson X)$, it follows that $L^\infty(U)$ is, up to unitary conjugacy, the unique Cartan subalgebra of $L(\mathcal{R})$. Therefore, there exists an isomorphism between $L(\Hig_\sigma\actson X)$ and $L(\mathcal{R})$ sending $L^\infty(X)$ to $L^\infty(U)$. By an observation of Singer \cite{Sin}, it follows that the actions $\Hig_\sigma\actson X$ and $\Gamma\times\mathbb{Z}/n\mathbb{Z}\actson Y\times\{1,\dots,n\}$ are stably orbit equivalent, and therefore so are the actions $\Hig_\sigma\actson X$ and $\Gamma\actson Y$. The first conclusion therefore follows from Corollary~\ref{cor:oe}. The same reasoning also shows that if the two actions are $W^*$-equivalent, then they are orbit equivalent, and the second conclusion thus follows from the second conclusion of Corollary~\ref{cor:oe}. 
	\end{proof}
	
\appendix

\section{$M_\kappa$-polyhedral complexes}\label{subsec:polyhedral complex}	
	
	The following definitions are taken from \cite[Definitions~I.7.34 and~I.7.37]{BH}, to which the reader is referred for precisions. Given a real number $\kappa$, a metric space $C$ is a \emph{convex $M_{\kappa}$-polyhedral cell} if there exists $n\in\mathbb{N}$ such that $C$ is isometric to the convex hull of a finite set of points in the complete simply connected $n$-dimensional Riemannian manifold $M$ with constant curvature $\kappa$ (when $\kappa>0$, one adds the extra condition that $C$ is contained in an open ball of $M$ of radius $\frac{\pi}{2\sqrt{\kappa}}$). An \emph{$M_{\kappa}$-polyhedral complex} is a pseudo-metric space obtained from the disjoint union of a collection of convex $M_\kappa$-polyhedral cells by gluing some of the cells along isometric faces. An $M_{0}$-polyhedral complex is also called a \emph{piecewise Euclidean} complex, and an $M_{-1}$-polyhedral complex is also called a \emph{piecewise hyperbolic} complex. An $M_{\kappa}$-polyhedral complex $X$ has \emph{finite shapes} if the collection of convex polyhedra we use to build the complex has only finitely many isometry types -- in this case $X$ turns out to be a complete geodesic metric space \cite[Theorem~I.7.50]{BH}. 
	
	From now on, we let $X$ be a 2-dimensional $M_{\kappa}$-polyhedral complex with finite shapes. We will always denote by $X^{(0)}$ the vertex set of $X$, and by $X^{(1)}$ its edge set. Let $x\in X^{(0)}$ be a vertex. We denote the \emph{link} of $x$ in $X$ by $\lk(x,X)$ (cf.\ \cite[I.7.15]{BH}). Topologically $\lk(x,X)$ is homeomorphic to an $\epsilon$-sphere $S_\epsilon$ around $x$ for $\epsilon$ small. 
	Moreover, $\lk(x,X)$ has the structure of a metric graph whose vertices correspond to intersections of $S_\epsilon$ with edges of $X$ emanating from $x$, and two vertices of $\lk(x,X)$ are joined by an edge $e$ of length $\theta$ if they correspond to a corner of a 2-cell $C$ of $X$ with angle $\theta$ (and the edge $e$ is naturally isometric to the unit sphere of the tangent space of $C$ at $x$). 
	
	A possible description of the link $\lk(x,X)$ is as the space of directions at $x$, as follows. Define two geodesic segments emanating from $x$ to be \emph{equivalent} if they share a common initial nondegenerate subsegment. A \emph{direction} at $x$ is an equivalence class of geodesic segments emanating from $x$. Then points in $\lk(x,X)$ are in 1-1 correspondence with directions at $x$.
	The \emph{angle} between two geodesics $\overline{xx_1}$ and $\overline{xx_2}$ emanating from $x$, denoted by $\angle_x(x_1,x_2)$, is defined to be $\inf\{d(z_1,z_2),\pi\}$ where $z_i\in\lk(x,X)$ corresponds to the direction of $\overline{xx_i}$ and the distance on $\lk(x,X)$ is the length metric induced by the metric graph structure discussed above. Note that $\inf\{d(z_1,z_2),\pi\}$ gives another metric on $\lk(x,X)$ (different from the length metric), which we will refer to as the \emph{angular metric}. Note that $\overline{xx_1}$ and $\overline{xx_2}$ fit together to be a local geodesic at $x$ if and only if $\angle_x(x_1,x_2)=\pi$. If $X$ is also CAT$(\kappa)$ for $\kappa\le 0$, then $\overline{xx_1}\cup \overline{xx_2}$ is a geodesic segment, as locally geodesic segments are geodesic segments in CAT$(\kappa)$ spaces with $\kappa\le 0$ \cite[Proposition~II.1.4(2)]{BH}.
	We will often use the following standard fact.
	
	\begin{lemma}
		\label{lemma:link condition}
		Let $\kappa\le 0$ and let $X$ be a 2-dimensional $M_{\kappa}$-polyhedral complex with finite shapes. Then $X$ is $\mathrm{CAT}(\kappa)$ if and only if $X$ is simply connected, and for every vertex $x\in X^{(0)}$, the link $\lk(x,X)$ does not contain any embedded cycle of length smaller than $2\pi$.
		
		Moreover, if $X$ is $\mathrm{CAT}(\kappa)$ and if $Y\subset X$ is a connected subcomplex such that for each vertex $x\in Y$, the natural embedding $\lk(x,Y)\to\lk(x,X)$ is isometric with respect to the angular metrics on $\lk(x,Y)$ and $\lk(x,X)$, then $Y$ is a convex subset of $X$ and $Y$ is also $\CAT(\kappa)$.
	\end{lemma}
	
	The first statement follows from \cite[Theorem~II.5.4, Lemma~II.5.6]{BH}. To see the moreover part, equip $Y$ with its intrinsic length metric. Since $X$ has finite shapes, having isometric embeddings of links implies that the inclusion map $Y\to X$ is locally an isometric embedding. As $X$ is simply connected (as it is $\CAT(\kappa)$ for $\kappa\le 0$), we know that $Y$ is also simply connected by \cite[Proposition~II.4.14(1)]{BH}. Hence the inclusion map $Y\to X$ is an isometric embedding by \cite[Proposition~II.4.14(2)]{BH}. The convexity of $Y$ in $X$ now follows from the fact that any two points in a $\CAT(\kappa)$ space with $\kappa\le 0$ are joined by a unique geodesic \cite[Proposition~II.1.4(1)]{BH}. Now $Y$ is $\CAT(\kappa)$ by the first part of the lemma.

\section{Measured groupoids}\label{sec:appendix-groupoids}
	
	We review some background on measured groupoids. The reader is referred to  \cite[Section~2.1]{AD}, \cite{Kid-survey} or \cite[Section~3]{GH} for more information.
	
	\paragraph*{Borel and measured groupoids, subgroupoids and restrictions.} Let $Y$ be a standard Borel space. A \emph{Borel groupoid} $\mathcal{G}$ over $Y$ is a standard Borel space which comes equipped with a source map and a range map $s,r:\calg\to Y$, both Borel (we think of every element $g\in\calg$ as an arrow with source $s(g)$ and range $r(g)$), with a Borel composition law, a Borel inverse map, and a neutral element $e_y$ for every $y\in Y$, satisfying the axioms of groupoids (see e.g.\ \cite[Definition~2.10]{Kid-survey}). All Borel groupoids considered in the present paper are \emph{discrete}, i.e.\ there are at most countably many elements of $\calg$ with a given source (or range). By a theorem of Lusin and Novikov (see e.g.\ \cite[Theorem~18.10]{Kec}), every discrete Borel groupoid $\calg$ is covered by a countable union of \emph{bisections}, i.e.\ Borel subsets $B\subseteq\calg$ such that $s_{|B},r_{|B}$ are Borel isomorphisms to Borel subsets $s(B),r(B)\subseteq Y$.  When $Y$ is equipped with a $\sigma$-finite Borel measure $\mu$, a \emph{measured groupoid} $\calg$ over $Y$ is a Borel groupoid over $Y$ such that $\mu$ is $\calg$-quasi-invariant (i.e.\ for every bisection $B\subseteq\calg$, one has $\mu(s(B))=0$ if and only if $\mu(r(B))=0$). There is a natural notion of a measured subgroupoid of $\mathcal{G}$, as well as a notion of restriction $\mathcal{G}_{|U}$ of a measured groupoid $\mathcal{G}$ to a Borel subset $U\subseteq Y$, defined by only considering elements $g\in\calg$ such that $s(g),r(g)\in U$. 
	
	We say that $\calg$ is \emph{trivial} if $\calg=\{e_y|y\in Y\}$, and \emph{stably trivial} if there exist a conull Borel subset $Y^*\subseteq Y$ and a partition $Y^*=\dunion_{i\in I} Y_i$ into at most countably many Borel subsets such that for every $i\in I$, the groupoid $\calg_{|Y_i}$ is trivial. On the contrary, it is \emph{nowhere trivial} if there does not exist any positive measure Borel subset $U\subseteq Y$ such that $\calg_{|U}$ is trivial.
	
	An important example of measured groupoid comes from  (quasi-)measure-preserving group actions. Let $G$ be a countable group, acting by Borel automorphisms in a quasi-measure-preserving way on a standard finite measure space $Y$. Then $G\times Y$ has a natural structure of a measured groupoid over $Y$, denoted by $G\ltimes Y$: the source and range maps are given by $s(g,y)=y$ and $r(g,y)=gy$, the composition law by $(g,hy)(h,y)=(gh,y)$, the inverse map by $(g,y)^{-1}=(g^{-1},gy)$, and the neutral elements are $e_y=(e,y)$.
	
	\paragraph*{Stable containment and stable equality of subgroupoids.} Given two measured subgroupoids $\calh,\calh'\subseteq\calg$, we say that $\calh$ is \emph{stably contained} in $\calh'$ if there exist a conull Borel subset $Y^*\subseteq Y$ and a partition  $Y^*=\dunion_{i\in I}Y_i$ into at most countably many Borel subsets such that for every $i\in I$, one has $\calh_{|Y_i}\subseteq\calh'_{|Y_i}$. We say that $\calh$ and $\calh'$ are \emph{stably equal} if each of them is stably contained in the other. In other words $\calh$ and $\calh'$ are stably equal if and only if there exist a conull Borel subset $Y^*\subseteq Y$ and a partition  $Y^*=\dunion_{i\in I}Y_i$ into at most countably many Borel subsets such that for every $i\in I$, one has $\calh_{|Y_i}=\calh'_{|Y_i}$. 
	
	\paragraph*{Finite-index subgroupoids.}
	
	The notion of a finite-index subgroupoid was introduced by Kida in \cite{Kid-BS}, generalizing a notion due to Feldman, Sutherland and Zimmer \cite{FSZ} for equivalence relations. Let $\calh\subseteq\calg$ be a measured subgroupoid. Given $y\in Y$, the \emph{local index} of $\calh$ in $\calg$ at $y$, denoted by $[\calg:\calh]_y$, is defined as the number of classes of the equivalence relation $\sim_\calh$ on $s^{-1}(y)$ given by $g\sim_\calh h$ if and only if $hg^{-1}\in\calh$. The subgroupoid $\calh$ has \emph{finite index} in $\calg$ if for almost every $y\in Y$, one has $[\calg:\calh]_y<+\infty$. 
	
	A measured groupoid $\calg$ is \emph{of infinite type} if there does not exist any positive measure Borel subset $U\subseteq Y$ such that the trivial subgroupoid has finite index in $\calg_{|U}$. In other words $\calg$ is of infinite type if and only if for every positive measure Borel subset $U\subseteq Y$, and for a.e.\ $y\in U$, there are infinitely many elements $g\in\calg_{|U}$ such that $s(g)=y$. Notice that restricting a measured groupoid of infinite type to a positive measure Borel subset again yields a measured groupoid of infinite type.
	
	The intersection of two finite-index subgroupoids of $\calg$ is again finite-index \cite[Lemma~3.6(i)]{Kid-BS}. Also, if $\calh$ has finite index in $\calg$, and if $U\subseteq Y$ is a Borel subset of positive measure, then $\calh_{|U}$ has finite index in $\calg_{|U}$ by \cite[Lemma~3.5(ii)]{Kid-BS}. 
	
	\paragraph*{Cocycles and equivariant maps.} Let $\calg$ be a measured groupoid over a standard measure space $Y$, and let $G$ be a countable group. A Borel map $\rho:\calg\to G$ is a \emph{strict cocycle} if for all $g_1,g_2\in\calg$ with $s(g_1)=r(g_2)$, one has $\rho(g_1g_2)=\rho(g_1)\rho(g_2)$. The strict cocycle $\rho$ \emph{has trivial kernel} if the only elements $g\in\calg$ such that $\rho(g)=e$ are the units $e_y$. 
	
As a crucial example, when $\calg=G\ltimes Y$ is a measured groupoid associated to a quasi-measure-preserving group action on a standard finite measure space $Y$, the map $\rho:\calg\to G$ defined by $\rho(g,y)=g$ is a strict cocycle. We call it the \emph{natural cocycle} of $G\ltimes Y$; it has trivial kernel.
	
 If $G^0\subseteq G$ is a finite-index subgroup, then $\rho^{-1}(G^0)$ is a finite-index subgroupoid of $\calg$ by \cite[Lemma~3.7]{Kid-BS}.
	
	Now assume that $G$ acts on a standard Borel space $\Delta$ by Borel automorphisms. A Borel map $\varphi:Y\to\Delta$ is \emph{$(\calg,\rho)$-equivariant} if there exists a conull Borel subset $Y^*\subseteq Y$ such that for all $g\in\calg_{|Y^*}$, one has $\varphi(r(g))=\rho(g)\varphi(s(g))$. It is \emph{stably $(\calg,\rho)$-equivariant} if there exists a partition $Y=\dunion_{i\in I}Y_i$ into at most countably many Borel subsets such that for every $i\in I$, the map $\varphi_{|Y_i}$ is $(\calg_{|Y_i},\rho)$-equivariant. In general, the \emph{$(\calg,\rho)$-stabilizer} of $\varphi$ is $\{g\in\calg|\varphi(r(g))=\rho(g)\varphi(s(g))\}$, a measured subgroupoid of $\calg$. An element $\delta\in\Delta$ is \emph{$(\calg,\rho)$-invariant} if the constant map with value $\delta$ is $(\calg,\rho)$-equivariant (in other words, if there exists a conull Borel subset $Y^*\subseteq Y$ such that $\rho(\calg_{|Y^*})\subseteq \Stab_G(\delta)$). We also say that $\delta$ is \emph{$(\calg,\rho)$-quasi-invariant} if there exists a finite index subgroupoid $\calg^0\subseteq\calg$ such that $\delta$ is $(\calg^0,\rho)$-invariant.
	
	In the following lemma, we denote by $\calp_{<\infty}(\Delta)$ the set of all nonempty finite subsets of $\Delta$, equipped with its natural Borel structure induced from that of $\Delta$.
	
	\begin{lemma}[{Kida \cite[Lemma~3.8]{Kid-BS}}]\label{lemma:fi}
		Let $\calg$ be a measured subgroupoid over a standard finite measure space $Y$, let $G$ be a countable group, and let $\rho:\calg\to G$ be a strict cocycle. Let $\Delta$ be a standard Borel space equipped with a $G$-action by Borel automorphisms. Let $\calh\subseteq\calg$ be a measured subgroupoid of finite index.
		
		If there exists a stably $(\calh,\rho)$-equivariant Borel map $\varphi:Y\to \Delta$, then there exists a stably $(\calg,\rho)$-equivariant Borel map $\Phi:Y\to \calp_{<\infty}(\Delta)$ such that for every $y\in Y$, one has $\varphi(y)\in\Phi(y)$.
	\end{lemma}
	
	\begin{rk}
		Notice that \cite[Lemma~3.8]{Kid-BS} is stated with the extra assumption that the local index $[\calg:\calh]_y$ is constant. However, this assumption is not required for the proof. Indeed, we can always partition $Y$ into at most countably many Borel subsets $Y_n$, where $Y_n=\{y\in Y|[\calg:\calh]_y=n\}$ -- that these sets are Borel follows from \cite[Lemma~3.5(i)]{Kid-BS}. All the subsets $Y_n$ are $\calg$-invariant, because $[\calg:\calh]_y=[\calg:\calh]_x$ whenever there exists $g\in\calg$ such that $s(g)=x$ and $r(g)=y$, see \cite[Lemma~3.5(i)]{Kid-BS}.
	\end{rk}
	
	\paragraph*{Action-like cocycles.}
	
	 A strict cocycle $\rho:\calg\to G$ is \emph{action-like} if
	\begin{enumerate}
	    \item $\rho$ has trivial kernel, 
	    \item whenever $H_1\subseteq H_2$ is an infinite index inclusion of subgroups of $G$, for every positive measure Borel subset $U\subseteq Y$, the subgroupoid $\rho^{-1}(H_1)_{|U}$ does not have finite index in $\rho^{-1}(H_2)_{|U}$, and
	    \item for every non-amenable subgroup $H\subseteq G$ and every positive measure Borel subset $U\subseteq X$, the subgroupoid $\rho^{-1}(H)_{|U}$ is non-amenable.
	    \end{enumerate}
	    This notion is slightly stronger than the notion of \emph{action-type} cocycle from \cite[Definition~3.20]{GH}, which corresponds to only requiring the first assumption, and the second assumption with $H_1=\{1\}$. Our main examples of action-like cocycles come from the following lemma. We mention that in an earlier version of this article, we were only making the first two assumptions, however adding the third assumption enabled us to remove an unnecessary assumption in the statement of Theorem~\ref{theo:full}. The idea of adding the last assumption comes from work of Escalier and the first-named author \cite{EH}.
	
	\begin{lemma}\label{lemma:action-like}
		Let $G\actson (Y,\mu)$ be a measure-preserving action of a countable group $G$ on a standard finite measure space $(Y,\mu)$.
		
		Then the natural cocycle $\rho:G\ltimes Y\to G$ is action-like.
	\end{lemma}
	
	\begin{proof}
		The fact that $\rho$ has trivial kernel is clear from its definition. 
		
	 We will now only check the second condition from the definition of an action-like cocycle, and refer to \cite[Lemma~4.12]{EH} for the third one. Let $H_1\subseteq H_2$ be an infinite index inclusion of subgroups of $G$, and let $U\subseteq Y$ be a positive measure Borel subset. Let $\{g_n\}_{n\in\mathbb{N}}$ be a set of representatives of the left $H_1$-cosets in $H_2$.  Let $V$ be the Borel subset of $Y$ consisting of all elements that belong to infinitely many translates $g_n U$. Notice that $V$ has positive measure at least $\mu(U)$: indeed, letting $V_n$ be the union of all $g_i U$ with $i\ge n$, then all subsets $V_n$ have measure at least $\mu(g_nU)=\mu(U)$, they form a non-increasing sequence, and $V$ is equal to their intersection, which implies that $\mu(V)\ge\mu(U)$ using that $\mu$ is a finite measure.  Let $V'\subseteq V$ be a Borel subset of positive measure contained in a single translate $g_{i_0} U$, and let $U'=g_{i_0}^{-1}V'$, a positive measure Borel subset of $U$. We will prove that the local index $[\rho^{-1}(H_2)_{|U}:\rho^{-1}(H_1)_{|U}]_y$ is infinite for every $y\in U'$, which will conclude the proof. So let $y\in U'$, and let $(g_{\sigma(n)})_{n\in\mathbb{N}}$ be an infinite subsequence such that for every $n\in\mathbb{N}$, one has $g_{i_0}y\in g_{\sigma(n)}U$. Then all elements of the form $(g_{\sigma(n)}^{-1}g_{i_0},y)$ satisfy $s(g_{\sigma(n)}^{-1}g_{i_0},y)=y$ and $r(g_{\sigma(n)}^{-1}g_{i_0},y)\in U$, and their pairwise differences do not belong to $\rho^{-1}(H_1)$.  
	\end{proof}
	
	\begin{rk}\label{rk:action-type-action-like}
	Working with the slightly more restrictive class of action-like cocycles, instead of action-type cocycles as in \cite{GH}, is harmless in this paper. For instance, all statements given in \cite[Section~4]{GH} are phrased with a cocycle rigidity assumption regarding all possible action-type cocycles from a measured groupoid towards a given group $G$, but their proofs only requires considering action-type cocycles coming from restrictions of group actions. Lemma~\ref{lemma:action-like} thus shows that, in \cite[Definition~4.1]{GH}, if rigidity with respect to action-type cocycles is replaced by rigidity with respect to action-like cocycles, then the statements from \cite[Section~4]{GH} remain valid. 
	\end{rk}

	\paragraph*{Quasi-normalized subgroupoids.}
	
	Let $\calg$ be a measured groupoid, and let $\calh,\calh'\subseteq\calg$ be two measured subgroupoids. Given a bisection $B\subseteq\calg$ with source $U$ and range $V$, we say that $\calh$ is \emph{$B$-quasi-invariant} if $(B\calh_{|U}B^{-1})\cap\calh_{|V}$ has finite index in both $\calh_{|V}$ and $B\calh_{|U}B^{-1}$. Following Kida \cite[Definition~3.15]{Kid-BS}, we say that $\calh$ is \emph{quasi-normalized} by $\calh'$ if $\calh'$ can be covered by countably many bisections $B_n$, in such a way that $\calh$ is $B_n$-quasi-invariant for every $n$. We say that $\calh$ is \emph{stably quasi-normalized} by $\calh'$ if there exists a countable Borel partition $Y=\dunion_{i\in I}Y_i$ such that for every $i\in I$, the groupoid $\calh_{|Y_i}$ is quasi-normalized by $\calh'_{|Y_i}$.
	
	Given a countable group $G$, a subgroup $H\subseteq G$ is \emph{quasi-normalized} (or \emph{commensurated}) by $G$ if for every $g\in G$, the intersection $gHg^{-1}\cap H$ has finite index in both $H$ and $gHg^{-1}$ -- notice that this coincides with the above definition, by viewing $G$ as a measured groupoid over a point. Now, if $\rho:\calg\to G$ is a strict cocycle, and if $H\subseteq G$ is quasi-normalized by $G$, then $\rho^{-1}(H)$ is quasi-normalized by $\calg$.

	\paragraph*{Amenable subgroupoids.}
	We refer to \cite{Kid-survey} for the definition of amenability of a measured groupoid, and here we only mention the properties that are used in the paper. 
	
	Assume that we have a strict cocycle $\rho:\calg\to G$, where $G$ is a countable group, and that $G$ acts by homeomorphisms on a compact metrizable space $K$. If $\calg$ is amenable, then by \cite[Proposition~4.14]{Kid-survey}, there exists a $(\calg,\rho)$-equivariant Borel map $Y\to\Prob(K)$. 
	
	Every measured subgroupoid of an amenable measured groupoid is amenable \cite[Theorem~4.16]{Kid-survey}, and likewise every restriction of an amenable groupoid is amenable. Also, if $Y^*\subseteq Y$ is a conull Borel subset of the base space and $Y^*=\dunion_{i\in I} Y_i$ is a partition into at most countably many Borel subsets, and if for every $i\in I$, the groupoid $\calg_{|Y_i}$ is amenable, then $\calg$ is amenable (see e.g.\ the discussion in \cite[Section~3.3.1]{GH}). 
	
	An important source of amenable subgroupoids is the following. Assume that we have a strict cocycle $\rho:\calg\to G$ with trivial kernel. If $A\subseteq G$ is an amenable subgroup, then $\rho^{-1}(A)$ is an amenable subgroupoid of $\calg$, see e.g.\ \cite[Corollary~3.39]{GH}. More generally, assume that $G$ acts by Borel automorphisms on a Polish space $\Delta$, and that the action is \emph{Borel amenable} in the sense that there exists a sequence of Borel maps $\mu_n:\Delta\to\Prob(G)$ such that for every $\delta\in\Delta$ and every $g\in G$, one has $||\mu_n(g\cdot\delta)-g\cdot\mu_n(\delta)||_1\to 0$ (as $n$ goes to $+\infty$). If there exists a $(\calg,\rho)$-equivariant Borel map $Y\to\Delta$, then $\calg$ is amenable (\cite[Proposition~3.38]{GH}, adapting \cite[Proposition~4.33]{Kid-memoir}).
	
	We say that $\calg$ is \emph{everywhere nonamenable} if for every positive measure Borel subset $U\subseteq Y$, the groupoid $\calg_{|U}$ is nonamenable. If $\rho:\calg\to G$ is action-like and $G$ contains a nonabelian free subgroup, then $\rho^{-1}(G)$ is everywhere nonamenable, as follows from \cite[Lemma~3.20]{Kid-me} (or, more precisely, from its proof, see e.g.\ \cite[Remark~3.3]{HH2}).

	\footnotesize
	
	\bibliographystyle{alpha}
	\bibliography{higman-bib}

	\begin{flushleft}
		Camille Horbez\\
		Universit\'e Paris-Saclay, CNRS,  Laboratoire de math\'ematiques d'Orsay, 91405, Orsay, France \\
		\emph{e-mail:~}\texttt{camille.horbez@universite-paris-saclay.fr}\\[4mm]
	\end{flushleft}
	
	\begin{flushleft}
		Jingyin Huang\\
		Department of Mathematics\\
		The Ohio State University, 100 Math Tower\\
		231 W 18th Ave, Columbus, OH 43210, U.S.\\
		\emph{e-mail:~}\texttt{huang.929@osu.edu}\\
	\end{flushleft}

\end{document}